\documentclass[11pt,onecolumn,oneside]{amsart}
\usepackage[a4paper, hmargin={2.8cm, 2.8cm}, vmargin={2.5cm, 2.5cm}]{geometry}

\usepackage{graphicx, color}
\usepackage{enumitem}
\usepackage{amssymb}
\usepackage{amsmath}
\usepackage{amsthm}
\usepackage{mathrsfs}
\usepackage{caption, subcaption}
\usepackage{wrapfig}
\usepackage{mathtools}
\usepackage{bm}
\usepackage{advdate}
\usepackage{url}
\usepackage{accents}
\usepackage{hyperref}
\usepackage{cleveref}

\usepackage{tikz}
\usetikzlibrary{positioning,arrows,chains,matrix,scopes,fit,decorations.markings,decorations,decorations.pathreplacing}
\usetikzlibrary{cd}
\usepackage{ytableau}

	\newcommand{\N}{\ensuremath{\mathbb{N}}}
	\newcommand{\A}{\ensuremath{\mathbb{A}}}
	
	\newcommand{\Z}{\ensuremath{\mathbb{Z}}}

	\newcommand{\C}{\ensuremath{\mathbb{C}}}
    \newcommand{\bz}{\ensuremath{\mathbf{z}}}

	\newcommand{\Gm}{\ensuremath{\mathbb{G}_m}}
		
	\newcommand{\Hom}{\ensuremath{\operatorname{Hom}}}
    \newcommand{\sym}{\ensuremath{\operatorname{Sym}}}
    \newcommand{\Hilb}{\ensuremath{\operatorname{Hilb}}}
    \newcommand{\NHilb}{\ensuremath{\operatorname{NHilb}}}
    \newcommand{\naNHilb}{\ensuremath{\operatorname{naNHilb}}}
	\newcommand{\id}{\ensuremath{\operatorname{id}}}

	\newcommand{\spec}{\ensuremath{\operatorname{Spec}}}
		
        \newcommand{\sires}{\ensuremath{\operatorname{Res}}}

	\newcommand{\im}{\ensuremath{\operatorname{im}}}
		
    \newcommand{\oo}{\ensuremath{\mathcal{O}}}

\addtocounter{page}{0}
\theoremstyle{plain}
\newtheorem{theorem}{Theorem}[section]
\newtheorem{corollary}{Corollary}[section]
\newtheorem{lemma}{Lemma}[section]
\newtheorem{prop}{Proposition}[section]
\theoremstyle{definition}
\newtheorem{exmp}{Example}[section]
\newtheorem{defi}{Definition}[section]
\newtheorem{remark}{Remark}[section]

\newtheorem{set}{Setup}[section]
\newtheorem{warning}{Warning}[section]

\title{Virtual invariants from the non-associative Hilbert scheme}
\author{Gergely Bérczi and Felix Minddal}

\begin{document}
\maketitle
\begin{abstract}
We introduce a non-associative model for the Hilbert scheme of points in arbitrary dimension. We define a smooth ambient space, which we call the non-associative Hilbert scheme, containing the classical nested Hilbert scheme $\operatorname{NHilb}^{\underline{d}}(\A^n)$ as the associativity, cut out by an explicit section of an associativity bundle. This construction yields canonical perfect obstruction theories and virtual fundamental classes on $\operatorname{NHilb}^{\underline{d}}(\A^n)$ for all $(n,\underline d)$. Using virtual localization, we obtain closed formulas for these virtual classes as sums over admissible nested partitions. Over the punctual locus, we rewrite these as a single multivariable iterated residue formula governing all virtual integrals. Our construction works for all 
$n$, produces positive-dimensional virtual classes when 
$n$ is large compared to the number of points, and we expect that they extend the non-commutative matrix model and virtual class construction on Calabi-Yau threefolds.
\end{abstract}
\tableofcontents
\section*{Introduction}

Hilbert schemes of points are fundamental objects in algebraic and enumerative geometry. For a smooth surface $S$, the Hilbert scheme $\operatorname{Hilb}^d(S)$ of length-$d$ zero-dimensional subschemes is a smooth, irreducible variety of dimension $2d$ \cite{Fogarty}. The topology and intersection theory of $\Hilb^d(S)$ are very well understood, and they carry rich structures such as Nakajima’s Heisenberg algebra action \cite{NakajimaHilb}, quantum cohomology structures in the work of Okounkov-Pandharipande \cite{OkounkovPandharipandeQCohHilb}, and relations to curve counting and modular forms via K3 surfaces.

In higher dimension, the geometry quickly becomes more subtle. For a threefold $X$, the Hilbert scheme $\Hilb^d(X)$ is typically singular for $d>1$ but admits a symmetric perfect obstruction theory in the Calabi-Yau case. Behrend-Fantechi’s theory of intrinsic normal cones and perfect obstruction theories produces a virtual fundamental class $[\Hilb^d(X)]^{\mathrm{vir}}$ \cite{BehrendFantechi}, and Behrend’s constructible function describes Donaldson-Thomas invariants as weighted Euler characteristics \cite{BehrendDT}. Applied to Calabi-Yau threefolds, this yields degree-zero Donaldson-Thomas (DT) invariants as weighted Euler characterisitics on $\Hilb^d(X)$.

In the special important case when $X=\A^3$, $\Hilb^d(\A^3)$ parametrizes ideals of colength $d$ in $\C[x_1,x_2,x_3]$ and is highly singular. Nevertheless, it can be modeled as a critical locus of a function on a smooth space of non-commuting matrices: one considers triples of matrices $(A_1,A_2,A_3)\in\mathrm{End}(V)^3$ together with a vector $v\in V$ generating $V$ under the $A_i$, and imposes the commuting relations $[A_i,A_j]=0$ via a potential \cite{BehrendFantechiSymmetric,BryanSzendroi}. This non-commuting matrix model leads to explicit formulae for DT invariants of $\A^3$, recovering the MacMahon product formula for the generating series of zero-dimensional DT invariants \cite{MNOPI,MNOPII}. From our point of view, this is the first instance of a general philosophy: one embeds a singular Hilbert scheme into a smoother, more flexible moduli space (here of unconstrained or non‑commuting matrices) and then imposes the constraints (here of commutativity) as equations cutting out the original moduli space.

\subsection*{Moduli of sheaves.}
The special role of threefolds and fourfolds in enumerative geometry is by now well established. On a Calabi-Yau threefold $X$, moduli spaces of stable sheaves or stable pairs carry symmetric perfect obstruction theories and hence zero-dimensional virtual classes. Integrals over these classes define Donaldson-Thomas invariants, which may be viewed as algebraic counterparts of BPS state counts in string theory and are related to Gromov-Witten invariants via the DT/PT correspondence \cite{ThomasDT,PTStablePairs}.

For Calabi-Yau fourfolds, the picture is more delicate: the derived moduli stack of (semi)stable sheaves is expected to be $(-2)$-shifted symplectic in the sense of Pantev-Toën-Vaquié-Vezzosi. Using this structure, Borisov-Joyce showed that such a derived moduli space carries a real virtual fundamental class of half the expected complex dimension, constructed via derived differential geometry and Kuranishi atlases \cite{BorisovJoyce}. Their theory applies in particular to moduli spaces of stable sheaves on Calabi-Yau fourfolds (including Hilbert schemes of points), but the resulting cycle is a priori only defined in real homology and depends on subtle orientation data.

Building on this, Oh-Thomas constructed an algebraic virtual cycle for moduli spaces of sheaves on Calabi-Yau fourfolds, using a localized square‑root Euler class for orthogonal bundles and a theory of complex Kuranishi structures. They proved a localization formula and showed that their algebraic cycle agrees with the Borisov-Joyce cycle after inverting~2 \cite{OhThomasCY4I,OhThomasCY4II}. This provides a robust framework for defining DT‑type invariants on Calabi–Yau fourfolds and for comparing different constructions.

In parallel, Feyzbakhsh-Thomas have developed a powerful wall-crossing and stability-conditions approach to Donaldson-Thomas invariants on Calabi-Yau threefolds. They express general rank-$r$ DT invariants in terms of simpler rank 0 or rank 1 theories via Bridgeland stability and wall crossing \cite{FeyzbakhshThomasRank1,FeyzbakhshThomasRank0}.

Within this sheaf-theoretic framework, Hilbert schemes of points on Calabi-Yau fourfolds play the role of moduli of zero-dimensional sheaves. Let $X$ be a smooth projective Calabi-Yau fourfold and $\Hilb^d(X)$ its Hilbert scheme of length-$d$ subschemes. Using the Borisov-Joyce theory together with orientability results of Cao-Gross-Joyce for moduli of sheaves on Calabi-Yau fourfolds \cite{CaoGrossJoyceOrientability}, one obtains a virtual class on $\Hilb^d(X)$. Cao-Kool then defined zero-dimensional $DT_4$ invariants by integrating the Euler classes of tautological bundles against this virtual class, and proved deformation invariance and a number of structural properties \cite{CaoKoolDT4}. These invariants can be viewed as fourfold analogues of degree-zero DT invariants on Calabi-Yau threefolds.

More recently, Bojko has used the Gross-Joyce-Tanaka wall-crossing formalism to study topological information encoded by the virtual fundamental classes of $\Hilb^d(X)$ as $d$ varies. Assuming the conjectural wall-crossing framework, Bojko proves the Cao-Kool conjectures for zero-dimensional sheaf counting and shows that the collection of virtual classes on $\Hilb^d(X)$ is equivalent to the data of all descendent integrals. As a consequence, many generating series of invariants are expressed in terms of explicit universal power series and are related to Nekrasov’s “genus” for $\Hilb^d(\C^4)$ and to Segre-Verlinde correspondences involving Quot schemes on elliptic surfaces \cite{BojkoHilbCY4}.

On the local and toric side, Nekrasov \cite{NekrasovMagnificent4} introduced an equivariant K-theoretic “Magnificent Four” partition function, which can be interpreted as a new genus of Hilbert schemes of points on $\C^4$. Cao-Kool-Monavari \cite{CaoKoolMonavariDTPT4} extended this to curves and stable pairs, formulating a K‑theoretic DT/PT correspondence for toric Calabi–Yau fourfolds and verifying it via a vertex formalism.

In a different but related direction, Oberdieck et al \cite{CaoOberdieckTodaGV4,CaoOberdieckTodaStablePairs4}have developed curve-counting theories on holomorphic symplectic fourfolds, including holomorphic symplectic analogues of Gopakumar-Vafa invariants . These are often expressed in terms of invariants of Hilbert schemes of points on K3 surfaces and their higher-dimensional analogues and connect the 4‑dimensional Hilbert scheme geometry to modularity and Jacobi forms.

\subsection*{The non‑associative nested Hilbert scheme.}
For a smooth variety $X$ of dimension $n$ and a dimension vector $\underline{d}=(d_0,\dots,d_r)$, the nested Hilbert scheme

\[\NHilb^{\underline{d}}(X)
=\{\xi_1\subset\cdots\subset\xi_{r+1}\subset X \Big| \dim\xi_i=0,\operatorname{length}(\xi_i)=d_0+\cdots+d_{i-1}\}
\]
parametrizes chains of zero-dimensional subschemes of prescribed lengths. Even on surfaces, these nested schemes are typically singular with poorly understood component structure. Nonetheless, they have proven extremely useful in sheaf counting theories. Gholampour-Sheshmani-Yau \cite{GSYNestedHilbI} constructed virtual classes on nested Hilbert schemes of points and curves on surfaces and showed that they recover reduced stable pair and Seiberg-Witten theory. Gholampour-Thomas \cite{GholampourThomasI,GholampourThomasII} then expressed nested Hilbert schemes as virtual resolutions of degeneracy loci of maps of bundles on smooth ambient varieties, and used Thom-Porteous‑type formulas to compute the resulting virtual cycles and their integrals. This provides a powerful technique for expressing Vafa-Witten and local DT/PT invariants in terms of tautological classes on nested Hilbert schemes of surfaces.

In this work we use a different smooth ambient space for Hilbert scheme of points. We will reinterpret the work of B\'erczi-Szenes \cite{bsz,bercziGT} and Kazarian \cite{KazarianNAHilb} where they showed that a certain filtered part of the punctual Hilbert scheme embeds into a smooth non‑associative Hilbert scheme and used this to compute Thom polynomials of Morin singularity classes.

In this paper, we extend the non-associative Hilbert scheme model to the whole nested Hilbert schemes. Fix a dimension vector $\underline{d}=(d_0,\dots,d_r)\in\Z_{\ge0}^{r+1}$ and $n\ge0$. Consider the nested Hilbert scheme of $d=d_0+\cdots+d_r$ points in $\A^n$ given as the moduli space 
\begin{align*}\operatorname{NHilb}^{\underline{d}}(\A^n)&=\left\{\xi_{1}\subset \cdots \subset \xi_{r+1}\subset \A^n \: | \: \operatorname{dim}\xi_i=0, \: \operatorname{Length}\xi_i=d_0+\cdots +d_{i-1}\right\}\\
&=\left\{I_{r+1}\subset \cdots \subset I_1\subset I_0=\C[x_1,\dots,x_n] \: | \: \operatorname{dim}_{\C}I_i/I_{i+1}=d_i \right\}.\end{align*}
We define the non-associative nested Hilbert scheme of $d=d_0+\cdots+d_r $ points on $\A^n$ as the moduli space 
\[\operatorname{naNHilb}^{\underline{d}}(\A^n)=\left\{I_{r+1}\subset \cdots \subset I_1\subset I_0=\C\{x_1,\dots,x_n\}_c \: | \: \operatorname{dim}_{\C}I_i/I_{i+1}=d_i \right\}\]
where $\C\{x_1,\dots,x_n\}_c$ is the free unital, commutative, non-associative algebra on $n$-generators and the $I_i$'s are ideals. We prove that this is a smooth variety and that $\operatorname{NHilb}^{\underline{d}}(\A^n)$ embeds in $\operatorname{naNHilb}^{\underline{d}}(\A^n)$ as the locus where $I_1$ is contained in the associativity ideal or equivalently where the quotient algebras are associative. 

The key point is that associativity can be imposed by a section of a vector bundle over the smooth ambient space $\naNHilb^{\underline{d}}(\A^n)$. That is, the associator conditions are encoded in a vector bundle $E_{\operatorname{ass}}$ on $\naNHilb^{\underline{d}}(\A^n)$ and a global section
\[
s_{\operatorname{ass}}\in\Gamma(\naNHilb^{\underline{d}}(\A^n),E_{\operatorname{ass}})
\]
whose zero locus is precisely $\NHilb^{\underline{d}}(\A^n)$. Thus $\NHilb^{\underline{d}}(\A^n)$ has a canonical perfect obstruction theory as the zero locus of a section on a smooth variety, and hence a virtual fundamental class
\[
[\NHilb^{\underline{d}}(\A^n)]^{\mathrm{vir}}\in A_* (\NHilb^{\underline{d}}(\A^n)).
\]
This provides an elementary and concrete construction of a virtual class for nested Hilbert schemes of points in all dimensions.

Our approach should be contrasted with the non-commuting matrix models used in threefold DT theory. There, one replaces the commutative coordinate ring by a free associative (non-commutative) algebra of matrices, and the Hilbert scheme is realized as a critical locus of a function cutting out the commuting relations. In our non-associative model, we instead keep commutativity but drop associativity, realizing the Hilbert scheme as the zero locus of associativity equations inside a smooth commutative but non-associative ambient space. This shift turns out to be well adapted to the combinatorics of monomial ideals and yields ambient varieties that are themselves built as towers of Grassmannian fibrations, in the spirit of Kazarian’s work. Crucially, our construction produces virtual classes of positive dimension whenever the ambient dimension $n$ is sufficiently large compared to the number of points (or total length) $d=d_0+\cdots d_r$, and our iterated residue formula for virtual integration then provides a uniform way to define and compute these virtual classes for all pairs $(n,\underline{d})$.

\subsection*{Virtual localization and iterated residue formulas.}
The affine torus $T=(\C^*)^n$ acts on $\A^n$ by scaling each coordinate, and this action lifts to both $\naNHilb^{\underline{d}}(\A^n)$ and $\NHilb^{\underline{d}}(\A^n)$. Our construction of the obstruction theory is $T$‑equivariant, so we obtain an equivariant virtual class
\[\NHilb^{\underline{d}}(\A^n)]^{\mathrm{vir}}\in A_*^T(\NHilb^{\underline{d}}(\A^n)).
\]
Then we apply the Graber–Pandharipande virtual localization formula to compute the restriction of this class in the localized equivariant Chow group $A_*^T(\NHilb^{\underline{d}}(\A^n))_{\operatorname{loc}}$ as a sum over $T$‑fixed points. These fixed points correspond to chains of monomial ideals, i.e., to nested partitions. We obtain an explicit iterated residue formula expressing $[\NHilb^{\underline{d}}(\A^n)]^{\mathrm{vir}}$ as a weighted sum over admissible $(n-1)$‑dimensional nested partitions of length $\underline{d}$, with weights given by rational functions in the torus weights (Theorem~\ref{thm:maina}).

We also single out the nilpotently filtered locus $\NHilb^{\underline{d}}_{\operatorname{nil-fil}}(\A^n)$, consisting of nested ideals
\[
I_{r+1}\subset\cdots\subset I_1\subset I_0=\C[x_1,\dots,x_n]
\]
such that $I_1$ is the maximal ideal at the origin and $I_1\cdot I_i\subset I_{i+1}$ for all $i$. This locus lies inside the punctual nested Hilbert scheme and has a natural virtual class induced by our embedding into the non‑associative model. When $\underline{d}=(1,\dots,1)$ is a full flag, we show that $\NHilb^{\underline{d}}_{\operatorname{nil-fil}}(\A^n)$ coincides with the usual punctual nested Hilbert scheme, and more generally we study its geometry for arbitrary $\underline{d}$. When $r=1$, $d_1=d$, we identify
\[
\NHilb^{\underline{d}}_{\operatorname{nil-fil}}(\A^n)
=\{I\subset\C[x_1,\dots,x_n]\Big| \dim_\C\C[x]/I=d+1,\mathfrak{m}^2\subset I\subset\mathfrak{m}\},
\]
which is empty whenever $d>n$; this illustrates that the nilpotently filtered locus is much smaller than the full punctual Hilbert scheme.

For virtual integrals supported on the nilpotently filtered locus $\NHilb^{\underline{d}}_{\operatorname{nil-fil}}(\A^n)$, we refine the localization calculation to an iterated residue formula. When $d_1+\cdots+d_r\leq n$ we construct a partial resolution of $\NHilb^{\underline{d}}_{\operatorname{nil-fil}}(\A^n)$ fibering over a flag variety $\operatorname{Flag}_{\hat{d}}(\C^n)$ (where $\hat{d}=(d_1,\cdots,d_r)$ and reduce integrals to a fiber over a reference flag. Following \cite{bsz}, Atiyah-Bott localization on this resolved space can be transformed into a multi-variable rational form whose virtual integral is given by an iterated residue at infinity. A key feature of our analysis is that all but one of the candidate residues vanish; the sole contribution comes from the “Porteous” nested partition consisting only of unit vectors in the natural order. The resulting formula expresses virtual integrals on $\NHilb^{\underline{d}}_{\operatorname{nil-fil}}(\A^n)$ as a single explicit iterated residue (Theorem \ref{thm:mainb}). This generalizes earlier iterated residue techniques of \cite{bsz} in singularity theory and Thom polynomial computations to the context of nested Hilbert schemes of points and virtual integration. 

Finally, using an embedding
\[
\NHilb^{\underline{d}}_{\operatorname{nil-fil}}(\A^n)\hookrightarrow \NHilb^{\underline{d}}_{\operatorname{nil-fil}}(\A^N)
\]
for $N\gg n$ and comparing the corresponding virtual structures, we show that the same residue formula continues to hold even when $d_1+\cdots +d_r>n$.

\subsection*{Tautological versus virtual intersection theory.} The curvilinear Hilbert scheme, denoted by $\mathrm{CHilb}^d_0(\A^n)$, parametrizes infinitesimally collinear subschemes of a smooth variety and forms a small, distinguished component of the punctual Hilbert scheme. In \cite{bercziGT} it is shown that this curvilinear component forms a natural projective compactification of the jet moduli space $J_k(1,n)/\mathrm{Diff}_k$. This perspective led to a systematic theory of tautological intersection numbers on curvilinear and full Hilbert schemes of points. In a series of papers \cite{bercziGT,berczitau2,berczitau3,berczitau4}, tautological integrals were expressed in terms of explicit iterated residue formulas for all tautological integral on the main component of the Hilbert schemes of points in arbitrary dimension. These formulas recover and generalise classical surface results of Göttsche and Nakajima, and they apply to a wide range of enumerative problems, including counting curves and hypersurfaces with prescribed singularities and higher-dimensional analogues of Lehn’s conjecture and Segre–Verlinde type formulas. A striking structural property of Hilbert scheme of points proved in \cite{berczisvendsen} is that all monomial ideals (i.e.torus fixed points) lie on the curvilinear component, strongly suggesting that the curvilinear locus controls the topology of the entire Hilbert scheme.

The present work develops a parallel virtual intersection theory on the full Hilbert schemes of points, not just on thew main (smoothable) component. Crucially, our construction yields virtual classes of positive dimension whenever the ambient dimension $n$ is sufficiently large compared to the total length of the subscheme, and the virtual localization formulas can again be written as iterated residues, now incorporating the contributions of the obstruction theory. In this sense, the virtual intersection theory developed here is formally very close to the tautological intersection theory of \cite{bercziGT,berczitau2,berczitau3,berczitau4}: in both settings, universal intersection numbers on Hilbert schemes of points are governed by explicit multi-variable residue formulas over the same combinatorial data of monomial ideals and partitions. In forthcoming work \cite{MinddalBerczi-NAHilb-III}, this framework is used to reconstruct Donaldson–Thomas invariants of Calabi–Yau threefolds from the non-associative Hilbert scheme model, revealing a direct bridge between the tautological and virtual theories.

\subsection*{Outline of the paper.}
In Section~\ref{sec:NA-nested-Hilb} we define the non‑associative nested Hilbert scheme $\naNHilb^{\underline{d}}(\A^n)$ and prove its smoothness and basic geometric properties, including the embedding of $\NHilb^{\underline{d}}(\A^n)$ as the associativity locus. In ~\ref{sec:NA-punct-Hilb} we construct the punctual locus, prove that it can be described as a tower of Grassmannian fibrations and construct a non-associative Hilbert-Chow morphism. In Section~\ref{sec:virtual-class} we construct the induced perfect obstruction theory and virtual class on $\NHilb^{\underline{d}}(\A^n)$ and derive an explicit $T$‑equivariant localization formula, expressed as a sum over admissible nested partitions. Section~\ref{sec:nil-fil} is devoted to the nilpotently filtered locus: we define $\NHilb^{\underline{d}}_{\operatorname{nil-fil}}(\A^n)$, analyze its geometry, construct a partial resolution over a flag variety and prove the iterated residue formula for virtual integrals.

\subsection*{Acknowledgements.} We thank Sergej Monavari for many enlightening discussions and for suggesting several interesting directions for us to explore.

\section{The results}
Fix a dimension vector $\underline{d}=(d_0,\dots, d_{r})\in\Z_{\geq 0}^{r+1}$ and $n\in \Z_{\geq 0}$. The main object of interest in this paper is the nested Hilbert scheme of $d=d_0+\cdots+d_r$ points in $\A^n$ given as the moduli space 
\begin{align*}\operatorname{NHilb}^{\underline{d}}(\A^n)&=\left\{\xi_{1}\subset \cdots \subset \xi_{r+1}\subset \A^n \: | \: \operatorname{dim}\xi_i=0, \: \operatorname{Length}\xi_i=d_0+\cdots +d_{i-1}\right\}\\
&=\left\{I_{r+1}\subset \cdots \subset I_1\subset I_0=\C[x_1,\dots,x_n] \: | \: \operatorname{dim}_{\C}I_i/I_{i+1}=d_i \right\}.\end{align*}
We study it by defining a new smooth variety for which this (highly singular) scheme embeds. We use this to define virtual fundamental lass on the nested Hilbert scheme and develop localization and iterated residue formulas for efficiently computing virtual invariants.

\subsection*{The non-associative Hilbert scheme.}
Extending the work of \cite{KazarianNAHilb} we construct $\operatorname{naNHilb}^{\underline{d}}(\A^n)$ the non-associative nested Hilbert scheme of $d=d_0+\cdots+d_r $ points on $\A^n$. It is the moduli space 
\[\operatorname{naNHilb}^{\underline{d}}(\A^n)=\left\{I_{r+1}\subset \cdots \subset I_1\subset I_0=\C\{x_1,\dots,x_n\}_c \: | \: \operatorname{dim}_{\C}I_i/I_{i+1}=d_i \right\}\]
where $\C\{x_1,\dots,x_n\}_c$ is the free unital, commutative, non-associative algebra on $n$-generators and the $I_i$'s are ideals. We prove that this is a smooth variety and that $\operatorname{NHilb}^{\underline{d}}(\A^n)$ embeds in $\operatorname{naNHilb}^{\underline{d}}(\A^n)$ as the locus where $I_1$ is contained in the associator ideal or equivalently where the quotient algebras are associative. \newline 

In the case where $d_0=1$ we construct a punctual locus $\operatorname{naNHilb}^{\underline{d}}_0(\A^n)\subset \operatorname{naNHilb}^{\underline{d}}(\A^n)$ which intersects the Hilbert scheme in the nilpotently filtered locus
\[\operatorname{NHilb}^{\underline{d}}_{\operatorname{nil-fil}}(\A^n)=\left\{[I_{r+1}\subset \cdots \subset I_1\subset I_0]\in\operatorname{NHilb}^{\underline{d}}(\A^n) \: | \:  I_1=\mathfrak{m}, \: I_1\cdot I_i \subset I_{i+1} \right\}\]
where $\mathfrak{m}=(x_1,\dots, x_n)$. Using the methods of \cite{KazarianNAHilb} we prove $\operatorname{naNHilb}^{\underline{d}}_0(\A^n)$ admits the structure of a tower of Grassmannian fibrations. In the special case where $d_i=1$ for all $0\leq i\leq r$ we are further able to construct an extension
\[\operatorname{naNHilb}^{\underline{d}}(\A^n)\to (\A^n)^{r+1}\]
of the (ordered) Hilbert-Chow morphism on $\operatorname{NHilb}^{\underline{d}}(\A^n)$. We use this to prove that $\operatorname{NHilb}^{\underline{d}}_{\operatorname{nil-fil}}(\A^n)$ identifies with the punctual Hilbert scheme of points
\[\operatorname{NHilb}_0^{\underline{d}}(\A^n)=\left\{\xi_{1}\subset \cdots \subset \xi_{r+1}\subset \A^n \: | \: \operatorname{dim}\xi_i=0, \: \operatorname{Length}\xi_i=i, \: \operatorname{Supp}\xi_i=\{0\} \right\}.\]
We always have $\operatorname{NHilb}_{\operatorname{nil-fil}}^{\underline{d}}(\A^n)\subset\operatorname{NHilb}_0^{\underline{d}}(\A^n)$ but the equality only holds for the full dimension vector where $d_i=1$ for all $i$. On the other end of the spectrum where $r=1$, $d_1=d$ we can identify
\[\operatorname{NHilb}^{\underline{d}}_{\operatorname{nil-fil}}(\A^n)=\left\{I\subset \C[x_1,\dots, x_n] \: | \:\operatorname{dim}_{\C}\C[x_1,\dots, x_n]/I=d+1  ,\mathfrak{m}^2\subset I \subset \mathfrak{m} \right\}.\]
This is much smaller than the punctual locus and becomes empty whenever $d>n$.

\subsection*{Virtual fundamental class and localization.}
One can easily prove that the associativity locus is the zero locus of a section of a vector bundle on $\operatorname{naNHilb}^{\underline{d}}(\A^n)$. This equips $\operatorname{NHilb}^{\underline{d}}(\A^n)$ with a perfect obstruction theory and therefore also a virtual fundamental class. Setting $T=(\C^*)^n$ this procedure can be made $T$-equivariant and we are able to calculate the virtual class via a localization formula. We obtain the following result, where we use the notation
\[\widetilde{\prod} \gamma:=\prod_{\gamma \neq 0} \gamma \]
for non-zero products.
\begin{theorem}\label{thm:maina}
    Let $\underline{d}=(d_0,\dots,d_r)\in \Z^{r+1}_{\geq 0}$, $d=d_0+\cdots +d_r$, $n\in \Z_{\geq 0}$ and $T=(\C^*)^n$. Let $A_T^*(\operatorname{pt})=\Z[s_1,\dots, s_n]$ where $s_1,\dots, s_n$ are the Chern roots of $\C^n$ given the standard torus action. In the Chow group $A_*^T(\operatorname{NHilb}^{\underline{d}}(\A^n))_{\operatorname{loc}}$ localized in linear functions of $s_1,\dots, s_n$ the virtual fundamental class $[\operatorname{NHilb}^{\underline{d}}(\A^n)]^{\operatorname{vir}}$ equals
    \[\sum_{\lambda_\bullet} [\lambda_\bullet]\cap\frac{\prod_{\substack{1\leq i,j,k\leq d-1 \\ i<k}} \widetilde{\prod}_{\substack{0\leq m\leq d-1\\ w(k),w(j)\leq w(m)}}(\mathbf{u}_i+\mathbf{u}_j+\mathbf{u}_k-\mathbf{u}_m)(\underline{s})\widetilde{\prod}_{j=1}^{d-1}\prod_{\substack{0\leq k\leq d-1\\ w(j)\leq w(k)}}\mathbf{u}_j-\mathbf{u}_k(\underline{s})}{\prod_{i=1}^n\widetilde{\prod}_{k=0}^{d-1}(e_i-\mathbf{u}_k)(\underline{s})\prod_{1\leq i\leq j\leq d-1} \widetilde{\prod}_{\substack{0\leq k\leq d-1\\ w(j)\leq w(k)}}(\mathbf{u}_i+\mathbf{u}_j-\mathbf{u}_k)(\underline{s})}.\]
    Here the sum is over all admissible $(n-1)$-dimensional nested partitions of length $\underline{d}$ 
    \[\lambda_\bullet=(\lambda_1\subset \cdots \lambda_{r+1}\subset \Z_{\geq 0}^n)\]
    viewed as a fixed point $\lambda_\bullet\in \operatorname{NHilb}^{\underline{d}}(\A^n)^T$ and where we have a chosen enumeration 
    \[\lambda_{k+1}=\{\mathbf{u}_0,\dots, \mathbf{u}_{d_0+\cdots +d_{k-1}-1}\}\subset \Z_{\geq 0}^n \]
    for each $\lambda_\bullet$. For $\mathbf{u}=(u_1,\dots, u_n)\in \Z^n$ we define
    \[\mathbf{u}(\underline{s})=u_1s_1+\cdots +u_ns_n.\]
    Being admissible means that $\lambda_{r+1}$ contains no vectors with 3 or more non-zero coordinates and that the projection to a two-dimensional plane sits inside the partition given by figure 1 below.
\end{theorem}

\begin{figure}[ht]
    \centering
\begin{ytableau}
\none  & & \none& \none& \none& \none  \\
\none  & & \none& \none& \none& \none  \\
\none  & & & \none& \none& \none  \\
\none  & & & & \none& \none  \\
\none  & & & & & \\
\end{ytableau}
\caption{}

\end{figure}

\subsection*{The iterated residue formula.}
Similarly, when $d_0=1$ the nilpotently filtered Hilbert scheme $\operatorname{NHilb}^{\underline{d}}_{\operatorname{nil-fil}}(\A^n)$ admits a virtual fundamental class. When the flag is full, i.e., $d_i=1$ for all $i$ we are able to describe virtual integrals on $\operatorname{NHilb}^{\underline{d}}(\A^n)$ in terms of a virtual integral on $\operatorname{NHilb}^{\underline{d}}_{\operatorname{nil-fil}}(\A^n)$. When the flag is not full we can still perform this procedure, provided the integrand is supported on the nilpotently filtered locus. We spent the last section of the paper developing an iterated residue for virtual integrals over the nilpotently filtered locus. If we let $\underline{d}=(1,\hat{d})=(1,d_1,\dots,d_r)$ and assume
\[d=1+d_1+\cdots +d_r\leq n+1\]
we can construct a partial resolution of $\operatorname{NHilb}_{\operatorname{nil-fil}}^{\underline{d}}(\A^n)$ that fibers over the flag variety $\operatorname{Flag}_{\hat{d}}(\C^n)$. We can use this to derive a residue formula depending on 
\[\int_{[H_{\id}]^{\operatorname{vir}}}\alpha\]
where $H_{\id}$ is the fiber in the partial resolution over the reference flag
\[f_\bullet=(\operatorname{span}(e_1,\dots, e_{d_0})\subset \cdots \subset \operatorname{span}(e_1,\dots, e_{\hat{d}}))\in \operatorname{Flag}_{\hat{d}}(\C^n).\]
Calculating this integral via localization, we get a sum of iterated residues and we can prove that all residues vanish except for the one at the Porteous point i.e., the nested partition containing only unit vectors nested in the natural way. By considering the map
\[\operatorname{NHilb}_{\operatorname{nil-fil}}^{\underline{d}}(\A^n)\hookrightarrow \operatorname{NHilb}^{\underline{d}}_{\operatorname{nil-fil}}(\A^N)\]
for $n<N$ we can prove that this formula remains valid even when $d>n+1$. We obtain the following result.
\begin{theorem}\label{thm:mainb}
     Let $\underline{d}=(1,d_1\dots,d_r)\in \Z^{r+1}_{\geq 0}$, $d=1+d_1+\cdots +d_r$, $n\in \Z_{\geq 0}$, and $T=(\C^*)^n$. Let $A_T^*(\operatorname{pt})=\Z[s_1,\dots, s_n]$ where $s_1,\dots, s_n$ are the Chern roots of $\C^n$ given the standard torus action. Pick a $q$-dimensional $T$-representation $W$ with Chern roots $\theta_1,\dots, \theta_q$ and consider the $T$-equivariant rank $q$ bundle 
     \[E=W\otimes \A^n\]
     on $\A^n$. Let
    \[-\eta_0, -\eta_1,\dots, -\eta_{d-1}\]
    be the Chern roots of the tautological bundle $\oo_{\A^n}^{[\underline{d}]}$ so that
    \[E^{[\underline{d}]}\cong W\otimes \oo_{\A^n}^{[\underline{d}]}\]
    has Chern roots $\theta_i-\eta_j$. Let
    \[P=P(\theta_1,\dots, \theta_q,\eta_1,\dots, \eta_{d-1})\]
    be a polynomial in the Chern classes of $E^{[\underline{d}]}$. Then 
    \[\int_{[\operatorname{NHilb}^{\underline{d}}_{\operatorname{nil-fil}}(\A^n)]^{\operatorname{vir}}}P \]
    can be calculated via the iterated residue formula
    \[\underset{\underline{z}=\infty}{\operatorname{Res}}\frac{P(\theta_1,\dots,\theta_q,\underline{z})\prod_{1\leq m\neq l\leq d-1}(z_m-z_l)\prod_{\substack{1\leq i,j,k\leq d-1 \\ i<k}} \widetilde{\prod}_{\substack{0\leq m\leq d-1\\ w(k),w(j)\leq w(m)}}(z_i+z_j+z_k-z_m)d\underline{z}}{\prod_{\substack{1\leq l<m\leq d-1 \\ w(l)<w(m)}}(z_m-z_l)\prod_{i=1}^n \prod_{l=1}^{d-1}(s_i-z_l)\prod_{1\leq i\leq j\leq d-1} \widetilde{\prod}_{\substack{1\leq k\leq d-1\\ w(j)< w(k)}}(z_i+z_j-z_k)}\]
    Here $\underline{z}=(z_1,\dots, z_{d-1})$ and the iterated residue is defined as the coefficient of $(z_1\cdots z_{d-1})$ in the expansion of the rational expression in the domain $z_1\ll \cdots \ll z_{d-1}$.
\end{theorem}

In the follow-up paper \cite{MinddalBerczi-NAHilb-III}, we will further develop the theory by comparing the virtual invariants obtained in this paper with other known virtual invariants using a rational map from the non-associative to the non-commutative nested Hilbert scheme. We note that several virtual invariants of the Hilbert scheme of points can be constructed explicitly using the non-commutative Hilbert scheme. The degree 0 rank 1 DT-invariants on $\operatorname{Hilb}^{d}(\A^3)$ can be described using the critical locus description in the non-commutative Hilbert scheme, see \cite{BryanSzendroi, Okounkov2017KTheory}. The Hilbert scheme of points on $\A^4$ can be exhibited as a zero locus of an isotropic section of an orthogonal bundle on the non-commutative Hilbert scheme thus providing a virtual class in the sense of \cite{OhThomasCY4I}, see \cite{kool2025proofmagnificentconjecture}. In \cite{MinddalBerczi-NAHilb-II} a nested version of the non-commutative Hilbert scheme is constructed and it is shown that  $\operatorname{NHilb}^{\underline{d}}(\A^2)$ can be exhibited as a zero locus of a section of a vector bundle in a way that recovers the virtual class of \cite{GSYNestedHilbI}. \newline 

In more detail, the existence and smoothness of the nested non-commutative Hilbert scheme 
\[
\operatorname{ncNHilb}^{\underline{d}}(\A^n),
\]
is proven in \cite{MinddalBerczi-NAHilb-II}. It is the moduli space parameterizing $(N_\bullet, A_1,\dots, A_n, v)$, where $N_\bullet$ is a flag of vector spaces of dimension $\underline{d}$, each $A_i$ is a flag-preserving operator, and $v \in N_0$ is a cyclic vector, meaning that every element of $N_\bullet$ is a linear combination of $A_{i_1} A_{i_2} \cdots A_{i_m} v$. This construction specializes to the ordinary non-commutative Hilbert scheme in the case where $\underline{d}=(d)$. We consider the rational map
\[
\begin{tikzcd}
\operatorname{naNHilb}^{\underline{d}}(\A^n) \arrow[dashed]{r} & \operatorname{ncNHilb}^{\underline{d}}(\A^n) \\[-20pt]
[N_\bullet, \psi_1, \psi_2, v] \arrow[mapsto]{r} & \left[ N_\bullet, \psi_2(\psi_1(x_1),-), \dots, \psi_2(\psi_1(x_n),-), v \right]
\end{tikzcd}
\]
where $(\C^n)^\vee = \operatorname{span}(x_1,\dots,x_n)$. In other words, it is the same flag $N_\bullet$ with the same unit $v$, and the operators are defined by left multiplication with $\psi_1(x_i)$. The domain of definition is the Zariski open locus where $N_\bullet$ is spanned by $v$ and by elements obtained by iterative left multiplication by $\psi_1(x_1),\dots,\psi_1(x_n)$. The nested Hilbert scheme of points lies in the domain of this map and the map restricts to the identity map of the nested Hilbert scheme of points viewed either as the associativity or the commutativity locus.\newline 

We hope to use this map to construct relative perfect obstruction theories on the identity map of the (nested) Hilbert scheme of points compatible with our obstruction theory in the domain and well-known obstruction theories that can be constructed from the non-commutative (nested) Hilbert scheme in the codomain.

\begin{remark}
One might be tempted to think that the map is well defined everywhere, since 
$y_1=\psi_1(x_1), \dots, y_n=\psi_1(x_n)$ generate $N_\bullet$ as a non-associative algebra. 
This is, however, not the case. The essential issue is that elements of the form
\[
(y_i y_j)(y_k y_m)
\]
cannot \emph{a priori} be written as elements obtained by iterative left multiplication by the $y_\ell$.
\end{remark}

Similarly, if we simply denote the operators by
\[A_{i}:=\psi_2(\psi_1(x_i),-)\]
we can define a potential 
\begin{center}
     \begin{tikzcd}
         \operatorname{naHilb}^{d}(\A^3) \arrow{r}{g} & \A^1 \\[-20pt]
         [N_\bullet, \psi_1,\psi_2,v] \arrow[mapsto]{r} & \operatorname{Tr}A_1[A_2,A_3]
     \end{tikzcd}
\end{center}
which is well defined on the whole non-associative Hilbert scheme. On easily sees that

\begin{center}
    \begin{tikzcd}
       & \A^1 & \\
       \operatorname{naHilb}^{d}(\A^3)  \arrow[dashed]{rr} \arrow{ur}{g} & &  \operatorname{ncHilb}^{d}(\A^3) \arrow[swap]{ul}{f} \\
       & \operatorname{Hilb}^{d}(\A^3) \arrow[swap,hook]{lu} \arrow[hook]{ru} &
    \end{tikzcd}
\end{center}
commutes and that the nested Hilbert scheme lies in the domain of definition of the rational map. 

\begin{remark}
    Note that the critical locus $\operatorname{Crit}(g)$ is exactly the locus where the operators $\psi_2(\psi_1(x_i),-)$ commutes. This commutativity yields associativity conditions. Indeed, let $A$ be a commutative non-associative algebra and for $a\in A$ denote by $L_a$ the operator given by left multiplication of $a$. If $L_aL_b=L_bL_a$ then for all $c\in A$ we have
    \[(ac)b=b(ac)=L_bL_a(c)=L_aL_b(c)=a(bc)=a(cb)\]
    i.e. the associator $[a,c,b]$ vanishes for all $c\in A$.
\end{remark}
\begin{remark}
    One might be tempted to think that $\operatorname{NHilb}^{\underline{d}}(\A^3)=\operatorname{Crit}(g)$. We have $\operatorname{NHilb}^{\underline{d}}(\A^3)\subset\operatorname{Crit}(g)$ but the inclusion is in general strict, even if we restrict to the domain of definition of the rational map. Let us consider the following counter example. Let $N=\operatorname{span}(v_0,v_1,v_2)$. Let $v_0$ be the unit and define a multiplication rule by 
    \begin{align*}
        v_1v_1=v_2,\: v_1v_2=v_2v_1=v_2 \: \text{and} \: v_2v_2=0.
    \end{align*}
    Set $\psi_1(x_1)=\psi_1(x_2)=\psi_1(x_3)=v_1$. This is in the domain of definition of the rational map since everything is spanned by the unit $v_0$ and left multiplications of $v_1$ with itself. Clearly the operators $\psi_2(\psi_1(x_i),-)$ and $\psi_2(\psi_1(x_j),-)$ commutes since $\psi_1(x_i)=\psi_1(x_j)$. However, the algebra is not associative as 
    \[(v_1v_1)v_2=v_2v_2=0\neq v_2=v_1v_2=v_1(v_1v_2).\]
\end{remark}
The rational map defined in this section provides us a direct approach to express DT invariants as tautological virtual integrals on the positive dimensional virtual classes in higher dimensions. 

\section{Preliminaries}
\subsection{Non-associative algebras}
We provide the definition of a non-associative $\oo_S$-algebra and describe how to construct quotients of the free non-associative algebra.
 \begin{defi}
     Let $S$ be a scheme. A (quasi-coherent) non-associative $\oo_S$-algebra is by definition a pair $(\mathcal{F},\psi)$ of a quasi-coherent sheaf $\mathcal{F}\in \operatorname{QCoh}(S)$ equipped with a multiplication rule 
     \[\psi:\mathcal{F}\otimes_S\mathcal{F}\to \mathcal{F}.\]
     A (quasi-coherent) unital non-associative $\oo_S$ algebra is a triple $(\mathcal{F},\psi,s)$ where $(\mathcal{F},\psi)$ is a non-associative $\oo_S$ algebra and $s:\oo_S\to \mathcal{F}$ is a global section such that
     \[\mathcal{F}\simeq \oo_S\otimes_S\mathcal{F}\xrightarrow{s\oplus \operatorname{id}} \mathcal{F}\otimes_S\mathcal{F} \xrightarrow{\psi} \mathcal{F}\]
     and 
    \[\mathcal{F}\simeq \mathcal{F}\otimes_S \oo_S\xrightarrow{\operatorname{id}\oplus s} \mathcal{F}\otimes_S\mathcal{F} \xrightarrow{\psi} \mathcal{F}\]
    agrees with $\id_{\mathcal{F}}$. A non-associative $\oo_S$-algebra $(\mathcal{F},\psi)$ is said to be commutative if $\psi$ factors through 
    \[\sym^2_S\mathcal{F}=\mathcal{F}\otimes_S\mathcal{F}\big{/}(a\otimes b-b\otimes a). \]
    A morphism of non-associative $\oo_S$-algebras $f: (\mathcal{F},\psi)\to (\mathcal{G},\varphi)$ is a morphism of $\oo_S$-modules $f:\mathcal{F}\to \mathcal{G}$ such that
    \begin{center}
        \begin{tikzcd}
            \mathcal{F}\otimes_S \mathcal{F} \arrow{d}{f\otimes f} \arrow{r}{\psi} & \mathcal{F} \arrow{d}{f} \\
            \mathcal{G}\otimes_S \mathcal{G} \arrow{r}{\varphi} & \mathcal{G}
        \end{tikzcd}
    \end{center}
    commutes. A morphism of unital non-associative $\oo_S$-algebras $f: (\mathcal{F},\psi,s)\to (\mathcal{G},\varphi,t)$ is a morphism $f: \mathcal{F}\to \mathcal{G}$ of non-associative $\oo_S$-algebras such that $t=fs$. 
 \end{defi}

\begin{lemma} \label{unique algebra structure surjection}
    Let $S\in \operatorname{Sch}_{\C}$, let $(\mathcal{F},\psi)$ be a unital non-associative algebra and let $\mathcal{G}\in \operatorname{QCoh}(S)$. Let $\pi:\mathcal{F}\twoheadrightarrow \mathcal{G}$ be an epimorphism of $\oo_S$-modules with kernel $\mathcal{I}$. Then there exists a necessarily unique unital non-associative algebra structure $\varphi$ on $\mathcal{G}$ making $\pi$ a morphism of unital non-associative algebras if and only if $\mathcal{I}$ is a two-sided ideal. If $\psi$ is commutative resp. associative, then the same will hold true for $\varphi$.
\end{lemma}
\begin{proof}
Suppose first that $\varphi$ exists. Then 
\[\pi(\psi(\mathcal{I}\otimes_S \mathcal{F}))=\varphi(\pi(\mathcal{I})\otimes_S\pi(\mathcal{F}))=0\]
and
\[\pi(\psi(\mathcal{F}\otimes_S \mathcal{I}))=\varphi(\pi(\mathcal{F})\otimes_S\pi(\mathcal{I}))=0\]
proving that $\psi(\mathcal{I}\otimes_S \mathcal{F}), \psi(\mathcal{F}\otimes_S \mathcal{I})\subset \mathcal{I}$. Conversely, if $\mathcal{I}$ is a two-sided ideal. We have two exact sequences given by
\[\mathcal{I}\otimes_S \mathcal{F}\to \mathcal{F}\otimes_S \mathcal
F\to \mathcal{G}\otimes_S\mathcal{F}\to 0\]
and 
\[\mathcal{G}\otimes_S \mathcal{I}\to \mathcal{G}\otimes_S \mathcal{F}\to \mathcal{G}\otimes_S\mathcal
G\to 0.\]
Since $\pi(\psi(\mathcal{I}\otimes_S\mathcal{F}))=0$ it follows that $\pi\psi$ uniquely induce a map $\mathcal{F}\otimes_S\mathcal{G}\to \mathcal
G$. Since $\pi(\psi(\mathcal{F}\otimes_S\mathcal{I}))=0$ it follows that this induced map takes $\mathcal{G}\otimes_S \mathcal{I}$ to $0$ and hence induces a unique map
\[\varphi: \mathcal{G}\otimes_S\mathcal
G\to \mathcal{G}.\]
Let $\tau_{\mathcal{G}}$ be the swap map of $\mathcal
{G}\otimes_S \mathcal{G}$. Commutativity of $\varphi$ can the be checked by the vanishing of the map $\varphi-\varphi\tau_{\mathcal{G}}$, which by surjectivity of $\pi$ can be checked after precomposing with $\pi\otimes \pi$. This composition yields the map $\pi(\psi-\psi\tau_{\mathcal{F}})$ which vanishes if $\psi$ is commutative. Similarly, associativity can be checked by the vanishing of the map $\varphi(\operatorname{id}\otimes\varphi)-\varphi(\varphi\otimes \operatorname{id})$, which again by surjectivity of $\pi$ can be checked after precomposing with $\pi\otimes \pi \otimes\pi$. This composition gives the map $\pi(\psi(\operatorname{id}\otimes \psi)-\psi(\psi\otimes\operatorname{id}))$ which vanishes if $\psi$ is associative.
\end{proof}
\begin{remark}
    We note that if $(\mathcal{F},\psi)$ is also equipped with a unital structure $s$ then we similarly get a unique unital structure $t$ on $\mathcal{G}$ by setting $t=\pi s$. 
\end{remark}

\begin{lemma} \label{free non-associative}
Let $S\in \operatorname{Sch}_{\C}$ and let $\operatorname{Qlg}(S)$ be the category of unital non-associative $\oo_S$ algebras. The forgetful functor 
\[\operatorname{Qlg}(S)\to \operatorname{QCoh}(S)\]
admits a left adjoint $F_{Qlg}(S)$. 
\end{lemma}
\begin{proof}
One could argue briefly using the adjoint functor theorem. Instead, we give concrete construction. We define for each $k\geq 1$ the set $B(k)$ corresponding to the set of ways to put bracket around a length $k$ expression. One can define this inductively setting $B(1)=\{\ast\}$ and 
\[B(k)=\coprod_{i+j=k}B(i)\times B(j).\]
For a $\mathcal{F}\in \operatorname{QCoh}(S)$ we then set 
\[F_{\operatorname{Qlg}}(\mathcal{F}):=\oo_S\oplus \bigoplus_{k=1}^\infty \bigoplus_{x\in B(k)} \mathcal{F}^{\otimes_{S}k}. \]
The first summand is the unit, and the multiplication rule 
\[F_{\operatorname{QLG}}(\mathcal{F})\otimes_S F_{\operatorname{QLG}}(\mathcal{F})\to F_{\operatorname{QLG}}(\mathcal{F}) \]
is given by first distributing the tensor product over the direct sums and then mapping $\mathcal{F}^{\otimes_Si}\otimes_S\mathcal{F}^{\otimes_Sj}$ corresponding to $x\in B(i)$ and $y\in B(j)$ to the summand corresponding to $(x,y)\in B(i+j)$. Given the datum of a map $f: \mathcal{F}\to \mathcal{G}$ in $\operatorname{QCoh}(S)$ and a unital non-associative structure $s\in H^0(S,\mathcal{G})$, $\psi:\mathcal{G}\otimes_S \mathcal{G}\to \mathcal{G}$ one can then easily via induction construct a necessarily unique map 
\[F_{\operatorname{Qlg}}(\mathcal{F})\to \mathcal{G}\]
of unital non-associative algebras such that the map agrees with $f$ when restricted to $\bigoplus_{x\in B(1)}\mathcal{F}\cong \mathcal{F}$. This proves that the construction is left adjoint to the forgetful functor. 
\end{proof}

\begin{lemma}
    Let $S\in \operatorname{Sch}_{\C}$  and let $\operatorname{Qlg}(S), \operatorname{QClg}(S), \operatorname{QAlg}(S)$ and $\operatorname{QCAlg}$ be the categories of respectively unital non-associative, unital commutative non-associative, unital associative and unital commutative associative $\oo_S$-algebras. The inclusion functors
\begin{center}
    \begin{tikzcd}
        & \operatorname{QClg}(S) \arrow{rd} & \\
        \operatorname{QCAlg}(S) \arrow{ru} \arrow{rd} & & \operatorname{Qlg}(S) \\ 
        & \operatorname{QAlg}(S) \arrow{ru}
    \end{tikzcd}
\end{center}
admits left adjoints. In particular, using \cref{free non-associative} we see that the forgetful functor from any of these categories to $\operatorname{QCoh}(S)$ admits left adjoints.
\end{lemma}
\begin{proof}
For $(\mathcal{F},\psi,s)\in \operatorname{Qlg}(S)$ we define the commutator ideal as the smallest two-sided ideal containing
\[\psi(x,y)-\psi(y,x)\]
for all $x,y\in \mathcal{F}$. Taking the quotient of $\mathcal{F}$ by this ideal yields the left adjoints of the upper right and lower left arrows of the diamond.  Similarly, we can define the associator ideal generated by
\[\psi(x,\psi(y,z))-\psi(\psi(x,y),z)\]
for $x,y,z\in \mathcal{F}$. Taking the quotient of $\mathcal{F}$ by this ideal yields the left adjoints of the lower right and upper left arrows of the diamond.
\end{proof}
\begin{remark}
    We will denote the left adjoints of the forgetful functors $\mathcal{C}\to \operatorname{QCoh}(S)$ by $F_{\mathcal{C}}$. Let $n\in \N$. We will throughout the paper use the following standard notation
    \begin{align*}
        &\oo_S\left\{x_1,\dots, x_n\right\}:=F_{\operatorname{Qlg}}\left((\oo_S^n)^\vee\right), \\
        &\oo_S\left\{x_1,\dots, x_n\right\}_c:=F_{\operatorname{QClg}}\left((\oo_S^n)^\vee\right), \\
        &\oo_S\langle x_1,\dots, x_n\rangle:=F_{\operatorname{QAlg}}\left((\oo_S^n)^\vee\right) \text{ and}\\
        &\oo_S[x_1,\dots, x_n]:=F_{\operatorname{QCAlg}}\left((\oo_S^n)^\vee\right)
    \end{align*}
for the various free algebras on $n$ generators. Notice that we take the dual to keep the convention of a covariant equivalence between vector bundles and locally free sheaves of finite rank.
\end{remark}

\begin{lemma}  \label{algebra determined from epi from free}
    Let $S\in \operatorname{Sch}_{\C}$ and let $\mathcal{F},\mathcal{G}$ be quasi-coherent sheaves. Let $\psi, s$ be the multiplication rule and unit of $F_{\operatorname{Qlg}}(\mathcal{F})$. The following data is equivalent. 
    \begin{enumerate}
        \item A map $f:\mathcal{F}\to \mathcal{G}$ and a unital non-associative algebra structure $(\varphi,t)$ on $\mathcal{G}$ such that the induced map $\pi: F_{\operatorname{Qlg}}(\mathcal{F})\to \mathcal{G}$ is an epimorphism.
        \item An epimorphism $\pi: F_{\operatorname{Qlg}}(\mathcal{F})\to \mathcal{G}$ of quasi-coherent sheaves such that the kernel is a two-sided ideal.
    \end{enumerate}
\end{lemma}
\begin{proof}
    By adjunction and \cref{unique algebra structure surjection} both types of data exactly describes a surjection $\pi$ of unital non-associative algebras.
\end{proof}
\begin{lemma} \label{Ass or comm checked by factorization}
    Let $S\in \operatorname{Sch}_{\C}$ and let $\mathcal{F},\mathcal{G}$ be quasi-coherent sheaves. Let $(\varphi, t)$ be a unital non-associative algebra structure on $\mathcal{G}$ and let $\pi: F_{\operatorname{Qlg}}(\mathcal{F})\to \mathcal{G}$ be an epimorphism of unital non-associative algebras. Then $\mathcal{G}$ is associative (resp. commutative) if and only if $\pi$ factors through $F_{\operatorname{Qlg}}(\mathcal{F})\to F_{\operatorname{QAlg}}(\mathcal{F})$ (resp. $F_{\operatorname{Qlg}}(\mathcal{F})\to F_{\operatorname{QClg}}(\mathcal{F})$).
\end{lemma}
\begin{proof}
    If $\mathcal{G}$ is associative (resp. commutative) then we get such a factorization since $F_{\operatorname{QAlg}}$ (resp. $F_{\operatorname{QClg}}$) is the composite of $F_{\operatorname{Qlg}}$ with the left adjoint of the inclusion $\operatorname{QAlg}(S)\to \operatorname{Qlg}(S)$ (reps. $\operatorname{QClg}(S)\to \operatorname{Qlg}(S)$). Conversely, if such a factorization exists associativity or commutativity of $\mathcal{G}$ will then directly follow from \cref{unique algebra structure surjection}
\end{proof}
\subsection{Obstruction theories and virtual classes and pullbacks}
Let $X/\C$ be a scheme and let 
\[\mathbb{L}_{X}=\tau_{\geq -1}L_{X}^\bullet\in D^{[-1,0]}(X)\]
be the truncation of the full cotangent complex $L_{X}^\bullet$ as defined in \cite{Illusie1971}. A perfect obstruction theory of $X\to Y$ as defined in \cite{BehrendFantechi} is a map
\[\phi:\mathbb{E}\to \mathbb{L}_{X}\]
in $D^{[-1,0]}(X)$, where $\mathbb{E}$ is a perfect complex of perfect amplitude $[-1,0]$ and where $h^0(\phi)$ is an isomorphism and $h^{-1}(\phi)$ is surjective. If we have 
\[\mathbb{E}=[E^{-1}\to E^0]\]
then a perfect obstruction theory gives rise to a cone 
\[\mathfrak{C}\hookrightarrow E_1=(E^{-1})^\vee\]
which in term gives rise to a virtual fundamental class 
\[[X]^{\operatorname{vir}}=0^*[\mathfrak{C}]\in A_{\operatorname{vd}}(X)\]
where $0:X\to E_1$ is the zero section and $0^*$ its Gysin. Here, 
\[\operatorname{vd}=\chi(\mathbb{E})=\operatorname{rk}E^0-\operatorname{rk}E^{-1}.\]
Similarly, for $f: X\to Y$ a map of schemes a relative perfect obstruction theory is a map 
\[\phi: \mathbb{E}\to \mathbb{L}_{X/Y}=\tau_{\geq -1} L_{X/Y}^\bullet\]
in $D^{[-1,0]}(X)$ where $\mathbb{E}$ and $\phi$ satisfies the same definitions as above. It is shown in \cite{virtpull} that $\phi$ defines a virtual pullback  
\[f^*=f_{\mathbb{E}}^*: A_*(Y)\to A_{*+\chi(\mathbb{E})}(X).\]
If $Y=S$ is a smooth (or pure dimensional) one may define a virtual class by
\[[X]^{\operatorname{vir}}=f^*[S].\]
This recovers the usual definition for $S=\spec \C$. It is further shown that if we have a compatible triple i.e. a distinguished triangle 
\[f^*\mathbb{E}_{Y/S}\to \mathbb{E}_{X/S}\to \mathbb{E}_{X/Y}\to f^*\mathbb{E}_{Y/S}[1]\]
of perfect obstruction theories with a map to the canonical distinguished triangle 
\[f^*\mathbb{L}_{Y/S}\to \mathbb{L}_{X/S}\to \mathbb{L}_{X/Y}\to f^*\mathbb{L}_{Y/S}[1]\]
then the corresponding virtual pullbacks are functorial and so in particular
\[f^*[Y]^{\operatorname{vir}}=[X]^{\operatorname{vir}}.\]
In this paper, we are interested in the following simple case.
\begin{exmp} \label{Basic perfect obstruction theory example}
    Suppose $\pi: Y\to S$ is a smooth morphism over a smooth base scheme $S$. Then
    \[\mathbb{L}_{Y/S}\cong \Omega_{Y/S}\]
    and $\phi=\operatorname{id}_{\mathbb{L}_{X/Y}}$ defines a perfect obstruction theory with virtual pullback being the usual flat pullback $\pi^*$. Note that 
    \[[Y]^{\operatorname{vir}}=\pi^*[S]=[Y].\]
    If $i: X\hookrightarrow Y$ is a closed immersion with ideal sheaf $I$ then 
    \[\mathbb{L}_{X/Y}=[I/I^2\to 0].\]
    If we can find a vector bundle
    \[E=\spec_Y \sym^\bullet  \mathcal{E}^{\vee}\]
    with a section $s\in H^0(Y,\mathcal{E})$ such that $X=Z(s)$ then the map 
    \[\mathcal{E}^{\vee}|_X \xrightarrow{ s^{\vee}} I/I^2\]
    defines a relative perfect obstruction theory. Its virtual pullback is given by $0^!$ the refined Gysin map of the zero section $0:Y\to E$ along $s$. Similarly, we see that
    \[\mathbb{L}_{X/S}=[I/I^2\to \Omega_{Y/S}]\]
    and that the diagram
    \begin{center}
        \begin{tikzcd}
            \mathcal{E}^{\vee}|_X \arrow{d}{s^{\vee}} \arrow{r}{d\circ s^\vee} & \Omega_{Y/S}|_X \arrow[equal]{d} \\
            I/I^2 \arrow{r}{d} & \Omega_{Y/S}|_X
        \end{tikzcd}
    \end{center}
    defines a perfect obstruction theory on $X/S$. We can turn the above theories into a compatible triple and so the virtual pullback of $X/S$ is simply $0^!\pi^*$. In particular, 
    \[[X]^{\operatorname{vir}}=0^![Y].\]
\end{exmp}
\subsection{Torus equivariant obstruction theories} \label{sec::torus}
Let $T=(\C^*)^m$ be an algebraic torus and let $\hat{T}=\operatorname{Hom}(T,\C^*)\cong \Z^m$ its characters. Letting $T=\{(t_1,\dots, t_m)\}$ we get a ring isomorphism 
\begin{center}
    \begin{tikzcd}
        K_0^T(\operatorname{pt}) \arrow{r}{\operatorname{Tr}} &\Z[t^{\mu} \: |\: \mu\in \hat{T}]\cong \Z[t_1^{\pm 1}, \dots, t_m^{\pm 1}]
    \end{tikzcd}
\end{center}
sending a $T$-representation $V$ to the character $t\mapsto \operatorname{Tr}_V(t)$. We will therefore sometimes abuse notation and denote an element of $K_0^T(\operatorname{pt})$ simply by its trace. Similarly, for the equivariant Chow cohomology group we obtain a graded ring isomorphism
\begin{center}
    \begin{tikzcd}
      \sym_{\Z}^{\bullet} [\hat{T}] \arrow{r} &\Z[s_1,\cdots ,s_m] \cong A_T^*(\operatorname{pt})
    \end{tikzcd}
\end{center}
which maps a character to the first Chern class of the induced 1-dimensional representation. Here, $s_i$ is the first Chern class of the 1-dimensional representation of weight $t_i$. If $X\to Y$ is a $T$-equivariant map of $T$-schemes equipped with a $T$-equivariant perfect obstruction theory $\mathbb{E}\to \mathbb{L}_{X/Y}$, then the virtual pullback extends to a map of $T$-equivariant Chow groups. In particular, for $Y=\spec \C$ we obtain a $T$-equivariant virtual class
\[[X]^{\operatorname{vir}}\in A^T_{\chi(\mathbb{E})}(X).\]
The main way of calculating this class is through the virtual localization formula of \cite{localization}. Recall that a $T$-equivariant sheaf on a scheme with trivial $T$-action splits into eigensheaves and therefore also into a moving and fixed part.
\begin{prop}
    \cite{localization} \label{graber localization}Let $X/\C$ be a $T$-equivariant scheme with a $T$-equivariant perfect obstruction theory $\mathbb{E}$. Suppose $X$ admits a $T$-equivariant embedding into a smooth $T$-scheme $Y$. Let $Y^T=\bigcup_{i}Y_i$ be the irreducible components of the fixed locus and set $X_i=Y_i\cap X$. Then the compostion 
    \[\mathbb{E}|_{X_i}^{\operatorname{fix}}\to \mathbb{L}_{X}|_{X_i}^{\operatorname{fix}}\to \mathbb{L}_{X_i}\]
    is a perfect obstruction theory equipping $X_i$ with a virtual class. If we set
    \[N_{X_i/X}^{\operatorname{vir}}=(\mathbb{E}|_{X_i}^{\operatorname{mov}})^{\vee}=[(E^0_{\operatorname{mov}})^\vee \to (E^{-1}_{\operatorname{mov}})^\vee]\]
    then the identity
    \[[X]^{\operatorname{vir}}=\iota_* \sum_{i} \frac{[X_i]^{\operatorname{vir}}}{e^T(N_{X_i/X}^{\operatorname{vir}})}\]
    holds true in the localized Chow ring $A_*^T(X)_{\operatorname{loc}}$ where we have inverted linear functions of $s_1,\dots, s_m$. Here
    \[e^T(N_{X_i/X}^{\operatorname{vir}})=\frac{e^T((E^0_{\operatorname{mov}})^\vee)}{e^T((E^{-1}_{\operatorname{mov}})^\vee)}\]
    is a fraction of $T$-equivariant Euler classes.
\end{prop}
\begin{remark}In the case where $X$ is non-proper with a proper fixed locus $X^T$ we will for $\alpha\in A_T^*(X)$ define virtual integrals via the localization formula above
\[\int_{[X]^{\operatorname{vir}}}\alpha:=\sum_i \int_{[X_i]^{\operatorname{vir}}}\frac{\alpha|_{X_i}}{e^T(N_{X_i/X}^{\operatorname{vir}})}.\]
\end{remark}
We will spend the rest of this section providing consequences and examples of virtual localization that we will need later. First we show that the localization formula is particularly easy to use if $X^T$ is zero-dimensional, reduced, and proper. Essentially $[x]^{\operatorname{vir}}$ for $x\in X^T$ is either $[x]$ or $0$ depending only on whether or not $\operatorname{rk}E^0_{\operatorname{fix}}=\operatorname{rk}E^{-1}_{\operatorname{fix}}$

\begin{lemma}
    Let $T=(\C^*)^n$ and let $\varphi: \mathbb{E}=[E^{-1}\to E^0]\to \mathbb{L}_{\operatorname{pt}}$ be a $T$-equivariant perfect obstruction on $\operatorname{pt}=\spec \C$. If $\mathbb{E}^{fix}=\mathbb{E}$ and 
    \[\operatorname{vd}=\operatorname{rk}E^0-\operatorname{rk}E^{-1}\neq0\]
    then $[\operatorname{pt}]^{\operatorname{vir}}=0$. If $\operatorname{vd}=0$ then $[\operatorname{pt}]^{\operatorname{vir}}=[\operatorname{pt}]$.
\end{lemma}
\begin{proof}
Clearly $\mathbb{L}_{\operatorname{pt}}=0$ and so any $T$-equivariant perfect obstruction theory is of the form $\mathbb{E}=[V\to 0]$ for a $V$ a finite dimensional $T$-representation. If $\mathbb{E}=\mathbb{E}^{fix}$ then $V$ is a trivial representation. The intrinsic normal cone of $\operatorname{pt}$ is $\operatorname{pt}$ and so the virtual fundamental class is 
\[0^*[\operatorname{pt}]\]
where $0:\operatorname{pt}\to V$ is the zero section and $0^*$ its Gysin. Since 
\[\operatorname{vd}=-\operatorname{dim}V\]
the result now clearly follows.
\end{proof}
\begin{exmp} \label{fixed obstruction theory isolated points}
    Suppose we have a smooth $T$-scheme $Y$ with a $T$-equivariant vector bundle $E=\spec_Y\sym^\bullet \mathcal{E}^{\vee}$ and section $s\in H^0(Y,\mathcal
    {E})$. Then 
    \[X=Z(s)\]
    can be equipped with a $T$-equivariant perfect obstruction theory as in \cref{Basic perfect obstruction theory example}. In the case where $X^T$ is zero dimensional, reduced, and proper we then get
    \[[X]^{\operatorname{vir}}=\sum_{x}[x]\cap \frac{e^T(\mathcal{E}_x^{\operatorname{mov}})}{e^T(T_{Y,x}^{\operatorname{mov}})}\]
    where the sum is over all fixed points $x\in X$ such that $\mathcal{E}^{\operatorname{fix}}_x=T_{Y,x}^{\operatorname{fix}}$.
\end{exmp}
Next, we show a variation of the localization formula, allowing us to work more flexibly with different subschemes of the fixed-point locus.
\begin{prop}  \label{general localization}
    Let $X$ and $Y$ be $T$-schemes with $T$-equivariant perfect obstruction theories $\mathbb{E}_X$ and $\mathbb{E}_Y$. Assume we have a $T$-equivariant map 
    \[g:X\to Y\]
    equipped with a relative perfect obstruction theory $\mathbb{E}_g$ compatible with $\mathbb{E}_X$ and $\mathbb{E}_Y$. Let $Y^T=\bigcup_iY_i$ as in \cref{graber localization}. Let $j: g^{-1}(Y^T)\hookrightarrow X$ be the inclusion. Then
    \[[X]^{\operatorname{vir}}=j_*\sum_{i}\frac{g^![Y_i]^{\operatorname{vir}}}{e^T(g^*N_{Y_i/Y}^{\operatorname{vir}})} \]
    where $g^!: A_*^T(Y^T)\to A_*^T(X^T)$ is the (refined) virtual pullback associated to $\mathbb{E}_g$.
\end{prop}
\begin{proof}
    Compatibility ensures that 
    \[[X]^{\operatorname{vir}}=g^![Y].\]
    Virtual localization on $Y$ shows 
    \[[Y]=\iota_*\sum_i\frac{[Y_i]^{\operatorname{vir}}}{e^T(N_{Y_i/Y}^{\operatorname{vir}})} \]
    for $\iota: Y^T\hookrightarrow Y$. Since virtual pullback is compatible with pushforward, we see that 
    \[g^!\iota_*=j_*g^!\]
    thus, finishing the proof. 
\end{proof}
\begin{exmp}
    Suppose in the above that $Y$ is smooth and equipped with the perfect obstruction theory $\mathbb{E}_Y=\mathbb{L}_Y=\Omega_Y$. Then the formula above is simply
      \[[X]^{\operatorname{vir}}=j_*\sum_{i}\frac{g^![Y_i]}{e^T(g^*N_{Y_i/Y})}. \]
\end{exmp}
We end this section by showing how virtual integrals can in some cases be calculated by another virtual integral on a smaller locus. 
\begin{prop} \label{relative virtual integrals}
    Let $X$ be a $T$-scheme. Let $E$ be a $T$-equivariant vector bundle over $X$, $s\in H^0(X,E)^T$ and set 
    \[j:Z=Z(s)\hookrightarrow X.\]
    Suppose we have perfect obstruction theories $\mathbb{E}_Z$ and $\mathbb{E}_X$ compatible with the relative perfect obstruction theory $\mathbb{E}_{Z/X}$ induced by $E$. Let $j':Z^T\hookrightarrow X^T$ be the induced inclusion and suppose it is regular of codimension 0 (i.e., also open). Then for
    \[\alpha\in A^*_{Z,T}(X)\]
    a local Chow cohomology class supported on $Z$ we have
\[\int_{[X]^{\operatorname{vir}}}\alpha=\int_{[Z]^{\operatorname{vir}}} \frac{j^*\alpha}{e^T(j^*E)}\]
    If it further holds that $Z^T=X^T$
    then 
    \[[X]^{\operatorname{vir}}=\frac{j_*[Z]^{\operatorname{vir}}}{e^T(E)}\]
    in the localized Chow group of $X$. In particular, the above formula holds for all $\alpha\in A_T^*(X)$.
\end{prop}
\begin{proof}
Form the pullback diagram
\begin{center}
    \begin{tikzcd}
        Z^T \arrow[hook]{r}{j'} \arrow[hook]{d}{\iota'} & X^T \arrow[hook]{d}{\iota} \\
        Z \arrow[hook]{r}{j} & X
    \end{tikzcd}
\end{center}
and let $\gamma'\in A_*^T(Z^T)$ and $\gamma\in A^T_*(X^T)$ be the unique classes such that $\iota'_*\gamma'=[Z]^{\operatorname{vir}}$ and $\iota_*\gamma=[X]^{\operatorname{vir}}$. If we let
\[j^!:A_*^T(X^T)\to A_*^T(Z^T)\]
be the refined virtual pullback of $j$, we see from compatibility with pushforward that
\[j^!\gamma=\gamma'.\]
Note further that the virtual pullback is the refined Gysin of $0:X\to E$ and that $j'$ is regular with excess bundle $E|_{Z^T}$. It follows by the excess intersection formula that
\[j^!(\gamma)=e^T(E|_{Z^T})\cap j'^*\gamma.\]
Let now $p:X^T\to \operatorname{pt}$. By definition
\[\int_{[X]^{\operatorname{vir}}}\alpha =p_*(\iota^*\alpha\cap \gamma).\]
On the other hand 
\[\int_{[Z]^{\operatorname{vir}}} \frac{j^*\alpha}{e^T(j^*E)}=p_*j'_*\left(\frac{j'^*\iota^*\alpha}{e^T(E|_{Z^T})}\cap \gamma'\right)\]
and using that $\gamma'=j^!\gamma=e^T(E|_{Z^T})\cap j'^*\gamma$ this equals
\[p_*j'_*(j'^*\iota^*\alpha\cap j'^*\gamma)=p_*j'_*j'^*(\iota_*\alpha\cap \gamma).\]
By definition of local Chow cohomology we can find a bivariant class $c\in A^*(Z\hookrightarrow X)$ such that $\alpha=j_*c$ and in particular $\iota^*\alpha=\iota^* j_*c=j'_*\iota^*c$. The integral now equals
\[p_*j'_*j'^*(j'_*\iota^*c\cap \gamma)=p_*j'_*j'^*j'_*(\iota^* c\cdot \gamma)\]
and by the self-intersection formula we have $j'^*j'_*=\operatorname{id}$ so that we get
\[p_*j'_*(\iota^* c\cdot \gamma)=p_*(j'_*\iota^*c\cap \gamma)=p_*(\iota^*\alpha \cap \gamma)\]
as desired. If we assume $j'=\operatorname{id}$ we can find a unique $\sigma$ such that $j'_*\sigma=\gamma$. But then 
\[e^T(E|_{Z_T})\cap \sigma=j^!j'_*\sigma=j^!\gamma=\gamma'\]
showing that
\[\sigma=\frac{\gamma'}{e^T(E|_{Z_T})}\]
as desired.
\end{proof}
\begin{exmp} \label{double section zero locus}
   We are primarily interested in the following example. Let $Y$ be a smooth $T$-scheme with $T$-equivariant vector bundles $E,F$ and $s\in H^0(Y,E)^T$ and $u\in H^0(Y,F)^T$. Suppose that $W=Z(s)$ is smooth. Set $X=Z(u)$ and equip it with a perfect obstruction theory coming from $F$ and $Y$. Then $Z=W\cap X=Z(s|_X)$ is cut out of $X$ by $s|_X$ and so we have a relative perfect obstruction theory on $j:Z\hookrightarrow X$. Equipping $Z$ with a perfect obstruction theory coming from $F$ and $W$ we obtain a compatible triple. If we further assume that $X^T$ is zero dimensional, reduced, and proper and that $Z^T$ is non-empty then $j':Z^T\hookrightarrow X^T$ is regular of codimension 0.
\end{exmp}
\section{Construction of the non-associative Hilbert scheme}
\label{sec:NA-nested-Hilb}
In this section we construct a moduli functor and prove that it is representable by a smooth scheme $\operatorname{naNHilb}^{\underline{d}}(\A^n)$ which we call the non-associative nested Hilbert scheme of $\underline{d}$ points. The proof will be in two steps. First we exhibit the functor as the quotient stack of a smooth scheme by a linear algebraic group acting freely. This proves that the functor is representable by a smooth algebraic space. Next we explicitly construct a cover of Zariski open subfunctors which we prove are representable. This proves that the algebraic space must in fact be a scheme.
\subsection{The definition}
Fix a dimension vector $\underline{d}=(d_0,\dots, d_r)\in \N^{r+1}$, $r\geq 0$. We recall that the nested Hilbert scheme $\operatorname{NHilb}^{\underline{d}}(\A^n)$ parameterize nesting of closed subschemes
\[\xi_{1}\hookrightarrow \xi_2\hookrightarrow\cdots \hookrightarrow \xi_{r+1}\hookrightarrow \A^n\]
such that $Z_{k}$ is finite of length $d_0+d_1+\cdots+d_{k-1}$. Equivalently it parameterize flags of ideals
\[I_{r+1}\subset I_r \subset \cdots \subset I_1\subset \C[x_1,\dots, x_n]\]
such that $\C[x_1,\dots, x_n]/I_k$ has an underlying vector space of dimension $d_0+\cdots +d_{k-1}$. Equivalent to that, we see that it parameterize a surjective map of $\C$-algebras 
\[\C[x_1,\dots, x_n]\twoheadrightarrow \C[x_1,\dots, x_n]/I_{r+1}=:N\]
and a flag of ideals on the codomain 
\[N=N_0\supset N_1\supset \cdots \supset N_{r+1}=0\]
 satisfying that $N_{k}/N_{k+1}$ is a $d_k$ dimensional vector space. The idea behind the non-associative Hilbert scheme is simply to take this last description and relax the condition of associativity. That is, if we denote by $\C\left\{ x_1,\dots, x_n\right\}_c$ the free unital and commutative (but not associative) $\C$-algebra on $n$ generators, the non-associative Hilbert scheme parameterizes surjections of unital, commutative (but not necessarily associative) algebras
 \[\C\left\{x_1,\dots, x_n\right\}_c\twoheadrightarrow N\]
 and a flag on ideals 
 \[N=N_0\supset N_1\supset \cdots \supset N_{r+1}=0\]
 such that $N_k/N_{k+1}$ is $d_k$ dimensional as a $\C$-vector space. To construct this space, we simply specify its functor of points. We recall that for the Hilbert scheme going from a $\C$-valued point to an $S$-valued we simply replace $\C$ with $\oo_S$ and vector space of a given dimension with vector bundle of a given rank in the above definitions. Doing the same for the non-associative Hilbert scheme, we arrive at the following definition

 \begin{defi}
     Let $\underline{d}=(d_0,\dots, d_r)\in \N^{r+1}$ be a dimension vector. We define a functor 
     \[na\mathcal{H}ilb^{n,\underline{d}}:\operatorname{Sch}_{/\C}^{\operatorname{op}}\to \operatorname{Set}\]
     as follows. For $S\in \operatorname{Sch}_{/\C}$ we set 
     \[na\mathcal{H}ilb^{n,\underline{d}}(S)=\left\{\pi:\oo_S\left\{x_1,\dots, x_n\right\}_c\twoheadrightarrow (\mathcal{F}_\bullet,\psi,s)\right\}\big{/}\operatorname{iso}.\]
     Here 
     \[\mathcal{F}_\bullet =(\mathcal{F}_0\supset \mathcal{F}_1 \supset \cdots \supset  \mathcal{F}_{r+1}=0)\]
     is a flag of locally free sheaves on $S$ such that $\mathcal{F}_k/\mathcal{F}_{k+1}$ is locally free of rank $d_k$, $\mathcal{F}_0$ is equipped with a unital commutative non-associative $\oo_S$-algebra structure $(\psi,s)$ which satisfies \[\psi(\mathcal{F}_0\otimes_S \mathcal{F}_k)\subset \mathcal{F}_k.\]
     The surjection $\pi$ is a morphism of unital non-associative $\oo_S$-algebras. Two such 
     \[\left(\mathcal{F}_\bullet, \psi, s,\pi\right) \quad \text{and} \quad \left(\mathcal{G}_\bullet, \psi', s',\pi'\right) \] are considered isomorphic if there exist an isomorphism of unital non-associative $\oo_S$-algebras $\varphi:\mathcal{F}_0\xrightarrow{\sim}\mathcal{G}_0$ preserving the flags such that $\varphi\pi=\pi'$.
 \end{defi}

 \begin{remark}
     We note that by \cref{algebra determined from epi from free} a point $[\mathcal{F}_\bullet, \psi, s,\pi]$ is uniquely determined by either $[\mathcal{F}_\bullet, \pi]$ or $[\mathcal{F}_{\bullet},\psi,s,f]$ where $f$ is the restriction of $\pi$ to $(\oo_S^n)^\vee$. Depending on the situation, we will use either of the three to describe a point on the non-associative Hilbert scheme. When using the last formulation, we will usually use the notation $[\mathcal{F}_\bullet, \psi_2,s,\psi_1]$, i.e. where $\psi_2=\psi$ and $\psi_1=f$ which agrees with the notation of \cite{KazarianNAHilb}.
 \end{remark}

\subsection{Smoothness}
We now exhibit the above moduli functor as the functor of points of a quotient stack. The idea is to fix a flag of $\C$-vector spaces $N_\bullet$. We then construct a space parameterizing the different ways we can equip $N_\bullet$ with a unital commutative non-associative algebra structure and a surjection
\[\C\{x_1,\dots, x_n\}_c\twoheadrightarrow N_\bullet\]
of unital non-associative algebras. Then we quotient out the choice of flag, by taking the group quotient of the group of flag preserving automorphisms of $N_\bullet$. This data is completely determined by the non-associative algebra structure $(\psi_2,v)$ on $N_\bullet$ along with a map of $\C$-vector spaces 
\[\psi_1:(\C^n)^\vee \to N_\bullet.\]
The data we want to parameterize is therefore $(\psi_1,\psi_2,v)$, but we need to find a condition so that the associated map of algebras is surjective. This leads to the following stability condition, which we define globally
\begin{defi}
    Let $S\in \operatorname{Sch}_{/\C}$ and let $(\mathcal{F},\psi_2,s)$ be a unital non-associative $\oo_S$ algebra. Let 
    \[\psi_1: (\oo_S)^{\vee}\to \mathcal{F}\]
    be a map of quasi-coherent sheaves. We say that  $(\mathcal{F},\psi_1,\psi_2,s)$ is stable if the only non-associative $\oo_S$-subalgebra of $\mathcal{F}$ for which $\psi_1$ factors through is $\mathcal{F}$ itself. 
\end{defi}
\begin{lemma} \label{alternative stability}
Let $S\in \operatorname{Sch}_{/\C}$ and let $(\mathcal{F},\psi_2,s)$ be a unital non-associative $\oo_S$ algebra and let 
    \[\psi_1: (\oo_S)^{\vee}\to \mathcal{F}\]
    be a map of quasi-coherent sheaves. Then $(\mathcal{F},\psi_1,\psi_2,s)$ is stable if and only if the associated map 
    \[\oo_S\{x_1,\dots, x_n\}\to \mathcal{F}\]
    is surjecive.
\end{lemma}
\begin{proof}
    This is clear as the image of the map 
    \[\oo_S\{x_1,\dots, x_n\}\to \mathcal{F}\]
    can be characterized as the smallest non-associative $\oo_S$-subalgebra containing $\psi_1$.
\end{proof}
We make the following definition. 
\begin{defi} \label{definition of prequotient space}
    Let $\underline{d}=(d_0,\cdots ,d_r)\in \N^{r+1}$ be a dimension vector. Fix a flag 
    \[N_0\supset N_1\supset \cdots \supset N_{r+1}=0\]
    of $\C$-vector spaces with $\operatorname{dim}_{\C}N_k/N_{k+1}=d_k$. We define
    \[M_{n,\underline{d}}\subset \Hom_{\C}((\C^n)^\vee, N_0)\oplus \Hom_{\C}(\sym^2N_0,N_0)\oplus N_0\]
    to be the smooth locally closed subscheme consisting of tuples $(\psi_1,\psi_2,v)$ satisfying the following properties.
    \begin{enumerate}
        \item For all $0\leq k\leq r+1$ it holds that $\psi_2(N_0\otimes_{\C} N_k)\subset N_k$.
        \item For all $x\in N_0$ we have $\psi_2(v,x)=x$. 
        \item The tuple $(N_0,\psi_1,\psi_2,v)$ is stable.
    \end{enumerate}
\end{defi}
We note that condition (1) restricts us to a subspace, condition (2) restricts us to an affine subspace and conditions (3) restrict us to a Zariski open set so that $M_{n,\underline{d}}$ is a smooth locally closed subscheme. Next, we introduce the group action. 
\begin{defi}
    Let $\underline{d}\in \N^{r+1}$ and let $N_\bullet$ be as in \cref{definition of prequotient space}. We let 
    \[P_{\underline{d}}\subset \operatorname{Aut}_{\C}(N_0)\]
    be the parabolic group of linear automorphisms of $N_0$ preserving the flag. We let $P_{\underline{d}}$ act on 
    \[\operatorname{Hom}_{\C}((\C^n)^\vee, N_0)\oplus \Hom( \sym^2 N_0,N_0)\oplus N_0\]
    by 
    \[g.(\psi_1,\psi_2,v)=(g\psi_1,\psi_2^{g},gv)\]
    where 
    \[\psi_2^{g}(x,y)=g\psi_2(g^{-1}x,g^{-1}y).\]
\end{defi}
\begin{lemma}
    The $P_{\underline{d}}$-action restricts to a free action on $M_{n,\underline{d}}$
\end{lemma}
\begin{proof}
    Let $(\psi_1,\psi_2,v)\in M_{n,\underline{d}}$ and let $g\in P_{\underline{d}}$. We check that $(g\psi_1,\psi_2^g ,gv)$ still satisfies the three conditions.
    \begin{enumerate}
        \item Since $g$ is flag preserving it is clear that
        \[\psi_2^g(N_0\otimes_{\C}N_k)\subset N_k\]
    for all $0\leq k\leq r+1$.
    \item For $x\in N_0$ we see that
    \begin{align*}\psi_2^g(gv,x)&=g\psi_2(g^{-1}gv,g^{-1}x) \\
    &=g\psi_2(v,g^{-1}x)\\
    &=gg^{-1}x\\
    &=x.\end{align*}
    \item Suppose $M\subset N_0$ is a subspace containing $v$ and $\im g\psi_1$ and which is closed under $\psi_2^g$. Then $g^{-1}M$ contains $v$ and $\im \psi_1$. For $g^{-1}x,g^{-1}y\in g^{-1}M$ we see that 
    \begin{align*}
    \psi_2(g^{-1}x,g^{-1}y)&=g^{-1}g\psi_2(g^{-1}x,g^{-1}y) \\
        &=g^{-1}\psi_2^g(x,y)
    \end{align*}
    which is in $g^{-1}M$ since $\psi_2^g(x,y)\in M$. We conclude by stability of $(\psi_1,\psi_2,v)$ that $g^{-1}M=N_0$ and thus $M=N_0$.
    \end{enumerate}
This proves that the $P_{\underline{d}}$-action restricts to an action on $M_{n,\underline{d}}$. Next, we assume that $(\psi_1,\psi_2,v)\in M_{n,\underline{d}}$ and $g\in P_{\underline{d}}$ is chosen so that $(\psi_1,\psi_2,v)=(g\psi_1,\psi_2^g,gv)$. Let $M\subset N_0$ be the subspace of $x\in N_0$ where $gx=x$. By definition $v\in M$ and $\im \psi_1\in M$. If $x,y\in M$ then 
\begin{align*}
    g\psi_2(x,y)&=g\psi_2(g^{-1}x,g^{-1}y) \\
    &=\psi_2^g(x,y)\\
    &=\psi_2(x,y)
\end{align*}
so that $M$ is closed under $\psi_2$. We conclude by stability that $M=N_0$ so that $g=\operatorname{id}_{N_0}$. This shows that $P_{\underline{d}}$ acts freely on $M_{n,\underline{d}}$.
\end{proof}
\begin{defi}
    Let $\underline{d}\in \N^{r+1}$ and $n\in \N$. We define the non-associative nested Hilbert scheme of $\underline{d}$ points in $\A^n$ to be the quotient stack
    \[\operatorname{naNHilb}^{\underline{d}}(\A^n):=[M_{n,\underline{d}}/P_{\underline{d}}].\]
    In the case where $r=0$ so $d\in \N$ we use the notation
    \[\operatorname{naHilb}^{d}(\A^n):=[M_{n,d}/\operatorname{GL}_d]\]
    which we will refer to as the non-associative Hilbert scheme of $d$ points in $\A^n$.
\end{defi}
\begin{corollary}
The stack $\operatorname{naNHilb}^{\underline{d}}(\A^n)$ is a smooth algebraic space.    
\end{corollary}
\begin{proof}
    This follows since it is the quotient of a smooth algebraic group acting freely on a smooth scheme.
\end{proof}
We are now ready to prove that $\operatorname{naNHilb}^{\underline{d}}(\A^n)$ represents the moduli functor $na\mathcal{H}ilb^{n,\underline{d}}$. We first introduce the functor which $M_{n,\underline{d}}$ represents.
\begin{defi}
    Let $\underline{d}\in \N^{r+1}$ and $n\in \N$ and fix a reference flag of $\C$-vector spaces
    \[N_\bullet=(N_0\supset \cdots \supset N_{r+1})\]
    as in \cref{definition of prequotient space}. We define
    \[\overline{na\mathcal{H}ilb}^{n,\underline{d}}: \operatorname{Sch}^{\operatorname{op}}_{/\C} \to \operatorname{Set}\]
    as follows. For $S\in \operatorname{Sch}_{\C}$ we set 
    \[\overline{na\mathcal{H}ilb}^{n,\underline{d}}(S):=\left\{(\mathcal{F}_\bullet,\psi, s,\pi ,\beta)\right\}\big{/}\operatorname{iso}\]
    where $(\mathcal{F}_\bullet,\psi,s,\pi)\in na\mathcal{H}ilb^{n,\underline{d}}(S) $ and 
    \[\beta: N_\bullet \otimes_{\C}\oo_S\to \mathcal{F}_{\bullet}\]
    is an isomorphism of flags. An isomorphism of such tuples $(\mathcal{F}_\bullet,\psi,s,\pi,\beta)$ and $(\mathcal{G}_\bullet, \psi',s',\pi',\beta')$ is an isomorphism of flags $\varphi:\mathcal{F}_\bullet \to \mathcal{G}_\bullet$ of non-associative unital algebras which preserves the flag structures and satisfies $\varphi\pi=\pi'$ and $\varphi\beta=\beta'$. There is a natural transformation
    \begin{center}
        \begin{tikzcd}
            \overline{na\mathcal{H}ilb}^{n,\underline{d}}(S)\arrow{r} & na\mathcal{H}ilb^{n,\underline{d}}(S) \\[-20pt]
            \left[\mathcal{F}_\bullet, \psi,\pi,\beta\right] \arrow[mapsto]{r} & \left[\mathcal{F}_\bullet, \psi,\pi\right]
        \end{tikzcd}
    \end{center}
    which forgets $\beta$.
\end{defi}
\begin{prop}
    The forgetful natural transformation 
    \[\overline{na\mathcal{H}ilb}^{n,\underline{d}}\to na\mathcal{H}ilb^{n,\underline{d}} \]
    identifies with the functor of points of the quotient map 
    \[M_{n,\underline{d}}\to [M_{n,\underline{d}}/P_{\underline{d}}]=\operatorname{naNHilb}^{\underline{d}}(\A^n).\]
\end{prop}
\begin{proof}
     For $d=d_0+\cdots+d_r$, we let $v_0,\dots, v_{d}$ be a basis of $N$ such that
    \[N_i=\operatorname{span}(v_{d_0+\cdots+d_{i-1}+1},v_{d_0+\cdots+d_{i-1}+2},\dots, v_{d})\]
for all $i=0,\dots, r$.
Let $X$ be the affine scheme corresponding to the vector space 
\[\Hom_{\C}((\C^n)^\vee, N_0)\oplus \operatorname{Hom}_{\C}(\sym^2N_0,N_0)\oplus N_0.\]
Using the basis, we identify $X$ with
\[\spec \C\left[(x_l)_t, a_{i,j}^k,w_m \right]_{l,t,i,j,k,m}.\]
Here, for $(\psi_1,\psi_2,w)\in X$ we have $\psi_1(e_l^\vee)=x_l$, $w_i$ the coordinates of $w$ and
\[\psi_2(v_i,v_j)=\sum_{k=1}^{d}a_{i,j}^kv_k.\]
Consider the flag of locally free sheaves $V_\bullet$ on $X$ given by 
\[V_\bullet :=N_\bullet \otimes_{\C}\oo_{X}\]
with $\beta_0=\id_{V_\bullet}$.
We have a distinguished section $\overline{\mathfrak{s}}\in H^0(X,V_0)$ given by
\[\sum_{m=1}^{d}v_m\otimes w_m\]
and we can similarly construct
\begin{center}
    \begin{tikzcd}
        \overline{\Psi}_1: (\oo_X^n)^\vee \arrow{r} & V_0 \\[-20pt]
        (e_l)^{\vee} \arrow[mapsto]{r} & \sum_{t=1}^{d} v_t\otimes (x_l)_t
    \end{tikzcd}
\end{center}
and 
\begin{center}
    \begin{tikzcd}
        \overline{\Psi}_2: \sym^2 V_0 \arrow{r} & V_0\\[-20pt]
        (v_i\otimes 1)\cdot (v_j\otimes 1) \arrow[mapsto]{r} & \sum_{k=1}^{d}v_k\otimes a_{i,j}^k
    \end{tikzcd}
\end{center}
which are both maps of sheaves. We restrict the sheaves and maps to the locally closed subscheme $M_{n,\underline{d}}$. By condition 1., 2. in the definition of  $M_{n,\underline{d}}$, it is clear that $(V_\bullet,\overline{\Psi}_1,\overline{\mathfrak{s}})$ is a unital commutative non-associative $\oo_{M_{n,\underline{d}}}$-algebra which satisfies 
\[\overline{\Psi}_2(V_0\otimes_{M_{n,\underline{d}}}V_k)\subset V_k\]
By the universal property of the free unital commutative non-associative algebra we obtain from $\overline{\Psi}_1$ a map 
\[\overline{\varpi}: \oo_{M_{n,\underline{d}}}\left\{x_1,\dots, x_n\right\}_c\to V_0 \]
of unital non-associative algebras which is surjective by condition 4. in the definition of $M_{n,\underline{d}}$. It follows that the tuple $(V_\bullet, \overline{\Psi}_2,\overline{\mathfrak{s}},\overline{\varpi},\beta_0)$ defines a $M_{n,\underline{d}}$-valued point of $\overline{na\mathcal{H}ilb}^{n,\underline{d}}$ giving us a natural transformation
\[M_{n,\underline{d}}\to\overline{na\mathcal{H}ilb}^{n,\underline{d}}. \]
We prove it is an isomorphism. Let $S\in \operatorname{Sch}_{\C}$ and let $[\mathcal{F}_\bullet, \psi,s,\pi,\beta]\in \overline{na\mathcal{H}ilb}^{n,\underline{d}}(S)$. From this data we obtain a section
\[\beta^{-1}s:\oo_S\to N_\bullet\otimes_{\C} \oo_S.\]
Let $\psi_1$ be the restriction of $\pi$ to $(\oo_S^n)^{\vee}$. By adjunction of $\beta^{-1}\psi(\beta\otimes \beta)$ and $\beta^{-1}\psi_1$ we similarly obtain sections
\[\oo_S \to \Hom_{\C}(\sym_{\C}^2N_\bullet, N_\bullet) \otimes_{\C}\oo_S\]
and 
\[\oo_S\to \Hom_{\C}((\C^n)^\vee, N_\bullet)\otimes_{\C}\oo_S.\]
Combined these three sections provides a map $f:S\to X$ and from the different assumptions on $[\mathcal{F}_\bullet, \psi,s,\pi,\beta]$ we see that $f$ factors through $M_{n,\underline{d}}.$ We note that by construction
\[f^*(V_\bullet, \overline{\Psi}_2,\overline{\mathfrak{s}},\overline{\varpi},\beta_0)=(\mathcal{F}_\bullet, \psi,s,\pi,\beta)\]
and one can directly check that this procedure provides an inverse to 
\[M_{n,\underline{d}}\to\overline{na\mathcal{H}ilb}^{n,\underline{d}}. \]
We now equip $V_\bullet=N_\bullet \otimes_{\C}\oo_{M_{n,\underline{d}}}$ with $P_{\underline{d}}$ equivariant structure by letting $P_{\underline{d}}$ act naturally on $N_\bullet$ (being a subgroup of its linear automorphisms). One can easily check that $\overline{\Psi}_2$, $\overline{\mathfrak{s}}$ and $\overline{\varpi}$ are equivariant wrt. to this structure. Letting 
\[q:M_{n,\underline{d}}\to [M_{n,\underline{d}}/P_{\underline{d}}]\]
be the quotient map, it follows that we have an induced flag of locally free sheaves
\[\mathcal{V}_\bullet\in \operatorname{QCoh}^{P_{n,\underline{d}}}(M_{n,\underline{d}})=\operatorname{QCoh}([M_{n,\underline{d}}/P_{\underline{d}}])\]
equipped with a section $\mathfrak{s}\in H^0([M_{n,\underline{d}}/P_{\underline{d}}],\mathcal{V}_0)$ and maps
\[\Psi_2:\sym^2 \mathcal{V}_0\to \mathcal{V}_0,\]
\[\varpi: \oo_{[M_{n,\underline{d}}/P_{\underline{d}}]}\left\{x_1,\dots, x_n\right\}_c\to \mathcal{V}_0\]
such that
\[q^*(\mathcal{V}_\bullet, \Psi_2,\mathfrak{s},\varpi)=(V_\bullet, \overline{\Psi}_2,\overline{\mathfrak{s}},\overline{\varpi}).\]
This defines a $[M_{n,\underline{d}}/P_{\underline{d}}]$ valued point of $na\mathcal{H}ilb^{n,\underline{d}}$ giving us a natural transformation
\[[M_{n,\underline{d}}/P_{\underline{d}}]\to na\mathcal{H}ilb^{n,\underline{d}}.\]
We explicitly define the inverse. Let $[\mathcal{F}_\bullet, \psi,s,\pi]\in na\mathcal{H}ilb^{n,\underline{d}}(S)$ and pick a cover of $(S_i)_{i\in I}$ of $S$ trivializing $\mathcal{F}_\bullet$ by explicit choices
\[\beta_i: N_\bullet\otimes_{\C} \oo_{S_i}\to \mathcal{F}_\bullet|_{S_i}\]
of isomorphisms of flags. For each $i$ the data $(\mathcal{F}_\bullet|_{S_i},\psi_i,s_i,\pi_i,\beta_i)$ determines a map 
\[f_i: S_i\to M_{n,\underline{d}}.\]
On overlaps $S_{j,i}$ the change from $\beta_i$ to $\beta_j$ determines a map 
\[g_{j,i}: S_{j,i}\to P_{\underline{d}}\]
satisfying $g_{j,i}\cdot f_i=f_j$. This means that the maps $q f_i$ and $qf_j$ agrees on $S_{j,i}$ so by gluing we get a map 
\[f:S\to [M_{n,\underline{d}}/P_{\underline{d}}].\]
By construction this satisfies that
\[f^*(\mathcal{V}_\bullet, \Psi_2,\mathfrak{s},\varpi)\simeq (\mathcal{F}_\bullet, \psi,s,\pi)\]
and one can now easily check that this is inverse to the above construction. 
\end{proof}
\begin{remark}
    By the moduli description we have that $\operatorname{naNHilb}^{\underline{d}}(\A^n)$ comes equipped with a universal non-associative algebra, a universal quotient map and a universal flag of ideals. This corresponds exactly to $(\mathcal{V}_\bullet,\Psi_2,\mathfrak{s},\varpi)$ as in the proof of the proposition above. We will denote the restriction of $\varpi$ to $(\oo^n)^\vee$ by $\Psi_1$.
\end{remark}

\subsection{Representability}
We prove that the smooth algebraic space $\operatorname{naNHilb}^{\underline{d}}(\A^n)$ is representable by a scheme. It will suffice to find a Zariski cover of schemes. We consider the following opens.
\begin{defi}
  Fix $\underline{d}\in \N^{r+1}$, $n\in \N$ and a flag $N_\bullet$ as in \cref{definition of prequotient space}. For any morphism 
  \[\gamma: N_0 \to \C\left\{x_1,\dots, x_n\right\}_{c}\]
  we define $na\mathcal{H}ilb^{n,\underline{d}}_\gamma\subset na\mathcal{H}ilb^{n,\underline{d}}$ to be the subfunctor which has $S$-valued points corresponding to $[\mathcal{F}_\bullet, \psi, s,\pi]$ such that
  \[\pi(1_{\oo_S}\otimes \gamma): \oo_S\otimes_{\C} N_0 \to \mathcal{F}_{0}\]
  is an isomorphism of $\oo_S$-modules.
\end{defi}
\begin{lemma}
    The subfunctor $na\mathcal{H}ilb^{n,\underline{d}}_\gamma\subset na\mathcal{H}ilb^{n,\underline{d}}$ is Zariski open. When $\gamma$ varies over all $\Hom_{\C}(N_0,\C\left\{x_1,\dots, x_n\right\}_c)$ this provides a cover of $na\mathcal{H}ilb^{n,\underline{d}}$. 
\end{lemma}
\begin{proof}
    For an $S$-valued point classified by $[\mathcal{F}_\bullet, \psi,s,\pi]$ the pullback along the inclusion of the subfunctor corresponds to the locus of $S$ where 
    \[\pi(1_{\oo_S}\otimes \gamma)\]
    is surjective. This is clearly a Zariski open locus. Furthermore, the map $\pi$ admits a section $\gamma$ Zariski locally on $S$, proving that the subfunctors provides a cover for varying $\gamma$.
\end{proof}
\begin{lemma}
    The functor $na\mathcal{H}ilb^{n,\underline{d}}_\gamma$ is representable by a scheme.
\end{lemma}
\begin{proof}
We let $\operatorname{Flag}_{\underline{d}}(N_0)$ be the flag variety of type $\underline{d}$ on $N_0$. It is a scheme having $S$-valued points corresponding to flags of locally free sheaves
\[0=\mathcal{G}_{r+1}\subset \cdots \subset \mathcal
{G}_0=N_0\otimes_{\C} \oo_S\]
such that $\mathcal{G}_{k}/\mathcal{G}_{k+1}$ is locally free of rank $d_k$. Let $d=d_0+\cdots d_r\in \N$. We consider the map 
\[na\mathcal{H}ilb_{\gamma}^{n,\underline{d}}\to \operatorname{Flag}_{\underline{d}}(N_0)\times \overline{na\mathcal{H}ilb}^{n,d}\cong \operatorname{Flag}_{\underline{d}}(N_0)\times M_{n,d}\]
which on $S$-valued points is given by 
\[[\mathcal{F}_{\bullet},\psi,s,\pi]\mapsto ([\mathcal{F}_\bullet],[\mathcal{F}_0,\psi,s,\pi,\pi(1_{\oo_S}\otimes \gamma)]).\]
We claim that it is a closed immersion thus finishing the proof. Indeed, let $([\mathcal{G}_\bullet],[\mathcal{F}_0, \psi, s,\pi,\beta])$ be an $S$-valued point of $\operatorname{Flag}_{\underline{d}}(N_0)\times M_{n,d}$. The pullback along the above map identifies with the locus of $S$ where $\pi(1_{\oo_S}\otimes \gamma)=\beta$ and where the map induced by $\psi$ and $\beta$ on $\mathcal{G}_0=N_0\otimes_{\C}\oo_S$ satisfies 
\[\psi(\mathcal{G}_0\otimes_S \mathcal{G}_{k})\subset \mathcal
G_{k}.\]
Both of these conditions are Zariski closed, as they can be described as the vanishing of sections of locally free sheaves.
\end{proof}

\begin{corollary}
    The algebraic space $\operatorname{naNHilb}^{\underline{d}}(\A^n)$ is representable by a scheme.
\end{corollary}
\begin{proof}
    This follows from the above since it admits a Zariski cover of open subschemes.
\end{proof}
\subsection{Embedding of the Hilbert scheme}
In this section we construct a map from the nested Hilbert scheme of points to the non-associative nested Hilbert scheme. We prove that the map is a closed immersion and identifies the Hilbert scheme of points with the zero locus of a section of a vector bundle on the non-associative Hilbert scheme. This equips the nested Hilbert scheme of points with a perfect obstruction theory. 
\begin{defi}
    We define 
    \[\iota: \operatorname{NHilb}^{\underline{d}}(\A^n)\to \operatorname{naNHilb}^{\underline{d}}(\A^n)\]
    via the functor of points. An $S$-valued point
    \[\Big{[}\mathcal{I}_{r+1}\subset \cdots \mathcal{I}_1\subset \oo_S[x_1,\dots, x_n] \Big{]}\]
    is mapped to $[\mathcal{F}_\bullet, \psi, s,\pi]$ where $(\mathcal{F}_0, \psi, s)$ is the commutative unital $\oo_S$-algebra given by $\oo_S[x_1,\dots, x_n]/\mathcal{I}_{r+1}$, the flag of ideals is defined by $\mathcal{F}_{k}=\mathcal{I}_k/\mathcal{I}_{r+1}$ and the quotient map is given by
    \[\pi: \oo_S \left\{x_1,\dots, x_n\right\}_c \twoheadrightarrow\oo_S [x_1,\dots, x_n] \twoheadrightarrow \mathcal{F}_0.\]
\end{defi}
\begin{defi}
    The associativity locus of $\operatorname{naHilb}^{\underline{d}}(\A^n)$ is the subfunctor which has $S$-valued points given by $[\mathcal{F}_\bullet, \psi, s,\pi]$ where $\psi$ is associative. By \cref{Ass or comm checked by factorization} it is the same as the locus where $\pi$ factors through $\oo_S[x_1,\dots, x_n]$. It follows that the locus is Zariski closed.
\end{defi}
 
\begin{lemma}
    The map $\iota$ is a closed immersion identifying $\operatorname{NHilb}^{\underline{d}}(\A^n)$ with the associativity locus.
\end{lemma} 

\begin{proof}
   It is clear by definition that $\iota$ maps the nested Hilbert scheme to this locus. Conversely, given an $S$-valued point $[\mathcal{F}_\bullet, \psi,s,\pi]$ of the associativity locus we let 
    \[\widetilde{\pi}: \oo_{S}[x_1,\dots, x_n]\twoheadrightarrow \mathcal{F}_0\]
    be the induced map and define an $S$-valued point on the nested Hilbert scheme by setting
    \[\mathcal{I}_k:=\widetilde{\pi}^{-1}(\mathcal{F}_k).\]
    It is easy to see that this operation is inverse to that of $\iota$.
\end{proof}
\subsection{The induced perfect obstruction theory}
We wish to prove that this locus can also be described as the zero locus of a section of a vector bundle. We note that $\operatorname{naHilb}^{\underline{d}}$ comes equipped with a universal flag $\mathcal{V}_\bullet$ of locally free sheaves and a universal multiplication rule $\Psi_2$. We note that the associativity conditions can now directly be checked by the vanishing of the associator section
\[\Psi_2(\operatorname{id}\otimes \Psi_2)-\Psi_2(\Psi_2\otimes\operatorname{id})\]
of $\mathcal{H}om(\mathcal{V}_0\otimes \mathcal{V}_0\otimes \mathcal{V}_0,\mathcal{V}_0)$. We notice that we can reduce the rank of the sheaf a bit. Firstly, we note that the symmetry of $\Psi_2$ implies that
\[\Psi_2(\Psi_2(x,y),x)-\Psi_2(x,\Psi_2(y,x))=\Psi_2(x,\Psi_2(y,x))-\Psi_2(x,\Psi_2(y,x))=0\]
so that the above section is alternating in the outer coordinates. Secondly, we see that
\[\Psi_2(\mathcal{V}_i\otimes V_j)\subset \mathcal{V}_{\operatorname{max}\{i,j\}}\]
so that the above trilinear map takes $\mathcal{V}_{i}\otimes \mathcal{V}_j\otimes \mathcal{V}_k$ to $\mathcal{V}_{\operatorname{max}\{i,j,k\}}$. At last, if we let $e=\mathfrak{s}(1)$ be the unit we see that
\begin{align*}
    \Psi_2(\Psi_2(e,y),z)-\Psi_2(e,\Psi_2(y,z))&=\Psi_2(y,z)-\Psi_2(y,z)=0 \\
    \Psi_2(\Psi_2(x,e),z)-\Psi_2(x,\Psi_2(e,z))&=\Psi_2(x,z)-\Psi_2(x,z)=0 \\
    \Psi_2(\Psi_2(x,y),e)-\Psi_2(x,\Psi_2(y,e))&=\Psi_2(x,y)-\Psi_2(x,y)=0
\end{align*}
so that the above trilinear map takes 
\[(\mathfrak{s}\otimes \operatorname{id}\otimes \operatorname{id})\oplus(\operatorname{id}\otimes\mathfrak{s} \otimes \operatorname{id})\oplus ( \operatorname{id}\otimes \operatorname{id}\otimes \mathfrak{s}) \]
to zero. With this in mind, we make the following definition.
\begin{defi}
    Let $(\mathcal{V}_\bullet,\Psi_2,\mathfrak{s})$ be the universal unital non-associative algebra of $\operatorname{naNHilb}^{\underline{d}}(\A^n)$. We define the associativity sheaf 
    \[\mathcal{E}_{\operatorname{ass}}\subset \mathcal{H}om((\mathcal{V}_0/\mathfrak{s})\otimes (\mathcal{V}_0/\mathfrak{s})\otimes (\mathcal{V}_0/\mathfrak{s}),\mathcal{V}_0)\]
    as the sub-sheaf of trilinear maps which satisfies the following properties.
    \begin{enumerate}
        \item They are alternating in the outer coordinates.
        \item They take $\mathcal{V}_{i}\otimes \mathcal{V}_j\otimes \mathcal{V}_k$ to $\mathcal{V}_{\operatorname{max}\{i,j,k\}}$.
    \end{enumerate}
    It comes equipped with a canonical associator section $s_{\operatorname{ass}}$ given by 
    \[\Psi_2(\Psi_2\otimes\operatorname{id})-\Psi_2(\operatorname{id}\otimes \Psi_2). \]
\end{defi}
\begin{remark}
Note that in the special case where $d_0=1$, the inclusion $\mathcal{V}_1\subset \mathcal{V}_0$ induces an isomorphism $\mathcal{V}_1\cong \mathcal{V}_0/\mathfrak{s}$.
\end{remark}
Since we can scheme theoretically identify $\operatorname{NHilb}^{\underline{d}}(\A^n)$ with the associativity locus which can again be identified with $Z(s_{ass})$ we obtain the following result.
\begin{corollary}
    The closed embedding $\iota: \operatorname{NHilb}^{\underline{d}}(\A^n)\to \operatorname{naNHilb}^{\underline{d}}(\A^n)$ identifies with the closed immersion $Z(s_{ass})$. In particular, it comes equipped with a perfect obstruction theory given by
    \begin{center}
        \begin{tikzcd}
            \mathbb{E}_{ass} \arrow{d}&[-30pt]:=&[-35pt] \Big{[}&[-35pt] (\mathcal{E}^{ass})^\vee\big{|}_{\operatorname{NHilb}^{\underline{d}}(\A^n)} \arrow{d}{s_{ass}}\arrow{r}{d\circ s_{ass}} & \Omega_{\operatorname{naNHilb}^{\underline{d}}(\A^n)}\big{|}_{\operatorname{NHilb}^{\underline{d}}(\A^n)}\arrow[equal]{d} \Big{]} \\
            \mathbb{L}_{\operatorname{NHilb}^{\underline{d}}(\A^n)}&[-30pt]\cong&[-35pt] \Big{[} &[-35pt] \mathcal{C}_{\operatorname{NHilb}^{\underline{d}}(\A^n)/\operatorname{naNHilb}^{\underline{d}}(\A^n)} \arrow{r}{d} & \Omega_{\operatorname{naNHilb}^{\underline{d}}(\A^n)}\big{|}_{\operatorname{NHilb}^{\underline{d}}(\A^n)} \Big{]}
        \end{tikzcd}
    \end{center}
    where $\mathcal{C}_{\operatorname{NHilb}^{\underline{d}}(\A^n)/\operatorname{naNHilb}^{\underline{d}}(\A^n)}$ is the conormal sheaf of the closed immersion $\iota$. This equips it with a virtual class 
    \[\left[\operatorname{NHilb}^{\underline{d}}(\A^n)\right]^{\operatorname{vir}}\in A_{\ast}(\operatorname{NHilb}^{\underline{d}}(\A^n)) \]
    which can also be described as the refined Gysin map of the zero section of $\mathcal{E}$ applied to the fundamental class of $\operatorname{naNHilb}^{\underline{d}}(\A^n)$. 
\end{corollary}
\section{The punctual locus in the full flag case}
\label{sec:NA-punct-Hilb}
In this section we will restrict ourselves to the full flag case, i.e., a dimension vector $\underline{d}=(1,1,\dots , 1)\in \N^{r+1}$. We will denote by
\[\operatorname{naNHilb}^{r+1}(\A^n):=\operatorname{naNHilb}^{(1,1,\dots , 1)}(\A^n)\]
and
\[\operatorname{NHilb}^{r+1}(\A^n):=\operatorname{NHilb}^{(1,1,\dots , 1)}(\A^n)\]
to ease the notation. We will prove two nice properties enjoyed by the non-associative Hilbert scheme with a full flag. Firstly, in this context we are able to extend the nested Hilbert-Chow morphism to a non-associative nested Hilbert-Chow morphism. Secondly, we can construct a punctual locus of the non-associative Hilbert scheme which we prove admits the structure of a tower of projective fibrations.
\subsection{The non-associative Hilbert-Chow morphism}
Recall that for the nested Hilbert scheme of points with full dimension vector, we have an ordered Hilbert-Chow morphism
\begin{center}
    \begin{tikzcd}
        \operatorname{NHilb}^{(1,\dots, 1)}(\A^n) \arrow{r}{HC} & (\A^n)^{r+1} \\[-20pt]
        (I_{r+1}\subset I_r \subset \cdots \subset I_0=\C[x_1,\dots, x_n]) \arrow[mapsto]{r} & \left(\operatorname{Supp}(I_0/I_1),\operatorname{Supp}(I_1/I_2), \dots, \operatorname{Supp}(I_r/I_{r+1})\right). 
    \end{tikzcd}
\end{center}
Since 
\[I_{i}/I_{i+1}=(I_i/I_{r+1})\big{/}(I_{i+1}/I_{r+1})=N_i/N_{i+1}\] this suggest what to do in the non-associative case. 
We note that if $(N_{\bullet},\psi_1,s,\psi_2)$ is a $\C$-valued point of the non-associative Hilbert scheme, we cannot quite talk about 
\[\operatorname{Supp}(N_i/N_{i+1})\]
as we only have a $\C$-module structure and not a $\C[x_1,\dots, x_n]$-module structure. However, promoting it to a $\C[x_1,\dots, x_n]$-module is the same as a choice of $n$ commuting endomorphisms of $N_i/N_{i+1}$. By definition $\psi_2$ induces a map
\[\psi_2: N_0\otimes_{\C} (N_{i}/N_{i+1})\to N_{i}/N_{i+1}.\]
We can then simply take the $n$-elements
\[\psi_1=\textstyle{\bigoplus}_{i=1}^n \iota_i: (\C^n)^\vee \to M_i\]
and define the endomorphisms by $\psi_2(\iota_i,-)$. Without associativity these might not commute. However, this is true for dimension reasons once we restrict to the one-dimensional space $N_i/N_{i+1}$. Having this picture in mind, we make the following construction.

\begin{prop}
    There exist a map 
    \[\operatorname{naHilb}^r(\A^n)\xrightarrow{naHC}(\A^n)^{r+1}\]
    which we will call \textit{the non-associative Hilbert-Chow morphism}. The composite
    \[\operatorname{NHilb}^{r+1}(\A^n)\xrightarrow{\iota}\operatorname{naNHilb}^{r+1}(\A^n) \xrightarrow{naHC} (\A^n)^{r+1}\]
    agrees with the ordinary nested Hilbert-Chow morphism $HC$.
\end{prop}
\begin{proof}
    We construct the morphism directly. Let $(\mathcal{V}_\bullet, \Psi_2,\mathfrak{s},\Psi_1)$ be the universal flag and maps of $\operatorname{naNHilb}^{r+1}(\A^n)$. Let $\mathcal{L}_i=\mathcal{V}_i/\mathcal{V}_{i+1}$ which is a line bundle. Then the maps $\Psi_2$ and $\Psi_1$ induces a map 
    \[\Psi_2(\Psi_1\otimes \operatorname{id}): (\oo^{n})^{\vee} \otimes \mathcal{L}_i \to \mathcal{L}_i\]
    and so, we obtain a map 
    \[(\oo^n)^\vee \to \mathcal{L}_i\otimes \mathcal{L}_i^{-1}\simeq \oo\]
    which induces a map $\operatorname{naNHilb}^{r+1}(\A^n)\to \A^n$. Doing this for $i=0,\dots, r$ we obtain the desired map 
    \[\operatorname{naNHilb}^{r+1}(\A^n)\to (\A^n)^{r+1}.\]
    Going through the definitions of $\iota$ and $HC$ on the functor of points one easily sees that $naHC\circ \iota=HC$.
\end{proof}
\subsection{The punctual locus}
We construct a closed locus $\operatorname{naNHilb}_0^{r+1}(\A^n)\hookrightarrow \operatorname{naNHilb}^{r+1}(\A^n)$ in such a way that it intersects with the associativity locus exactly in the punctual nested Hilbert scheme. The first idea would be to simply take the preimage of the origin under the Hilbert-Chow morphism. However, it turns out that choosing a slightly smaller locus yields better formal properties. To motivate this locus recall that the first point of the ordered Hilbert Chow morphism is given by $\operatorname{Supp}(I_0/I_1)$. For this to be zero we need multiplication with $x_i$ to vanish on $I_0/I_1$ which is the same as requiring $I_1=(x_1,\dots, x_n)$. Now for $\operatorname{Supp}(I_k/I_{k+1})$ to be zero we need $x_i\cdot I_k\subset I_{k+1}$ for $i=1,\dots, n$, but since $I_1=(x_1,\dots, x_n)$ this is the same as requiring 
\[(I_1/I_{r+1})\cdot(I_k/I_{r+1})\subset (I_{k+1}/I_{r+1}).\]
Translating these conditions to the non-associative Hilbert scheme, we arrive at the following definition
\begin{defi}
    Consider $\operatorname{naNHilb}^{r+1}(\A^n)$ with the universal flag and maps $(\mathcal{V}_\bullet,\Psi_2,s,\Psi_1)$. The punctual nested non-associative Hilbert scheme is the closed subscheme \[\operatorname{naNHilb}^{r+1}_0(\A^n)\subset \operatorname{naNHilb}^{r+1}(\A^n)\]
    where $\Psi_1$ factors through $\mathcal{V}_1$ and where $\Psi_2$ maps $\mathcal{V}_1\otimes \mathcal{V}_k$ to $\mathcal{V}_{k+1}$ for all $k=1,\dots, r$. If we let $\mathcal{Q}$ be the quotient of $\mathcal{H}om(\sym^2 \mathcal{V}_1,\mathcal{V}_1)$ by the subsheaf of maps taking $\mathcal{V}_1\otimes \mathcal{V}_k$ to $\mathcal{V}_{k+1}$ then we can similarly describe the punctual locus as the zero locus of the section of
    \[\mathcal{H}om\left((\oo^n)^\vee,\mathcal{V}_0/\mathcal{V}_1\right)\oplus\mathcal{Q}\]
    given by $\Psi_1\oplus \Psi_2$.
\end{defi}
We prove that this actually is an extension of the punctual locus 
\[\operatorname{NHilb}_0^{r+1}(\A^n)=HC^{-1}(\mathbf{0})\subset \operatorname{NHilb}^{r+1}(\A^n)\]
of the nested Hilbert scheme of points.
\begin{lemma}
It holds that
\[\operatorname{naHC}(\operatorname{naNHilb}_0^{r+1}(\A^n))\subset \mathbf{0}.\]
\end{lemma}
\begin{proof}
  Let again $\mathcal{L}_i=\mathcal{V}_i/\mathcal{V}_{i+1}$. Since $\Psi_1$ factors through $\mathcal{V}_1$ on the punctual locus it follows that the map 
  \[\Psi_2(\Psi_1\otimes \operatorname{id}):(\oo^n)^\vee \otimes \mathcal{L}_0\to \mathcal{L}_0\]
  vanishes. Similarly, we can factor
  \[\Psi_2(\Psi_1\otimes \operatorname{id}):(\oo^n)^\vee \otimes \mathcal{L}_k\to \mathcal{V}_1\otimes \mathcal{L}_k\xrightarrow{\Psi_2} \mathcal{L}_k\]
  and use that $\Psi_2(\mathcal
  V_1\otimes \mathcal{V}_k)\subset \mathcal{V}_{k+1}$ to see that this map also vanishes.
\end{proof}
\begin{lemma}
There exists a closed immersion $\iota_0: \operatorname{NHilb}_0^{r+1}(\A^n) \to \operatorname{naNHilb}_0^{r+1}(\A^n) $ such that the diagram
\begin{center}
    \begin{tikzcd}
        \operatorname{NHilb}_0^{r+1}(\A^n) \arrow{r}{\iota_0} \arrow{d} & \operatorname{naNHilb}_0^{r+1}(\A^n) \arrow{d} \\
        \operatorname{NHilb}^{r+1}(\A^n) \arrow{r}{\iota} & \operatorname{naNHilb}^{r+1}(\A^n)  
    \end{tikzcd}
\end{center}
is a pullback. 
\end{lemma}
\begin{proof}
    We prove the statement on the functor of points. Given $S\in \operatorname{Sch}_{\C}$ and 
    \[\Big{[}\mathcal{I}_{r+1}\subset \cdots \subset \mathcal{I}_{1}\subset \mathcal{I}_0= \oo_S[x_1,\dots, x_n]\Big{]}\]
    of $\operatorname{NHilb}^{r+1}(\A^n)$ we should prove that it factors through $\operatorname{NHilb}_0^{r+1}(\A^n)$ if and only if it factors through $\operatorname{naNHilb}_0^{r+1}(\A^n)$ after composing it with $\iota$. The "if" part follows from the above, since $\operatorname{naHC}(\operatorname{naNHilb}_0^{r+1}(\A^n))\subset \mathbf{0}$ and $naHC\circ \iota=HC$. Conversely, if the point is punctual, then $x_i\mathcal{I}_k\subset \mathcal{I}_{k+1}$ for all $i=1, \dots ,n$ and all $k=0,\dots, r$. For $k=0$ this means that $x_i\in \mathcal{I}_1$ for all $i$ hence $\mathcal{I}_1=(x_1,\dots, x_n)$. Since $x_i\cdot \mathcal{I}_k\subset \mathcal{I}_{k+1}$ we conclude that 
    \[(\mathcal{I}_1/\mathcal{I}_{r+1})\cdot(\mathcal{I}_{k}/\mathcal{I}_{r+1})\subset (\mathcal{I}_{k+1}/\mathcal{I}_{r+1})\]
    thus, proving that its image under $\iota$ is in $\operatorname{naNHilb}_0^{r+1}(\A^n)$.
\end{proof}
\subsection{Projective tower structure}
We prove that the punctual non-associative Hilbert scheme can be exhibited as an iterated tower of projective fibrations. The proof is essentially the same as in \cite[Proposition 6.3]{KazarianNAHilb} with a slightly more general space. Before proceeding, we show how the nested non-associative Hilbert scheme generally maps to a nested non-associative Hilbert scheme with a truncated dimension vector. 
\begin{defi}
    Let $\underline{d}=(d_0,\dots, d_r)$ and consider $\operatorname{naNHilb}^{\underline{d}}(\A^n)$ with universal flag and maps $(\mathcal{V}_\bullet, \Psi_2,\mathfrak{s},\varpi)$. Define the flag $\mathcal{F}_\bullet$ with dimension vector $(d_0,\dots, d_{r-1})$ by setting
    \[\mathcal{F}_k:=\mathcal{V}_k/\mathcal{V}_r.\]
    We get an induced unital non-associative algebra structure from $\Psi_2,\mathfrak{s}$ exhibiting $\mathcal{F}_\bullet$ as a quotient algebra of $\mathcal{V}_\bullet$, so that we also get an induced surjection via $\varpi$. This data determines a map 
    \[\operatorname{naNHilb}^{(d_0,\dots, d_r)}(\A^n)\to\operatorname{naNHilb}^{(d_0,\dots, d_{r-1})}(\A^n). \]
\end{defi}

\begin{remark}
    Since a quotient of an associative algebra is again associative this map preserves the associativity locus. One can easily see that on the nested Hilbert scheme of points this is simply the map 
    \[\Big{[} I_{r+1}\subset I_r \subset \cdots \subset I_1\subset \C[x_1,\dots, x_n]\Big{]}\mapsto \Big{[}  I_r \subset \cdots \subset I_1\subset \C[x_1,\dots, x_n]\Big{]}  \]
    which forgets the largest subscheme of the nesting. 
\end{remark}

\begin{lemma}
    The natural map 
    \[\operatorname{naNHilb}^{r+1}(\A^n)\to \operatorname{naNHilb}^{r}(\A^n)\]
    restricts to the punctual locus i.e., to a map
    \[\operatorname{naNHilb}^{r+1}_0(\A^n)\to \operatorname{naNHilb}^{r}_0(\A^n).\]
\end{lemma}
\begin{proof}
    If $\Psi_1$ factors through $\mathcal{V}_1\subset \mathcal{V}_0$ then the induced map factors through $\mathcal{V}_1/\mathcal{V}_r\subset \mathcal{V}_0/\mathcal{V}_r$. Similarly, if $\Psi_2$ maps $\mathcal{V}_1\otimes \mathcal{V}_k$ to $\mathcal{V}_{k+1}$ then the induced map maps $(\mathcal{V}_1/\mathcal{V}_r)\otimes (\mathcal{V}_k/\mathcal{V}_r)$ to $(\mathcal{V}_{k+1}/\mathcal{V}_r)$.
\end{proof}
Before proving the projective tower statement, we show that the unit does not contain any additional information in the punctual case.
\begin{defi}
     Let $S\in \operatorname{Sch}_{S}$ and let $\oo_S\{x_1,\dots x_n\}_c^{+}\subset \oo_S\{x_1,\dots x_n\}_c$ be the ideal spanned by $x_1,\dots, x_n$. We view it as the free non-unital non-associative algebra of $(\oo_S^n)^\vee$.
\end{defi}
\begin{lemma}
    Let $S\in \operatorname{Sch}_{S}$. The map 
    \begin{center}
        \begin{tikzcd}
           \left( \operatorname{naNHilb}_0^{r+1}(\A^n) \right)(S)\arrow{r} & \left\{(\mathcal{G}_\bullet,\psi_2,\psi_1)\right\}\big{/}\text{iso} \\[-20pt]
            \left[\mathcal{F}_\bullet,\psi_2,s,\psi_1\right] \arrow[mapsto]{r} & \left[(\mathcal{F}_1)_\bullet,\psi_2,\psi_1\right]
        \end{tikzcd}
    \end{center}
    which forgets $\mathcal{F}_0$ and the unit $s$ is a bijection. Here the tuples in the right-hand side set consists of a flag of sheaves 
    \[\mathcal{G}_1\supset \mathcal{G}_2 \supset \cdots \mathcal{G}_{r+1}=0\]
    such that $\mathcal{G}_k/\mathcal{G}_{k+1}$ is an invertible sheaf for all $k=1,\dots r$, a non-unital commutative non-associative algebra structure $\psi_2:\sym^2\mathcal{G}_1\to \mathcal{G}_1$ such that $\psi_2(\mathcal{G}_1\otimes_S\mathcal{G}_k) \subset \mathcal{G}_{k+1}$ and a map $\psi_1: (\oo_S^n)^\vee \to \mathcal{G}$ such that the induced map $\oo_S\left\{x_1,\dots, x_n\right\}_c^{+}\to \mathcal{G}$ is an epimorphism. 
\end{lemma}
\begin{proof}
    We first prove that the map is well defined. Indeed, since we are in the punctual locus we have $\psi_2(\mathcal{F}_1\otimes_S \mathcal{F}_k)\subset \mathcal{F}_{k+1}$ and $\psi_1$ factors through $\mathcal{F}_1$. We have a commutative diagram
    \begin{center}
        \begin{tikzcd}
            0\arrow{r} & \oo_S\{x_1,\dots, x_n\}_c^{+} \arrow{r} \arrow{d} & \oo_S\{x_1,\dots, x_n\}_c \arrow{r} \arrow{d} & \oo_S \arrow{r} \arrow{d}{s} & 0 \\ 
            0 \arrow{r} & \mathcal{F}_1 \arrow{r} & \mathcal{F}_0 \arrow{r} & \mathcal{F}_0/\mathcal{F}_1 \arrow{r} & 0
        \end{tikzcd}
    \end{center}
    with exact rows. The right vertical map is an isomorphism, and the middle vertical map is an epimorphism, so the 5-lemma implies that the left verical map is also an epimorphism. This shows that the map is well defined. Conversely given a point $[\mathcal{G}_\bullet,\psi_2,\psi_1]$ we define $\mathcal{F}_\bullet$ by setting
    \[\mathcal{F}_0=\oo_S\oplus \mathcal{G}_1\]
    and $\mathcal{F}_k=\mathcal{G}_k$ for $k=1,\dots, r$. We note that 
    \[\sym^2\mathcal{F}_0\cong \oo_S \oplus \mathcal{G}_1\oplus \sym^2\mathcal{G}_1 \cong \mathcal{F}_0\oplus \sym^2 \mathcal{G}_1\]
    and extend $\psi_2$ to a map 
    \[\sym^2 \mathcal{F}_0\to \mathcal{F}_0\]
    by letting it act by the identity on the first summand and by $\psi_2$ on the second. We let $s$ be the inclusion 
    \[\oo_S\hookrightarrow \oo_S\oplus \mathcal{G}_1=\mathcal{F}_0\]
    and let $\psi_1$ be the composite
    \[(\oo_S^n)^\vee \xrightarrow{\psi_1} \mathcal{G}_1\hookrightarrow \oo_S\oplus \mathcal{G}_1=\mathcal{F}_0.\]
    We have exactly the same commutative diagram as above. Since the left vertical map is an epimorphism and the right vertical map is an isormophism it follows that the middle vertical map is an epimorphism. This proves that the map is well defined. One easily checks that these maps are inverse to each other. 
\end{proof}

\begin{lemma} \label{Kazarian stability}
    Let $\underline{d}=(1,d_1,\dots, d_r)\in \N^r$, let $S\in \operatorname{Sch}_{\C}$ and $(\mathcal{F}_\bullet, \psi_2,s,\psi_1)$ defined as follows.
    \begin{enumerate}
        \item $\mathcal{F}_\bullet$ is a flag of sheaves 
    \[\mathcal{F}_0\supset \mathcal{F}_1\supset \cdots \mathcal{F}_{r+1}=0\]
    with $\mathcal{F}_0/\mathcal{F}_{1}$ invertible and $\mathcal{F}_{k}/\mathcal{F}_{k+1}$ locally free of rank $d_k$.
    \item $(\psi_2,s)$ is a unital non-associative algebra structure on $\mathcal{F}_0$ such that $\psi_2(\mathcal{F}_1\otimes_S\mathcal{F}_{k})\subset \mathcal{F}_{k+1}$ and $s\notin \mathcal{F}_1$. 
    \item $\psi_1:(\oo_S^n)^\vee\to \mathcal{F}_1$ is a map of sheaves.
    \end{enumerate}
Then the induced map 
\[\oo_S\{x_1,\dots x_n\}_c \to \mathcal{F}_0\]
is surjective if and only if the map
\[\psi_1\oplus \psi_2: (\oo_S^n)^\vee \oplus \sym^2\mathcal{F}_1\to \mathcal{F}_1\]
is surjective.
\end{lemma}
\begin{proof}
    We first note that since $s:\oo_S\to \mathcal{F}_0/\mathcal{F}_1$ is an isomorphism (since $\mathcal{F}_0/\mathcal{F}_1$ is an invertible sheaf) it follows that the map 
    \[\oo_S\{x_1,\dots x_n\}_c \to \mathcal{F}_0\]
    is surjective if and only if 
    \[\oo_S\{x_1,\dots x_n\}_c^{+} \to \mathcal{F}_1\]
    is surjective. Clearly this implies that $\psi_1\oplus \psi_2$ is surjective since any element of $\mathcal{F}_1$ is a linear combination of products of elements in the image of $\psi_1$ i.e., elements either in the image $\psi_1$ or in the image of $\psi_2$. To prove the converse, we proceed by induction on $r$. For $r=1$ we see that the multiplication map takes $\mathcal{F}_1\otimes_S \mathcal{F}_1$ to $\mathcal{F}_2=0$ so that $\psi_2=0$. Therefore, if $\psi_1\oplus \psi_2$ is surjective then $\psi_1$ is surjective so that
   \[\oo_S\{x_1,\dots,x_n\}_c^{+} \to \mathcal{F}_1\]
    is also surjective. Let now $\mathcal{F}_\bullet, \psi_1,\psi_2$ be given with $\psi_1\oplus \psi_2$ surjective. By the induction hypothesis the composite
\[\oo_S\{x_1,\dots,x_n\}_c^{+} \to \mathcal{F}_1\to \mathcal{F}_1/\mathcal{F}_r\]
is surjective so that the map 
\[\oo_S\{x_1,\dots,x_n\}_c^{+}\oplus \mathcal{F}_r\to \mathcal{F}_1\]
given by the quotient and the inclusion is surjective. In particular, the map
\[(\oo_S^n)^\vee\oplus \sym^2(\oo_S\{x_1,\dots,x_n\}_c^{+}\oplus \mathcal{F}_r)\to (\oo_S^n)^\vee \oplus \sym^2\mathcal{F}_1\to \mathcal{F}_1\]
is surjective. We have
\[\sym^2(\oo_S\{x_1,\dots,x_n\}_c^{+}\oplus \mathcal{F}_r)\cong \sym^2(\oo_S\{x_1,\dots,x_n\}_c^{+})\oplus (\oo_S\{x_1,\dots,x_n\}_c^{+}\otimes \mathcal{F}_r)\oplus \sym^2 \mathcal{F}_r\]
and since $\psi_2(\mathcal{F}_1\otimes_S\mathcal{F}_r)\subset \mathcal{F}_{r+1}=0$ it follows that the above map vanishes on the last two summands. We conclude that the map 
\[(\oo_S^n)^{\vee}\oplus\sym^2(\oo_S\{x_1,\dots,x_n\}_c^{+})\to \mathcal{F}_1 \]
is surjective. The image of $\oo_S\{x_1,\dots,x_n\}_c^{+}$ is closed under $\psi_2$ so that it contains the image of the above map. We conclude that it is the whole of $\mathcal{F}_1$.
\end{proof}
\begin{corollary}
    An $S$-valued point of $\operatorname{naNHilb}_0^{r+1}(\A^n)$ is the same as a tuple $(\mathcal{F}_\bullet,\psi_2,\psi_1)$ where 
    \begin{enumerate}
        \item $\mathcal{F}_\bullet$ is a flag in $\operatorname{QCoh}(S)$ 
        \[\mathcal{F}_1\supset \mathcal{F}_2 \supset \cdots \supset \mathcal{F}_{r+1}=0\]
        such that $\mathcal{F}_k/\mathcal{F}_{k+1}$ is invertible.
        \item $\psi_2$ is a non-associative algebra structure on $\mathcal{F}_1$ such that $\psi_2(\mathcal{F}_1\otimes_S\mathcal{F}_k)\subset \mathcal{F}_{k+1}$.
        \item $\psi_1:(\oo_S^n)^\vee\to \mathcal{F}_1$ is a map of sheaves such that $\psi_1\oplus \psi_2$ is surjective.
    \end{enumerate}
    It follows that $\operatorname{naHilb}_0^{r+1}(\A^n)$ comes equipped with universal such structure, which we will denote by $(\mathcal{N}_\bullet,\Psi_2,\Psi_1)$. 
\end{corollary}
\begin{prop} \label{projective tower structure}
    Consider $\operatorname{naNHilb}^{r}(\A^n)$ with universal structure $(\mathcal{N}_\bullet,\Psi_2,\Psi_1)$. Then the natural map 
    \[\operatorname{naNHilb}_0^{r+1}(\A^n)\to \operatorname{naNHilb}_0^{r}(\A^n)\]
    identifies with the structure map
    \[\mathbb{P}_{\operatorname{naNHilb}^{r}(\A^n)}\left((\operatorname{ker}\Psi_1\oplus \Psi_2)^\vee \right)\to \operatorname{naNHilb}^{r}(\A^n).\]
\end{prop}
\begin{proof}
We should prove that given $S\in \operatorname{Sch}_{\C}$ and a map $S\to \operatorname{naNHilb}^{r}(\A^n)$ classified by $[\mathcal{F}_\bullet, \psi_2,\psi_1]$ the set of lifts
\begin{center}
    \begin{tikzcd}
        & \operatorname{naNHilb}_0^{r+1}(\A^n) \arrow{d}\\
        S \arrow[dashed]{ru} \arrow{r} & \operatorname{naNHilb}_0^{r}(\A^n)
    \end{tikzcd}
\end{center}
is in bijection with the set of $(\mathcal{L},q)$ where $\mathcal{L}$ is an invertible sheaf on $S$ and 
\[q: \operatorname{ker}\psi_1\oplus \psi_2\to \mathcal{L} \]
is a surjective map. We recall that a lift is classified by $[\widetilde{\mathcal{F}}_\bullet, \widetilde{\psi}_2,\widetilde{\psi}_1]$ such that 
\[\widetilde{\mathcal{F}}_k/\widetilde{\mathcal{F}}_{r}\simeq \mathcal{F}_k\]
with the maps induced by $\widetilde{\psi}_2,  \widetilde{\psi}_1$ agreeing with $\psi_1,\psi_2$. If we have a lift then we consider the commutative diagram
\begin{center}
    \begin{tikzcd}
        & \widetilde{\mathcal{F}}_1\otimes_S \widetilde{\mathcal{F}}_{r} \arrow{r} \arrow{d}{\widetilde{\psi}_2} & (\oo_S^n)^\vee \oplus \sym^2\widetilde{\mathcal{F}}_1 \arrow{r} \arrow{d}{\widetilde{\psi}_1\oplus \widetilde{\psi}_2}& (\oo^n_S)^\vee \oplus \sym^2\mathcal{F}_1 \arrow{d}{\psi_1\oplus \psi_2} \arrow{r} & 0 \\
        0 \arrow{r} & \widetilde{\mathcal{F}}_r \arrow{r} & \widetilde{\mathcal{F}}_1 \arrow{r} & \mathcal{F}_1 \arrow{r} & 0
    \end{tikzcd}
\end{center}
with exact rows. The middle and right vertical maps are surjective by assumption and since $\widetilde{\psi}_2(\widetilde{\mathcal{F}}_1\otimes_S \widetilde{\mathcal{F}}_{r})\subset \widetilde{\mathcal{F}}_{r+1}=0$ we conclude that the left vertical map is zero. If we set $\mathcal{L}:=\widetilde{\mathcal{F}}_{r}$ the snake lemma the provides the desired map 
\[q:\operatorname{ker}(\psi_1\oplus \psi_2)\to \widetilde{\mathcal{F}}_r.\]
We note that we in fact have a map 
\[(\oo^n_S)^\vee \oplus \sym^2\mathcal{F}_1 \to \widetilde{\mathcal{F}_1}\]
which factors as $q$ when restricted to $\operatorname{ker}(\psi_1\oplus \psi_2)$.
Conversely, if we are given $(\mathcal{L},q)$ we let $\widetilde{\mathcal{F}}_1$ be the pushout of $q$ along the inclusion of the kernel into $(\oo^n_S)^\vee \oplus \sym^2\mathcal{F}_1$. The left and outer square of the diagram
\begin{center}
    \begin{tikzcd}
        \operatorname{ker}(\psi_1\oplus \psi_2) \arrow[hook]{d} \arrow[two heads]{r}{q} & \mathcal{L} \arrow[hook]{d} \arrow{r} & 0 \arrow{d} \\
        (\oo^n_S)^\vee \oplus \sym^2\mathcal{F}_1 \arrow[two heads]{r} & \widetilde{\mathcal{F}}_1 \arrow[two heads]{r} & \mathcal{F}_1
    \end{tikzcd}
\end{center}
are pushouts hence so is the right square by the pasting lemma. This proves that $\widetilde{\mathcal{F}}_1$ is an extension of $\mathcal{F}_1$ by $\mathcal{L}$. We define $\widetilde{\mathcal{F}}_k\subset \widetilde{\mathcal{F}}_1$ to be the preimage of $\mathcal{F}_k\subset \mathcal{F}_1$ for $k=1,\dots r-1$ and $\widetilde{\mathcal{F}}_{r}=\mathcal{L}$. We define $\widetilde{\psi}_1\oplus \widetilde{\psi}_2$ as the composite
\[ (\oo^n_S)^\vee \oplus \sym^2\widetilde{\mathcal{F}}_1\twoheadrightarrow (\oo^n_S)^\vee \oplus \sym^2\mathcal{F}_1 \twoheadrightarrow\widetilde{\mathcal{F}}_1 \]
where the second map is from the pushout and the first map is the identity on the first factor and the quotient on the second. Since $\widetilde{\mathcal{F}}_1\otimes_S\widetilde{\mathcal{F}}_r$ is mapped to $0$ by the quotient map it is mapped to to $0=\widetilde{\mathcal{F}}_{r+1}$ by $\psi_2$. Since the composite
\[\sym^2 \mathcal{F}_1\to \widetilde{\mathcal{F}}_1 \to \mathcal{F}_1\]
agrees with $\psi_2$ we also see from the definition of $\widetilde{\mathcal{F}}_k$ that $\widetilde{\psi}_2(\widetilde{\mathcal{F}}_1\otimes_S\widetilde{\mathcal{F}}_k)\subset \widetilde{\mathcal{F}}_{k+1}$ thus proving that we have a well-defined $S$-valued point on $\operatorname{naNHilb}_0^{r+1}(\A^n) $. We prove these constructions are inverse to each other. Indeed, in the above we see that $\widetilde{\mathcal{F}}_r=\mathcal{L}$ and that $q$ agrees with the restriction of $(\oo^n_S)^\vee \oplus \sym^2\mathcal{F}_1 \to \widetilde{\mathcal{F}}_1$ to $\operatorname{ker}(\psi_1\oplus \psi_2)$ so that we regain $(\mathcal{L},q)$ when applying both constructions. Conversely, if we start with $[\widetilde{\mathcal{F}}_\bullet, \widetilde{\psi}_1,\widetilde{\psi}_2]$ we have a commutative diagram
\begin{center}
    \begin{tikzcd}
        \operatorname{ker}(\psi_1\oplus\psi_2) \arrow[two heads]{r} \arrow[hook]{d} & \widetilde{\mathcal{F}}_r \arrow[hook]{d} \\
        (\oo^n_S)^\vee \oplus \sym^2\mathcal{F}_1  \arrow[two heads]{r} & \widetilde{\mathcal{F}}_1
    \end{tikzcd}
\end{center}
which we need to prove is a pushout. Indeed, let \[\mathcal{K}=\operatorname{ker}\left( (\oo^n_S)^\vee \oplus \sym^2\mathcal{F}_1\to \widetilde{\mathcal{F}}_1\right)\]
and 
\[\mathcal{K}'=\operatorname{ker}\left(\operatorname{ker}(\psi_1\oplus \psi_2)\to \widetilde{\mathcal{F}}_r\right).\]
We have a commutative diagram 
\begin{center}
    \begin{tikzcd}
        0\arrow{r} & \operatorname{ker}(\psi_1\oplus \psi_2) \arrow{r} \arrow{d} &  (\oo^n_S)^\vee \oplus \sym^2\mathcal{F}_1 \arrow{d} \arrow{r} & \mathcal{F}_1 \arrow{r} \arrow[equal]{d} & 0 \\
        0\arrow{r} & \widetilde{\mathcal{F}}_r \arrow{r} & \widetilde{\mathcal{F}}_1  \arrow{r} & \mathcal{F}_1 \arrow{r} & 0
    \end{tikzcd}
\end{center}
with exact rows so that $\mathcal{K}'\to \mathcal{K}$ is an isomorphism by the snake lemma. Given any $\mathcal{G}$ and a commuting diagram 
\begin{center}
    \begin{tikzcd}
        \operatorname{ker}(\psi_1\oplus\psi_2) \arrow[two heads]{r} \arrow[hook]{d} & \widetilde{\mathcal{F}}_r \arrow{d} \\
        (\oo^n_S)^\vee \oplus \sym^2\mathcal{F}_1  \arrow{r} & \mathcal{G}
    \end{tikzcd}
\end{center}
it then follows that $\mathcal{K}'\to \mathcal{K}\to \mathcal{G}$ is zero and hence $\mathcal{K}\to \mathcal{G}$ is zero proving that the lower horizontal map factors through a unique map $\widetilde{\mathcal{F}}_1\to \mathcal{G}$. This proves that the diagram is a pushout thus finishing the proof.
\end{proof}
\begin{exmp}
Let us calculate this tower explicitly for small $r$. We first note that for $r=0$ we have $\operatorname{naNHilb}^{1}_0(\A^n)\cong \spec \C$ with universal bundle $\mathcal{N}_1=0$ and $\Psi_1\oplus \Psi_2: (\C^n)^\vee \oplus 0\to 0$. This has kernel $(\C^n)^\vee$ so that
\[\operatorname{naNHilb}^{2}_0(\A^n)\cong \mathbb{P}(\C^n)\cong \mathbb{P}^{n-1}.\]
The above proof gives the recipe for building the universal bundle. We simply define 
\[0=\mathcal{N}_2\subset \mathcal{N}_1=\oo(1)\]
with multiplication map
\[\Psi_2:\sym^2 \mathcal{N}_1\to0= \mathcal{N}_2\]
and 
\[\Psi_1: (\oo_{\mathbb{P}^{n-1}}^n)^\vee\to \oo(1)\]
dual to the universal sub-linebundle
\[\oo(-1)\to \oo_{\mathbb{P}^{n-1}}^n.\]
Continuing this procedure we get
\[\operatorname{naNHilb}_0^3(\A^n)\cong \mathbb{P}_{\mathbb{P}^{n-1}}(\operatorname{ker}(\Psi_1\oplus \Psi_2)^\vee)\cong  \mathbb{P}_{\mathbb{P}^{n-1}}\left(\oo_{\mathbb{P}^{n-1}}^n/\oo_{\mathbb{P}^{n-1}}(-1)\oplus \sym^2 \oo_{\mathbb{P}^{n-1}}(-1)\right).\]
The universal bundle is then constructed as follows. Denote by $\oo_{\mathbb{P}^{n-1}}(-1)$ and $\mathcal{N}_1$ the pullback of the bundles to $\operatorname{naNHilb}_0^3(\A^n)$ and let $\oo(-1)$ be the universal line bundle on $\operatorname{naNHilb}_0^3(\A^n)$. We define the new universal bundle $\widetilde{\mathcal{N}}_1$ as the pushout
\begin{center}
    \begin{tikzcd}
        (\oo^n/\oo_{\mathbb{P}^{n-1}}(-1))^\vee \oplus \sym^2(\mathcal{N}_1) \arrow{d} \arrow{r} & \oo(1) \arrow{d} \\
        (\oo^n)^{\vee} \oplus \sym^2 \mathcal{N}_1 \arrow{r} & \widetilde{\mathcal{N}}_1
    \end{tikzcd}
\end{center}
where the top horizontal map is the dual of the inclusion of the universal line bundle. We set $\widetilde{\mathcal{N}_2}=\oo(1)$, so that $\widetilde{\mathcal{N}}_1/\widetilde{\mathcal{N}}_2\simeq \mathcal{N}_1$. The universal maps are given by 
\[ (\oo^n)^{\vee} \oplus \sym^2 \widetilde{\mathcal{N}}_1\to  (\oo^n)^{\vee} \oplus \sym^2 \mathcal{N}_1 \to \widetilde{\mathcal{N}}_1\]
where the first map is the identity and quotient and the second map is the structural map from the pushout.
\end{exmp}
\begin{corollary}
    The scheme $\operatorname{naNHilb}^{r+1}_0(\A^n)$ is smooth and projective.
\end{corollary}
\begin{proof}
    This follows directly from the projective tower description since \[\operatorname{naNHilb}^{1}_0(\A^n)\cong \spec \C.\]
\end{proof}

\subsection{Other dimension vectors and relations to the work of Kazarian}
We prove a slight generalization of the above where we allow different dimension vectors. We construct the moduli functor of non-unital non-associative algebras with a flag structure satisfying a nilpotency condition. We prove that it embeds into the non-associative Hilbert scheme by freely adjoining a unit. We prove that its intersection with the associativity locus is contained in the punctual Hilbert scheme but will in general not be the whole punctual Hilbert scheme. A proof completely similar to the section above shows that the locus admits a description as a tower of Grassmannian fibrations. At last, we show that it contains a closed locus of algebras satisfying a stronger nilpotency condition which agrees with the non-associative Hilbert schemes appearing in the work of Kazarian \cite{KazarianNAHilb}. 
\begin{defi} \label{non-unital moduli functor}
    Let $\underline{d}=(1,d_1,\cdots, d_r)\in \N^{r+1}$ be a dimension vector beginning with $1$. We define a moduli functor $\operatorname{naNHilb}^{\underline{d}}_0(\A^n)$ by
\[S\in \operatorname{Sch}_{\C}\mapsto \{(\mathcal{F}_\bullet, \psi_1,\psi_2)\}\big{/}\text{iso}.\]
The tuples of the right-hand set consists of the following data.
\begin{enumerate}
    \item  A flag of locally free sheaves 
\[\mathcal{F}_\bullet=(\mathcal{F}_1\supset \cdots \supset \mathcal{F}_{r+1}=0)\]
with $\mathcal{F}_k/\mathcal{F}_{k+1}$ locally free of rank $d_k$
    \item A non-unital non-associative algebra structure $\psi_2$ on $\mathcal{F}_1$ such that $\psi_2(\mathcal{F}_1\otimes_S \mathcal{F}_k)\subset \mathcal{F}_{k+1}$
    \item A map of sheaves $\psi_1:(\oo_S^n)^\vee \to \mathcal{F}_1$ such that $\psi_1\oplus \psi_2$ is surjective. 
\end{enumerate}
\end{defi} 
\begin{lemma}
    Let $\underline{d}=(1,d_1,\cdots, d_n)\in \N^{r+1}$ and let $(\mathcal{V}_\bullet,\Psi_2,s,\Psi_1)$ be the universal structure on $\operatorname{naNHilb}^{\underline{d}}(\A^n)$. There exists a closed immersion 
    \[\operatorname{naNHilb}^{\underline{d}}_0(\A^n)\to \operatorname{naNHilb}^{\underline{d}}(\A^n)\]
    identifying the domain with the locus where $\Psi_1$ factors through $\mathcal{V}_1$ and where $\Psi_2(\mathcal{V}_1\otimes \mathcal{V}_k)\subset \mathcal{V}_{k+1}$.
\end{lemma}
\begin{proof}
    On the functor of points, we simply map $(\mathcal{F}_\bullet, \psi_2,\psi_1)$ to $(\mathcal{G}_\bullet, \psi_2,s,\psi_1)$ where $\mathcal{G}_0=\oo_S\oplus \mathcal{F}_0$ and $\mathcal{G}_k=\mathcal{F}_k$ for $k=1,\dots, r$ with induced algebra structure from adjoining a unit. From \cref{Kazarian stability} we see that this is indeed a point on $\operatorname{naNHilb}^{\underline{d}}(\A^n)$. Conversely, if we have an $S$-valued point on the described locus we can simply forget the unit and the 0'th piece of the filtration to obtain an $S$-valued point on $\operatorname{naNHilb}^{\underline{d}}_0(\A^n)$. One checks that these constructions are inverse to each other. 
\end{proof}
\begin{remark}
    From the moduli description we see that the scheme $\operatorname{naNHilb}^{\underline{d}}_0$ comes equipped with a universal non-unital non-associative algebra. We denote this universal data by $(\mathcal{N}_\bullet, \psi_1,\psi_2)$.
\end{remark}
\begin{lemma}
    Let $\underline{d}=(1,d_1.,\dots, d_r)\in \N^{r+1}$. The natural map 
    \[\operatorname{naNHilb}^{(1,d_1,\dots,d_r)}(\A^n)\to \operatorname{naNHilb}^{(1,d_1,\dots, d_{r-1})}(\A^n)\]
    restricts to a map 
    \[\operatorname{naNHilb}_0^{(1,d_1,\dots,d_r)}(\A^n)\to \operatorname{naNHilb}_0^{(1,d_1,\dots, d_{r-1})}(\A^n).\]
\end{lemma}
\begin{proof}
    If $\Psi_1$ factors through $\mathcal{V}_1$ then the induced map factors through $\mathcal{V}_1/\mathcal{V}_r$. If $\Psi_2(\mathcal{V}_1\otimes \mathcal{V}_k)\subset \mathcal{V}_{k+1}$ then the induced map takes $(\mathcal{V}_1/\mathcal{V}_r)\otimes (\mathcal{V}_k/\mathcal{V}_r)$ to $\mathcal{V}_{k+1}/\mathcal{V}_r$.
\end{proof}
\begin{prop}
    Let $(\mathcal{N}_\bullet, \Psi_2,\Psi_1)$ be the universal structure on $\operatorname{naNHilb}_0^{(1,d_1,\dots, d_{r-1})}(\A^n)$. The natural map
    \[\operatorname{naNHilb}_0^{(1,d_1,\dots,d_r)}(\A^n)\to \operatorname{naNHilb}_0^{(1,d_1,\dots, d_{r-1})}(\A^n)\]
    identifies with the structure map 
    \[\operatorname{Gr}_{\operatorname{\operatorname{naNHilb}_0^{(1,d_1,\dots, d_{r-1})}(\A^n)}}(d_r,\operatorname{ker}(\Psi_1\oplus \Psi_2)^\vee)\to \operatorname{naNHilb}_0^{(1,d_1,\dots, d_{r-1})}(\A^n).\]
\end{prop}
\begin{proof}
We should prove that given $S\in \operatorname{Sch}_{\C}$ and a map $S\to \operatorname{naNHilb}^{(1,d_1,\dots, d_{r-1})}(\A^n)$ classified by $[\mathcal{F}_\bullet, \psi_2,\psi_1]$ the set of lifts
\begin{center}
    \begin{tikzcd}
        & \operatorname{naNHilb}_0^{(1,d_1,\dots, d_{r})}(\A^n) \arrow{d}\\
        S \arrow[dashed]{ru} \arrow{r} & \operatorname{naNHilb}_0^{(1,d_1,\dots,d_{r-1})}(\A^n)
    \end{tikzcd}
\end{center}
is in bijection with the set of $(\mathcal{G},q)$ where $\mathcal{G}$ is a locally free sheaf of rank $d_r$ on $S$ and 
\[q: \operatorname{ker}\psi_1\oplus \psi_2\to \mathcal{G} \]
is a surjective map. The proof is exactly the same as that of \cref{projective tower structure}.
\end{proof}
\begin{corollary}
    Let $\underline{d}=(1,d_1,\dots, d_r)$. The scheme $\operatorname{naNHilb}_0^{\underline{d}}(\A^n)$ is smooth and projective.
\end{corollary}
\begin{proof}
    This follows inductively from the above since $\operatorname{naNHilb}^{1}_0(\A^n)\cong \spec \C$.
\end{proof}
\begin{lemma}
    Let $\underline{d}=(1,d_1,\dots, d_r)$. Let $S\in \operatorname{Sch}_{\C}$ and  let $S\to \operatorname{NHilb}^{\underline{d}}(\A^n)$ be a map of schemes. If 
    \[S\to \operatorname{NHilb}^{\underline{d}}(\A^n) \xrightarrow{\iota} \operatorname{naNHilb}^{\underline{d}}(\A^n)\]
    factors through $\operatorname{naNHilb}^{\underline{d}}_0(\A^n)$, then the first map factors through $\operatorname{NHilb}_0^{\underline{d}}(\A^n)$.
    \end{lemma}
    \begin{proof}
        Let the map be classified by 
        \[\mathcal{I}_{r+1}\subset \mathcal{I}_r\subset \cdots \subset\mathcal{I}_1\subset\mathcal{I}_0= \oo_S[x_1,\dots, x_n].\]
        We set
        \[\mathcal{F}_k=\mathcal{I}_k/\mathcal{I}_{r+1}\]
        equipped with $\psi_2$ corresponding to the multiplication map of the algebra, $s$ its unit and $\psi_1$ the restriction of the quotient map to the linear functions. Since $s$ does not factor through $\mathcal{F}_1$ while $\psi_1$ does it follows that we have a map
        \[\oo_S[x_1,\dots, x_n]/(x_1,\dots, x_n) \to \mathcal{F}_0/\mathcal{F}_1\cong \oo_S[x_1,\dots, x_n]/\mathcal{I}_1\]
        which is an isomorphism by rank reasons. This means that 
        \[\mathcal{I}_1\cong \mathfrak{m}:=(x_1,\dots, x_n)\subset \oo_S[x_1,\dots, x_n]. \]
        Since
        \[\psi_2(\mathcal{F}_1,\mathcal{F}_k)\subset \mathcal{F}_{k+1}\]
        this means
        \[(\mathcal{I}_1/\mathcal{I}_{r+1})\cdot (\mathcal{I}_k/\mathcal{I}_{r+1})\subset (\mathcal{I}_{k+1}/\mathcal{I}_{r+1})\]
        i.e.
        \[\mathcal{I}_1\cdot \mathcal{I}_{k}\subset \mathcal{I}_{k+1}.\]
        We conclude by induction that
        \[\mathfrak{m}^{k}=\mathcal{I}_1^k\subset \mathcal{I}_k\subset \mathcal{I}_1=\mathfrak{m}\]
        which shows that $\mathcal{I}_k\subset \oo_S[x_1,\dots, x_n]$ is punctual.
    \end{proof}
    \begin{corollary}
         Let $\underline{d}=(1,d_1,\dots, d_r)$. We have an inclusion of schemes
        \[\operatorname{naNHilb}_0^{\underline{d}}(\A^n)\cap \operatorname{NHilb}^{\underline{d}}(\A^n)\subset \operatorname{NHilb}_0^{\underline{d}}(\A^n).\]
    \end{corollary}
\begin{warning}
In general (unless $\underline{d}=(1,1,\dots,1)$) we do not have an equality in the inclusion above. For a $\C$-point $I_{r+1}\subset \cdots I_{1}\subset \C[x_1,\dots, x_n]$ to factor through $\operatorname{naNHilb}_0^{\underline{d}}(\A^n)$ means exactly that $\mathfrak{m}\subset I_1$ (and thus $\mathfrak{m}=I_1$ since they have the same colength) and that $I_1\cdot I_k\subset I_{k+1}$ for all $k$. An easy example of a punctual point not satisfying this is
\[(x^3)\subset (x)\subset \C[x]\]
in $\operatorname{NHilb}_0^{(1,2)}(\A^1)$, as $(x)\cdot(x)=(x^2)\not\subset (x^3)$. 
\end{warning}

\begin{remark} \textbf{(Relation to \cite{KazarianNAHilb})}
    One can consider a version of \cref{non-unital moduli functor} where we require $\psi_2(\mathcal{F}_{i}\otimes_S \mathcal{F}_j)\subset \mathcal{F}_{i+j}$ for all $1\leq i,j\leq r$. We denote this locus by $\operatorname{naFHilb}_0^{\underline{d}}(\A^n)$ and again, note that we get a closed immersion 
    \[\operatorname{naFHilb}_0^{\underline{d}}(\A^n)\subset \operatorname{naNHilb}^{\underline{d}}(\A^n)\]
    by adjoining a unit and that the projection map on the latter restricts to a map 
    \[\operatorname{naFHilb}_0^{(1,d_1,\dots, d_r)}(\A^n)\to \operatorname{naFHilb}_0^{(1,d_1,\dots, d_{r-1})}(\A^n).\]
    Let $(\mathcal{N}_\bullet, \Psi_2,\Psi_1)$ be the universal structure for $\operatorname{naFHilb}^{(1,\dots, d_{r-1})}$. Set $S_{r}\subset \sym^2\mathcal{N}_1$ spanned by $\mathcal{N}_i\otimes \mathcal{N}_j$ where $i+j>r$. We get an induced maps 
    \[\overline{\Psi}_2: \sym^2 \mathcal{N}_1/S_r\to \mathcal{N}_1\]
    and one can again using the same arguments as the section above show that
    \[\operatorname{naFHilb}_0^{(1,d_1,\dots, d_{r})}(\A^n)\cong \operatorname{Gr}_{\operatorname{naFHilb}_0^{(1,d_1,\dots, d_{r-1})}(\A^n)}(d_r, \operatorname{ker}(\Psi_1\oplus \overline{\Psi_2})^\vee)\]
    and that the above projection is simply the structure map. In fact, this is exactly the proof of theorem 6.3 of \cite{KazarianNAHilb} and $\operatorname{naFHilb}^{\underline{d}}(\A^n)$ agrees with the notion of the non-associative Hilbert scheme appearing in \cite{KazarianNAHilb}. Its intersection with the associative locus yields what we refer to as \textit{the filtered Hilbert scheme}, which we informally describe as the following closed locus of the punctual nested Hilbert scheme
    \[\operatorname{FHilb}_0^{\underline{d}}(\A^n)=\left\{I_{r+1}\subset I_{r}\subset\cdots \subset I_1=\mathfrak{m} \: | \: I_i\cdot I_j\subset I_{i+j}, \: 1\leq i,j\leq r \right\}\]
    where $\mathfrak{m}=(x_1,\dots, x_n)\subset \C[x_1,\dots, x_n]$.
\end{remark}

\section{Virtual localization}
\label{sec:virtual-class}
In this section we extend the standard torus action of $T=\Gm^n$ on $\operatorname{NHilb}^{\underline{d}}(\A^n)$ to the non-associative Hilbert scheme. We prove that $\mathcal{E}_{\operatorname{ass}}$ can be equipped with $T$-equivariant structure making the section $s_{\operatorname{ass}}$ $T$-equivariant. It follows that our perfect obstruction theory admits a $T$-equivariant enhancement. We then apply virtual localization and compute the local contributions for each fixed point.

\subsection{The torus action}
\begin{set}
In this section we set $T=(\C^*)^n$ an $n$-dimensional torus. We follow the notation of \cref{sec::torus} so that 
\[K_T^0(\operatorname{pt})\cong \Z[t^{\mu} \: | \: \mu \in \hat{T}] \cong  \Z[t_1^{\pm 1},\dots, t_n^{\pm 1}]\]
and 
\[A_T^*(T)\cong \operatorname{Sym}_{\Z}^\bullet(\hat{T})\cong \Z[s_1,\dots, s_n].\]
Explicitly, $t_i$ is the class of a one-dimensional $T$-representation of weight $e_i$ while $s_i$ is its first Chern class.
\end{set}

\begin{defi}
     We define a torus action on $M_{n,\underline{d}}$ by 
    \[\mathbf{t}.(\psi_1,\psi_2,v)=(\psi_1^{\mathbf{t}},\psi_2,v)\]
   where $\psi_1^{\mathbf{t}}$ denotes the map given by 
   \[(\C^n)^{\vee}\xrightarrow{\mathbf{t}^{-1}.}(\C^n)^\vee \xrightarrow{\psi_1} N_0\]
   where $T$ acts on $(\C^n)^\vee$ in the standard way. Clearly this commutes with the $P_{\overline{d}}$-action and thus induces a $T$-action on the quotient 
   \[[M_{n,\underline{d}}/P_{\underline{d}}]=\operatorname{naNHilb}^{\underline{d}}(\A^n).\]
\end{defi}

\begin{lemma}
    The universal flag bundle $\mathcal{V}_\bullet$ on $\operatorname{naNHilb}^{\underline{d}}(\A^n)$  carries natural $T$-equivariant structure such that the universal maps $\Psi_1, \Psi_2, \mathfrak{s}$ are equivariant.
\end{lemma}
\begin{proof}
    We simply define $V_\bullet= N_\bullet\otimes  \oo_{M_{n,\underline{d}}}$ on $M_{n,\underline{d}}$.  It admits a $P_{\underline{d}}$-equivariant structure identifying it with $\mathcal{V}_\bullet$ under 
    \[\operatorname{QCoh}([M_{n,\underline{d}}//P_{\underline{d}}])\simeq \operatorname{QCoh}^{P_{\underline{d}}}(M_{n,\underline{d}}).\]
    Equipping $V_\bullet$ with the fiberwise trivial $T$-equivariant structure induces the desired $T$-equivariant structure on $\mathcal{V}_\bullet$.
\end{proof}
\begin{corollary}
    The perfect obstruction theory $\mathbb{E}_{ass}\to \mathbb{L}_{\operatorname{NHilb}^{\underline{d}}(\A^n)}$ admits a $T$-equivariant refinement.
\end{corollary}
\begin{proof}
    This follows since the $T$-equvariant structure on $\mathcal{V}_\bullet$ naturally equips $\mathcal{E}_{\operatorname{ass}}$ with $T$-equivariant structure. The section $s_{\operatorname{ass}}$ is then $T$-equivariant since it is built from $\Psi_2$.
\end{proof}

\subsection{Weight calculations}
We will calculate
\[T_{\operatorname{naNHilb}^{\underline{d}}(\A^n),p}\]
and 
\[\mathcal{E}_{\operatorname{ass},p}\]
as elements of $K_0^T(\operatorname{pt})$ for $p\in \operatorname{NHilb}^{\underline{d}}(\A^n)^T$ a fixed point. We recall that the fixed locus $ \operatorname{NHilb}^{\underline{d}}(\A^n)^T$ is 0-dimensional, reduced, and proper. Its points are in bijection with flags of monomial ideals, which is again in an order reversing bijection with flags of $(n+1)$-dimensional partitions. 
\begin{set}
Let $\lambda_\bullet=(\lambda_1\subset \lambda_2\subset \cdots \subset \lambda_{r+1})$ be a flag of $(n+1)$-dimensional partitions (i.e. $\lambda_i\subset (\Z_{\geq 0})^{n}$) with $|\lambda_i|-|\lambda_{i-1}|=d_{i-1}$. Let 
\[I_{\lambda_{r+1}}\subset \cdots \subset I_{\lambda_1}\subset \C[x_1,\dots, x_n]\]
be the associated flag of monomial ideals. By abuse of notation, we will denote by
\[\lambda_\bullet \in \operatorname{NHilb}^{\underline{d}}( \A^n)^T\hookrightarrow \operatorname{naNHilb}^{\underline{d}}(\A^n)^T\]
the associated $T$-fixed point. We denote by 
$N^{\lambda}_i\in K_T^0(\operatorname{pt})$
the $K$-theory class of $I_{\lambda_{i}}/I_{\lambda_{r+1}}$where we set $I_{\lambda_{0}}=\C[x_1,\dots, x_n]$. We note that
\[N_0^{\lambda}=\oo_{\A^n, I_{\lambda_\bullet}}^{[\underline{d}]}\]
where the right-hand side is the stalk of the tautological bundle of the structure sheaf.
\end{set}
\subsubsection{The tangent space}
\begin{prop} \label{K-theory of tanget space}
Let $\underline{d}\in (\Z_{\geq 0})^{r+1}$, $n\in \N$ and let $\lambda_\bullet \in \operatorname{NHilb}^{\underline{d}}(\A^n)^T$. We have an equality
\[T_{\operatorname{naNHilb}^{\underline{d}}(\A^n),\lambda_\bullet}=(t_1+t_2+\cdots +t_n+1)N^{\lambda}_0+\operatorname{Hom}^{\operatorname{fil}}(\sym^2 N_{\bullet}^{\lambda},N_{\bullet}^\lambda)-2\operatorname{End}^{\operatorname{fil}}(N^{\lambda}_\bullet)\]
in $K_T^0(\operatorname{pt})$. Here, the superscript $\operatorname{fil}$ denotes filtration preserving maps and the filtration structure on $\sym^2 N_\bullet^{\lambda}$ is given by letting the $i$'th part be the image of $N_0^{\lambda}\otimes N_i^{\lambda}$.
\end{prop}
\begin{proof}
    Since $M_{n,\underline{d}}\to \operatorname{naNHilb}^{\underline{d}}(\A^n)$ is a $T$-equivariant principal $P_{\underline{d}}$-bundle we have an equality 
    \[T_{\operatorname{naNHilb}^{\underline{d}}(\A^n),\lambda_\bullet}=T_{M_{n,\underline{d}},p}-\mathfrak{p}_{\underline{d}}\]
    for a choice of lift $p\in M_{n,\underline{d}}$ of $\lambda_{\bullet}$. The $T$-structure on $T_{M_{n,\underline{d}},p}$ is constructed as follows. We let
    \[\rho: T\to G\]
    uniquely defined by $\rho(t).p=t.p$. Twisting the $T$-action on $M_{n,\underline{d}}$ with $(\rho)^{-1}$ yields a $T$-action for which $p$ is fixed and thus a $T$-space structure on the tangent space of $p$. The $T$-space structure on $\mathfrak{p}_{\underline{d}}$ is given by the adjoint action precomposed with $(\rho)^{-1}$. We choose the lift $p$ as follows. We can choose a basis $(v_u)$ of $N_0$ indexed by the $u\in \lambda_{r+1}$ such that $v_u\in N_{i}$ if and only if $u\notin \lambda_i$. We define
    \[\psi_2(v_{u}\otimes v_{u'})=v_{u+u'},\]
     \[v=v_{0}\]
     and
    \[\psi_1(x_i)=\left\{\begin{array}{cl}
        v_{e_i} & \text{if }e_i\in \lambda_{r+1} \\
        0 & \text{else} 
    \end{array}\right.\]
    This defines the lift $p$. Next, we calculate $\rho$. From the identity $\rho(t).p=t.p$ we immediately see that
    \[\rho(t).v_0=v_0\]
    and for $e_i\in \lambda_{r+1}$ we have
    \[\rho(t)v_{e_i}=\rho(t)\psi_1(x_i)=\psi_1^t(x_i)=\psi_1(t_ix_i)=t_iv_{e_i}.\]
    Using that $\psi_2^{\rho(t)}=\psi_2$ we further see that
    \begin{align*}
        \rho(t)v_{u+u'}&=\rho(t)\psi_2(v_u\otimes v_{u'}) \\
        &=\rho(t)\psi_2\left((\rho(t)^{-1}\rho(t)v_u)\otimes (\rho(t)^{-1}\rho(t)v_{u'})\right)\\
        &=\psi_2^{\rho(t)}\left((\rho(t)v_u)\otimes (\rho(t)v_{u'})\right)\\
        &=\psi_2\left((\rho(t)v_u)\otimes (\rho(t)v_{u'})\right).
    \end{align*}
    Combining these facts an induction argument shows that $\rho(t)v_u=t^{u}v_u$ i.e., the $T$-space structure on $N_i$ induced from $\rho^{-1}$ is isomorphic to $N_i^{\lambda}$. From this we immediately get that 
    \[\mathfrak{p}_{\underline{d}}=\operatorname{End}^{\operatorname{fil}}(N_\bullet^\lambda)\in K_T^{0}(\operatorname{pt})\]
    when $\mathfrak{p}_{\underline{d}}$ is equipped with the adjoint action composed with $\rho^{-1}$. The variety $M_{n,\underline{d}}$ is a Zariski open of the affine sub-space
    \[W=f^{-1}(\operatorname{id}_{N_0})\subset V:= \Hom((\C^n)^\vee, N_0)\oplus \operatorname{Hom}^{\operatorname{fil}}(\sym^2N_0,N_0)\oplus N_0\]
    where 
    \begin{center}
        \begin{tikzcd}
            f: V \arrow{r} & \operatorname{End}^{\operatorname{fil}}(N_\bullet) \\[-20pt] 
            (\psi_1,\psi_2,v) \arrow[mapsto]{r} & (x\mapsto \psi_2(v,x)).
        \end{tikzcd}
    \end{center}
This map is an equivariant submersion when we identify $N_\bullet$ with $N_{\bullet}^{\lambda}$. It follows that
\[T_{M_{n,\underline{d}},p}=T_{W,p}=T_{V,p}-T_{\operatorname{End}^{\operatorname{fil}}(N_\bullet),\operatorname{id}}=T_{V,p}-\operatorname{End}^{\operatorname{fil}}(N_\bullet^{\lambda})\]
from which the statement now immediately follows since 
\[T_{V,p}=(t_1+t_2+\cdots +t_n)N^{\lambda}_0+\operatorname{Hom}^{\operatorname{fil}}(\sym^2 N_{\bullet}^{\lambda},N_{\bullet}^\lambda)+N_0^{\lambda}.\]
\end{proof}
To further compute this, we introduce notation to keep track of the equivariant weights.
\begin{set}
    Let \label{Partition setup}
    \[\lambda_\bullet=(\lambda_1\subset \lambda_2\subset \cdots \subset \lambda_{r+1}\subset (\Z_{\geq 0})^n)\]
    be an $n$-dimensional nested partition of length $\underline{d}=(d_0,\dots, d_r)$ and set $d=d_0+\cdots +d_r$. In the following we choose an enumeration
    \[\mathbf{u}_0,\mathbf{u}_2,\dots ,\mathbf{u}_{d-1}\in (\Z_{\geq0}^n)\]
    of the partition $\lambda_{r+1}$ in such a way that 
    \[\{\mathbf{u}_0,\dots, \mathbf{u}_k\}\]
    is a partition for all $0\leq k\leq d-1$ and that
    \[\{\mathbf{u}_0,\dots, \mathbf{u}_{d_0+\cdots +d_{i-1}-1}\}=\lambda_i.\]
    For $0\leq k\leq d-1$ we define $w(k)=i$ if 
    \[d_0+\cdots +d_{i-2}-1<k\leq d_0+\cdots +d_{i-1}-1,\]
    i.e. 
    \[(w(0),w(1),\dots, w(d-1))=(\underbrace{0,0,\dots,0}_{d_0\text{ times}},\underbrace{1,1,\dots,1}_{d_1 \text{ times}},\dots,\underbrace{r,r\cdots, r}_{d_r\text{ times}}).\]
    It follows that $x^{\mathbf{u}_k}\in N_\bullet^{\lambda}$ is in $N_i^{\lambda}$ if and only if $w(k)\geq i$. Note also that
    \[N_0=\operatorname{span}(x^{\mathbf{u}_k} \:| \: 0\leq k\leq d-1)\]
    with $x^{\mathbf{u}_k}$ a $T$-eigenvector of weight $t^{-\mathbf{u}_k}$.
\end{set}
\begin{lemma}
    Let $\lambda_\bullet$ be an $n$-dimensional nested partition with enumeration $\mathbf{u}_0,\dots, \mathbf{u}_{d-1}$ as in \cref{Partition setup}. Then the trace of $T_{\operatorname{naNHilb}^{\underline{d}}(\A^n),\lambda_\bullet}$ is given by
    \[(t_1+t_2+\cdots+t_n)\sum_{k=0}^{d-1} t^{-\mathbf{u}_k}+\sum_{1\leq i\leq j\leq d-1} \sum_{\substack{0\leq k\leq d-1\\ w(j)\leq w(k)}}t^{\mathbf{u}_i+\mathbf{u}_j-\mathbf{u}_k}-\sum_{j=1}^{d-1}\sum_{\substack{0\leq k\leq d-1\\ w(j)\leq w(k)}}t^{\mathbf{u}_j-\mathbf{u}_k}.\]
\end{lemma}
\begin{proof}
     We have 
     \[T_{\operatorname{naNHilb}^{\underline{d}}(\A^n),\lambda_\bullet}=(t_1+t_2+\cdots +t_n+1)N^{\lambda}_0+\operatorname{Hom}^{\operatorname{fil}}(\sym^2 N_{\bullet}^{\lambda},N_{\bullet}^\lambda)-2\operatorname{End}^{\operatorname{fil}}(N^{\lambda}_\bullet)\]
     so, using the $T$-eigenvector basis $x^{\mathbf{u}_1},\dots, x^{\mathbf{u}_{d-1}}$ we see that the trace identifies with 
    \[(t_1+t_2+\cdots+t_n+1)\sum_{k=0}^{d-1} t^{-\mathbf{u}_k}+\sum_{0\leq i\leq j\leq d-1} \sum_{\substack{0\leq k\leq d-1\\ w(j)\leq w(k)}}t^{\mathbf{u}_i+\mathbf{u}_j-\mathbf{u}_k}-2\sum_{\substack{0\leq j,k\leq d-1\\ w(j)\leq w(k)}}t^{\mathbf{u}_j-\mathbf{u}_k}.\]
    We see some immediate cancellation. The "$i=0$" terms of the second sum cancels with one copy of the third sum. The last copy of the first sum corresponding to the coefficient 1 cancels with the "$j=0$" terms of the third sum. This leaves us with 
\[(t_1+t_2+\cdots+t_n)\sum_{k=0}^{d-1} t^{-\mathbf{u}_k}+\sum_{1\leq i\leq j\leq d-1} \sum_{\substack{0\leq k\leq d-1\\ w(j)\leq w(k)}}t^{\mathbf{u}_i+\mathbf{u}_j-\mathbf{u}_k}-\sum_{j=1}^{d-1}\sum_{\substack{0\leq k\leq d-1\\ w(j)\leq w(k)}}t^{\mathbf{u}_j-\mathbf{u}_k}.\]
\end{proof}
\begin{remark}
From here on the cancellation gets trickier. A term of the form $t^{\mathbf{u}_j-\mathbf{u}_k}$ cancels with either a term in the first or second sum depending on $\mathbf{u}_j\in (\Z_{\geq 0})^n$. If $\mathbf{u}_j=e_i$ is a unit vector, it will cancel with $t_it^{-\mathbf{u}_k}$ from the first sum. If it is not a unit vector, we can rewrite $\mathbf{u}_j=\mathbf{u}_{i_1}+\mathbf{u}_{i_2}$ for $1\leq i_1\leq i_2\leq j$. Since $i_2\leq j$ we have $w(i_2)\leq w(j)\leq w(k)$ and so the term cancels with $t^{\mathbf{u}_{i_1}+\mathbf{u}_{i_2}-\mathbf{u}_k}$ from the second sum. Another way to keep track of only the positive weights is via the following multiset definition.
\end{remark}
\begin{lemma} \label{recursive tangent weights}
    For a nested partition 
    \[\lambda_\bullet=(\lambda_1\subset \lambda_2\subset \cdots \subset \lambda_{r+1})\]
    of length $\underline{d}$ we recursively define multisets  
    \[S_m^T(\lambda_\bullet)\]
    for $-1\leq m\leq r$as follows. Choose an enumeration $\mathbf{u}_0,\dots, \mathbf{u}_{d-1}$ as in \cref{Partition setup}. We set 
    \[S_{-1}^T(\lambda_{\bullet})=\{e_1,\dots, e_n\}\]
    and recursively define 
    \[S_{m}^T(\lambda_{\bullet})=\left(S_{k-1}(\lambda_\bullet)\cup\{\mathbf{u}_i+\mathbf{u}_j \: |\: 1\leq i\leq j, \: w(j)=m\}\right)\big{\backslash} \{\mathbf{u}_j \: | \: 1\leq j,\: w(j)=m\}.\]
    Then the trace of $T_{\operatorname{naNHilb}^{\underline{d}}(\A^n),\lambda_\bullet}$ is given by
    \[\sum_{m=0}^r \sum_{\mathbf{v}\in \lambda_{m+1}\backslash \lambda_m} \sum_{\mathbf{u}\in S_m^T(\lambda_\bullet)} t^{\mathbf{u}-\mathbf{v}}.\]
\end{lemma}
\begin{proof}
    One easily sees by induction that 
    \[S_k^T(\lambda_\bullet)=\left(\{e_1,\dots, e_n\}\cup\{\mathbf{u}_i+\mathbf{u}_j \: |\: 1\leq i\leq j, \: w(j)\leq k\}\right)\big{\backslash} \{\mathbf{u}_j \: | \: w(j)\leq k\}\]
    and since $\mathbf{u}_k\in \lambda_{m+1}\backslash \lambda_m$ if and only if $w(k)=m$ the result now easily follows from the above calculation.
\end{proof}
Lastly, we compute the rank of the fixed part.
\begin{corollary}
    Let $\lambda_\bullet$ be an $n$-dimensional nested partition with enumeration $\mathbf{u}_0,\dots, \mathbf{u}_{d-1}$ as in \cref{Partition setup}.
    Define
    \[W_T(\mathbf{u}_k)=\left\{\begin{array}{cl}
       0  & \mathbf{u}_k=e_i \text{ for some }i=1,\dots,n \\
       \#\{(i,j)\: | \: 1\leq i \leq j\leq d-1, \mathbf{u}_i+\mathbf{u}_j=\mathbf{u}_k\}-1  & \text{else}
    \end{array}\right.\]
    for $k=1,\cdots d-1$ and $e_i$ the $i$'th unit vector. Then the trace of $T_{\operatorname{naNHilb}^{\underline{d}}(\A^n),\lambda_\bullet}^{\operatorname{fix}}$ is
    \[\sum_{k=1}^{d-1} W_T(\mathbf{u}_k).\]
\end{corollary}
\begin{proof}
The positive trivial weights of 
\[(t_1+t_2+\cdots+t_n)\sum_{k=0}^{d-1} t^{-\mathbf{u}_k}+\sum_{1\leq i\leq j\leq d-1} \sum_{\substack{0\leq k\leq d-1\\ w(j)\leq w(k)}}t^{\mathbf{u}_i+\mathbf{u}_j-\mathbf{u}_k}-\sum_{j=1}^{d-1}\sum_{\substack{0\leq k\leq d-1\\ w(j)\leq w(k)}}t^{\mathbf{u}_j-\mathbf{u}_k}.\]
     are exactly of the form $t_it^{-\mathbf{u}_k}$ for $\mathbf{u}_k=e_i$ or $t^{\mathbf{u}_i+\mathbf{u}_j-\mathbf{u}_k}$ for $\mathbf{u}_i+\mathbf{u}_j=\mathbf{u}_k$, $i\leq j$ and $w(j)\leq w(k)$. We note that since $\mathbf{u}_j\leq \mathbf{u}_k$ we have $j\leq k$ so the condition $w(j)\leq w(k)$ is redundant. For each $1\leq k$ we have exactly one negative trivial weight of the form $-t^{\mathbf{u}_k-\mathbf{u}_k}$ from the third sum. It follows that the trace of $T_{\operatorname{naNHilb}^{\underline{d}}(\A^n),\lambda_\bullet}^{\operatorname{fix}}$ is given by 
     \[\sum_{k=1}^{d-1} Q_k(t)\]
     where 
     \[Q_k(t)=t_it^{-\mathbf{u}_k}-t^{\mathbf{u}_k-\mathbf{u}_k}=0\]
     if $\mathbf{u}_k=e_i$ for some $i=1,\dots, n$ and 
     \[Q_k(t)=\sum_{\substack{1\leq i\leq j\leq d-1\\ \mathbf{u}_i+\mathbf{u}_j=\mathbf{u}_k} }t^{\mathbf{u}_i+\mathbf{u}_j-\mathbf{u}_k}-t^{\mathbf{u}_k-\mathbf{u}_k}\]
     else. It is clear that $Q_k(t)=W_T(\mathbf{u}_k)$ from which the statement follows directly
\end{proof}

\begin{corollary}
    The fixed part of the tangent space $T_{\operatorname{naNHilb}^{\underline{d}}(\A^n),\lambda_\bullet}^{\operatorname{fix}}$ depends only on $\lambda_{r+1}$ and not on the chosen nesting.
\end{corollary}
\begin{proof}
    The above formula only depends on a choice of enumeration $\mathbf{u}_0,\dots, \mathbf{u}_{d-1}$. However, the allowable choices of enumerations does depend on the nesting, so we should check that the formula is independent of the enumeration. Clearly the number of unit vectors is independent of the enumeration and for $u=\mathbf{u}_k\in \lambda_{r+1}$ not a unit vector, we see that
    \[\#\{(i,j)\: | \: 1\leq i \leq j\leq d-1, \mathbf{u}_i+\mathbf{u}_j=\mathbf{u}_k\}=\#\{\{v,w\}\: | \: v,w\in \lambda_{r+1}, v+w=u\}\]
    which is clearly independent of the choice of enumeration.
\end{proof}

\subsubsection{The associativity bundle}
We calculate the equivariant K-theory of $\mathcal{E}_{\operatorname{ass},\lambda_\bullet}$ and $\mathcal{E}_{\operatorname{ass},\lambda_\bullet}^{\operatorname{fix}}$. 
\begin{lemma}
    Let $\underline{d}\in (\Z_{\geq 0})^{r+1}$, $n\in \N$ and let $\lambda_\bullet \in \operatorname{NHilb}^{\underline{d}}(\A^n)^T$. Let $\mathcal{V}_\bullet$ be the universal flag on $\operatorname{naNHilb}^{\underline{d}}(\A^n)$. We have an equality
    \[\mathcal{V}_{\bullet,\lambda_\bullet}=N_\bullet^{\lambda}\]
    in $K_T^{0}(\operatorname{pt})$.
\end{lemma}
\begin{proof}
    Let $U\subset M_{n,\underline{d}}$ be the preimage of $\lambda_\bullet$ under the quotient map. Using the lift $p\in U$ as in the proof of \cref{K-theory of tanget space} we get a $T$-equivariant isomorphism $U\cong P_{\underline{d}}$ where $P_{\underline{d}}$ is equipped with a $T$-structure coming from $\rho$. It follows that the induced $T$-representation on $\mathcal{V}_{\bullet,\lambda}\cong N_{\bullet}$ is the one coming from the $P_{\underline{d}}$ action precomposed with $\rho^{-1}$. This again naturally identifies with $N_\bullet^{\lambda}$, see the proof of \cref{K-theory of tanget space}.
\end{proof}

\begin{corollary}
    Let $\underline{d}\in (\Z_{\geq 0})^{r+1}$, $n\in \N$ and let $\lambda_\bullet \in \operatorname{NHilb}^{\underline{d}}(\A^n)^T$. Let 
    \[N_\bullet^{+,\lambda}=N_{\bullet}^{\lambda}/(1\cdot \C)\]
    for $1\in N_{\bullet}^{\lambda}$ the image of $1\in \C[x_1,\dots, x_n]$. We have an identification 
    \[\mathcal{E}_{\operatorname{ass},\lambda_\bullet}=\Hom^{\operatorname{fil},\operatorname{alt}}(N_\bullet^{+,\lambda}\otimes N_\bullet^{+,\lambda}\otimes N_\bullet^{+,\lambda},N_\bullet^{\lambda}). \]
    where the right-hand side is trilinear maps which are alternating in the outer coordinates and takes $N_i\otimes N_j\otimes N_k$ to $N_{\operatorname{max}\{i,j,k\}}$. In particular, if we choose an enumeration $\mathbf{u}_0,\dots, \mathbf{u}_{d-1}$ as in \cref{Partition setup}, then the trace of $\mathcal{E}_{\operatorname{ass},\lambda_\bullet}$ is given by
    \[\sum_{\substack{1\leq i,j,k\leq d-1 \\ i<k}} \sum_{\substack{0\leq m\leq d-1\\ w(k),w(j)\leq w(m)}}t^{\mathbf{u}_i+\mathbf{u}_j+\mathbf{u}_k-\mathbf{u}_m}.\]
\end{corollary}
Again, it will be convenient for later notation to have a definition similar to \cref{recursive tangent weights}. 
\begin{lemma} \label{recursive bundle weights}
    We define for $0\leq m\leq r$ the set $S_m^{\operatorname{ass}}(\lambda_\bullet)$ as follows. Picking an enumeration $\mathbf{u}_0,\dots, \mathbf{u}_{d-1}$ we define 
    \[S^{\operatorname{ass}}_m(\lambda_{\bullet})=\{\mathbf{u}_i+\mathbf{u}_j+\mathbf{u}_k\: |\: 1\leq i,j,k\leq d-1, \: i<k, \: w(k),w(j)\leq w(m)\}.\]
    Then the trace of $\mathcal{E}_{\operatorname{ass},\lambda_\bullet}$ is given by 
    \[\sum_{m=0}^r \sum_{\mathbf{v}\in \lambda_{m+1}\backslash \lambda_m} \sum_{\mathbf{u}\in S_m^{\operatorname{ass}}(\lambda_\bullet)} t^{\mathbf{u}-\mathbf{v}}.\]
\end{lemma}
\begin{corollary}
     Let $\lambda_\bullet$ be an $n$-dimensional nested partition with enumeration $\mathbf{u}_0,\dots, \mathbf{u}_{d-1}$ as in \cref{Partition setup}.
    Define 
    \[W_B(\mathbf{u}_m)=\#\{(i,j,k)\in \{1,\dots ,d-1\}^{ 3}\: |\: i<k, \: \mathbf{u}_i+\mathbf{u}_j+\mathbf{u}_k=\mathbf{u}_m\}\]
    for $1\leq m \leq d-1$. Then the trace of $\mathcal{E}_{\operatorname{ass},\lambda_\bullet}^{\operatorname{fix}}$ is
    \[\sum_{m=1}^{d-1} W_B(\mathbf{u}_m)\]
\end{corollary}
\begin{proof}
    This follows almost directly from the above formula. The only thing we need to notice is that if $\mathbf{u}_i+\mathbf{u}_j+\mathbf{u}_k=\mathbf{u}_m$ for $1\leq i,j,k\leq d-1$ then $\mathbf{u}_j,\mathbf{u}_k\leq \mathbf{u}_m$ so $i,k\leq m$ hence it already holds that $w(j),w(k)\leq w(m)$.
\end{proof}
\begin{corollary}
    The fixed part of the associativity bundle $\mathcal{E}_{\operatorname{ass},\lambda_\bullet}^{\operatorname{fix}}$ at nested partition $\lambda_\bullet$ depends only on $\lambda_{r+1}$ and not on the nesting.
\end{corollary}
\begin{proof}
For any enumeration $\mathbf{u}_1,\dots, \mathbf{u}_{d-1}$ we see for $u=\mathbf{u}_m$ that
\begin{align*} W_B(\mathbf{u}_m)&=\#\{(i,j,k)\in \{1,\dots ,d-1\}^{ 3}\: |\: i<k, \: \mathbf{u}_i+\mathbf{u}_j+\mathbf{u}_k=\mathbf{u}_m\}\\
&=\#\{(v,s,w)\in \lambda_{r+1}\times \lambda_{r+1}\times \lambda_{r+1}\: |\: v\neq w, v+s+w=u\}\big{/}2
\end{align*}
and the last expression clearly only depends on $\lambda_{r+1}$.
\end{proof}
\subsection{The fixed virtual class}
The $T$-equivariant perfect obstruction theory $\mathbb{E}$ induces a virtual class on the fixed locus $\operatorname{NHilb}^{\underline{d}}(\A^n)^T$. Since the latter is zero dimensional, reduced and proper we are in the situation of \cref{fixed obstruction theory isolated points}. With that in mind we make the following definition.

\begin{defi}
    Let $\lambda_\bullet$ be a flag of partitions corresponding to a fixed point $\lambda_\bullet\in \operatorname{NHilb}^{\underline{d}}(\A^n)$. We will say that $\lambda_\bullet$ is admissible if $[\lambda_\bullet]^{\operatorname{vir}}\neq [0]$ i.e., if 
    \[\operatorname{rk}\mathcal{E}_{\operatorname{ass},[\lambda_\bullet]}^{\operatorname{fix}}=\operatorname{rk}T_{\operatorname{naNHilb}^{\underline{d}}(\A^n),[\lambda_\bullet]}^{\operatorname{fix}}.\]
\end{defi}
We recall that if we choose an enumeration $\mathbf{u}_0,\dots, \mathbf{u}_{d-1}$ of $\lambda_\bullet$ as in \cref{Partition setup} then we have 
\[\operatorname{rk}\mathcal{E}_{\operatorname{ass},\lambda_\bullet}^{\operatorname{fix}}=\sum_{m=0}^{d-1}W_B(\mathbf{u}_m) \text{ and } \operatorname{rk}T_{\operatorname{naNHilb}^{\underline{d}}(\A^n),\lambda_\bullet}^{\operatorname{fix}}=\sum_{m=0}^{d-1}W_T(\mathbf{u}_m) \]

where
\[W_T(\mathbf{u}_m)=\left\{\begin{array}{cl}
       0  & \mathbf{u}_m=e_i \text{ for some }i=1,\dots,n \\
       \#\{(i,j)\: | \: 1\leq i \leq j\leq d-1, \mathbf{u}_i+\mathbf{u}_j=\mathbf{u}_m\}-1  & \text{else}
    \end{array}\right.\]
and 
\[W_B(\mathbf{u}_m)=\#\{(i,j,k)\in \{1,\dots ,d-1\}^{ 3}\: |\: i<k, \: \mathbf{u}_i+\mathbf{u}_j+\mathbf{u}_k=\mathbf{u}_m\}.\]
We also note that admissibility of $\lambda_{\bullet}$ only depends on $\lambda_{r+1}$.

\begin{lemma}
    Let $\lambda_\bullet$ be an nested partition with chosen enumeration $\mathbf{u}_0,\dots, \mathbf{u}_{d-1}$ as in \cref{Partition setup}. Then 
    \[W_T(\mathbf{u}_m)\leq W_B(\mathbf{u}_m)\]
    for $m=1,\dots, d-1$.
\end{lemma}
\begin{proof}
We can assume that $\mathbf{u}_m$ is not a unit vector, since in that case we simply get $0\leq 0$. For each $i$ let 
\[\phi(i)=\min \left\{j \big|  \mathbf{u}_j \text{ is a unit vector and }\mathbf{u}_j\leq \mathbf{u}_i\right\}.\]
Note that $\phi(m)<m$ since $\mathbf{u}_m$ is not a unit vector. Let $\kappa$ be the map $\kappa(\mathbf{u}_i)=i$ be the map extracting the index. It will suffice to prove that
\begin{center}
    \begin{tikzcd}
        \left\{(i,j,k)\big|i<k,\mathbf{u}_i+\mathbf{u}_j+\mathbf{u}_k=\mathbf{u}_m \right\}\arrow{r} & \left\{(i,j)|\phi(m)<i\leq j,\mathbf{u}_i+\mathbf{u}_j=\mathbf{u}_m\right\} \\[-20pt]
        (i,j,k) \arrow[mapsto]{r} & \left(\min\{\kappa(\mathbf{u}_i+\mathbf{u}_j),k\},\max\{\kappa(\mathbf{u}_i+\mathbf{u}_j),k\}\right)
    \end{tikzcd}
\end{center}
is a surjection. Indeed, the cardinality of the codomain corresponds to 
\[\#\left\{(i,j)|i\leq j,\mathbf{u}_i+\mathbf{u}_j=\mathbf{u}_m\right\}-1\]
since we have removed the single element of the set corresponding to the pair 
\[(\kappa(\mathbf{u}_{\phi(m)}),\kappa(\mathbf{u}_m-\mathbf{u}_{\phi(m)})).\]
To show surjectivity we let $\mathbf{u}_i+\mathbf{u}_j=\mathbf{u}_m$ with $\phi(m)<i\leq j$. We must have either $\mathbf{u}_{\phi(m)}\leq \mathbf{u}_i$ or $\mathbf{u}_{\phi(m)}\leq \mathbf{u}_j$. In the first case we let $\mathbf{u}_k=\mathbf{u}_i-\mathbf{u}_{\phi(m)}$, then clearly $(i,j)$ is the image of $(\phi(m),k,j)$. Similarly in the second case we can just take $\mathbf{u}_k=\mathbf{u}_j-\mathbf{u}_{\phi(m)}$ then $(i,j)$ is in the image of $(\phi(m),j,i)$.
\end{proof}

\begin{prop}
Let $\lambda_{\bullet}$ be a flag of partitions. Then $\lambda_{\bullet}$ is admissible if and only if the following conditions hold.
        \begin{enumerate}
        \item For all $1\leq k\leq n$ and $\alpha>0$ we have 
        \[\alpha e_k\in \lambda_{r+1} \implies \alpha\leq 4.\]
        \item For all $1\leq k_1<k_2\leq n$ and all $\alpha_1,\alpha_2>0$ we have 
        \[\alpha_1e_{k_1}+\alpha_2e_{k_2}\in \lambda_{r+1} \implies \alpha_1+\alpha_2\leq 3.\]
        \item For all $m\geq 3$ and $1\leq k_1<k_2<\dots <k_m \leq m $ and all $\alpha_1,\dots, \alpha_m>0$ we have
        \[\sum_{i=1}^m \alpha_i e_{k_i}\notin \lambda_{r+1}.\]
    \end{enumerate} 
\end{prop}
\begin{proof}
    Choose an enumeration $\mathbf{u}_0,\dots \mathbf{u}_{d-1}$ as in \cref{Partition setup}. One can then explicitly check that for
    \[\mathbf{u}_m\in \{\alpha e_k, \alpha_1e_{k_1}+\alpha_2e_{k_2} \: | \: 1\leq k,k_1,k_2\leq n, \: \alpha\leq 4, \: \alpha_1+\alpha_2\leq 3\}\]
    we have $W_T(\mathbf{u}_m)=W_B(\mathbf{u}_m)$. This shows that these partitions are admissible. One can further check that 
    \begin{align*}
        W_T(5e_k)&< W_B(5e_k), \\
        W_T(\alpha_1e_{k_1}+\alpha_2e_{k_2})&< W_B(\alpha_1e_{k_1}+\alpha_2e_{k_2}), \\
        W_T(e_{k_1}+e_{k_2}+e_{k_3}) &< W_B(e_{k_1}+e_{k_2}+e_{k_3})
    \end{align*}
    for $\alpha_1+\alpha_2=4$. It follows that any partition not contained in
    \[\{\alpha e_k, \alpha_1e_{k_1}+\alpha_2e_{k_2} \: | \: 1\leq k,k_1,k_2\leq n, \: \alpha\leq 4, \: \alpha_1+\alpha_2\leq 4\}\]
    is non-admissible.
\end{proof}

\begin{remark}The theorem states that an admissible nested partition is admissible if and only if it does not contain any vector with 3 or more non-zero coordinates and for every choice of two coordinates the projected partition must be contained in the boxes corresponding to figure 2 below.
\begin{figure}[ht]
    \centering
\begin{ytableau}
\none  & & \none& \none& \none& \none  \\
\none  & & \none& \none& \none& \none  \\
\none  & & & \none& \none& \none  \\
\none  & & & & \none& \none  \\
\none  & & & & & \\
\end{ytableau}
\caption{}

\end{figure}
\end{remark}
\subsection{The localization formula}
We now apply the virtual localization formula of \cite{localization}. As explained in \cref{fixed obstruction theory isolated points} the formula yields
\[[\operatorname{NHilb}^{\underline{d}}(\A^n)]^{\operatorname{vir}}=\sum_{\lambda_\bullet}[\lambda_\bullet ]\cap\frac{e^T(\mathcal{E}_{\operatorname{ass},\lambda_\bullet})}{e^T(T_{\operatorname{naNHilb}^{\underline{d}}(\A^n),\lambda_{\bullet}})}\]
where the sum is over the admissible partitions. The above calculations now immediately give us the formula below.
\begin{remark}
    We use the following notation
\[\widetilde{\prod}_{\gamma \in \Gamma} \gamma=\prod_{0\neq \gamma\in \Gamma}\gamma. \]
for non-zero products in the formulas to come. Note that if $V$ is a $T$-representation with weights $\operatorname{Wt}(V)\subset \Z^n\cong \hat{T}$ then 
\[e^T(V^{\operatorname{mov}})=\widetilde{\prod}_{\mathbf{u}\in \operatorname{Wt}(V)}\mathbf{u}(\underline{s})\in A_T^*(\operatorname{pt})\cong \Z[s_1,\dots, s_n]\]
where we for $\mathbf{u}=(u_1,\dots, u_n)$ define
\[\mathbf{u}(\underline{s})=u_1s_1+\cdots+u_ns_n.\]
\end{remark}
\begin{theorem}
    Let $\underline{d}=(d_0,\dots,d_r)\in \Z^{r+1}_{\geq 0}$, $d=d_0+\cdots +d_r$ and $n\in \Z_{\geq 0}$. In the Chow group $A_*^T(\operatorname{NHilb}^{\underline{d}}(\A^n))_{\operatorname{loc}}$ localized in linear functions of $s_1,\dots, s_n$ the virtual fundamental class $[\operatorname{NHilb}^{\underline{d}}(\A^n)]^{\operatorname{vir}}$ equals
    \[\sum_{\lambda_\bullet} [\lambda_\bullet]\cap\frac{\prod_{\substack{1\leq i,j,k\leq d-1 \\ i<k}} \widetilde{\prod}_{\substack{0\leq m\leq d-1\\ w(k),w(j)\leq w(m)}}(\mathbf{u}_i+\mathbf{u}_j+\mathbf{u}_k-\mathbf{u}_m)(\underline{s})\widetilde{\prod}_{j=1}^{d-1}\prod_{\substack{0\leq k\leq d-1\\ w(j)\leq w(k)}}\mathbf{u}_j-\mathbf{u}_k(\underline{s})}{\prod_{i=1}^n\widetilde{\prod}_{k=0}^{d-1}(e_i-\mathbf{u}_k)(\underline{s})\prod_{1\leq i\leq j\leq d-1} \widetilde{\prod}_{\substack{0\leq k\leq d-1\\ w(j)\leq w(k)}}(\mathbf{u}_i+\mathbf{u}_j-\mathbf{u}_k)(\underline{s})}.\]
    Here the sum is over all admissible $(n-1)$-dimensional nested partitions $\lambda_\bullet$ of length $\underline{d}$ and where we for each $\lambda_{\bullet}$ have a chosen enumeration $\mathbf{u}_0,\dots, \mathbf{u}_{d-1}$ as in \cref{Partition setup}.  Alternatively, using the multisets of \cref{recursive tangent weights} and \cref{recursive bundle weights} we have
\[[\operatorname{NHilb}^{\underline{d}}(\A^n)]^{\operatorname{vir}}=\sum_{\lambda_\bullet}[x_{\lambda_\bullet}]\cap \prod_{m=0}^r \prod_{\mathbf{v}\in \lambda_{m+1}\backslash \lambda_m}\frac{\prod_{\substack{\mathbf{u}\in S_m^{ass}(\lambda_\bullet)\\ \mathbf{u}\neq \mathbf{v}}} (\mathbf{u}-\mathbf{v})(\underline{s})}{\prod_{\substack{\mathbf{u}\in S_m^T(\lambda_\bullet)\\ \mathbf{u}\neq \mathbf{v}}} (\mathbf{u}-\mathbf{v})(\underline{s})}\]
where the sum is again over admissible nested partitions $\lambda_\bullet$ of length $\underline{d}$.
\end{theorem}
\begin{exmp}
    Let us consider the (non-nested) Hilbert scheme $\operatorname{Hilb}^{d+1}(\A^3)$ of $d$ points on $\A^n$. Then the formula reads
    \[\sum_{\lambda_\bullet} [x_{\lambda_\bullet}]\cap\frac{\prod_{\substack{1\leq i,j,k\leq d-1 \\ i<k}} \prod_{m=0}^{d-1}(\mathbf{u}_i+\mathbf{u}_j+\mathbf{u}_k-\mathbf{u}_m)(\underline{s})\prod_{j=1}^{d-1}\prod_{k=0}^{d-1}(\mathbf{u}_j-\mathbf{u}_k)(\underline{s})}{\prod_{i=1}^n\prod_{k=0}^{d-1}(e_i-\mathbf{u}_k)(\underline{s})\prod_{1\leq i\leq j\leq d-1} \prod_{k=0}^{d-1}(\mathbf{u}_i+\mathbf{u}_j-\mathbf{u}_k)(\underline{s})}\]
    Let us calculate this explicitly for $d=3$, i.e., the Hilbert scheme of three points. We have two types of partitions. They are either given by $\{0,e_i,2e_i\}$ or $\{0,e_i,e_j\}$ for $i\neq j$ and are both admissible. For the first type of partition, we get a contribution of 
    \[\frac{4s_i3s_i2s_i5s_i4s_i3s_is_i(-s_i)2s_is_i}{2s_is_i3s_i2s_is_i4s_i3s_i2s_i\prod_{k=1}^n s_k(s_k-2s_i)\prod_{k\neq i}(s_k-s_i)}\]
    which simplfies to 
    \[\frac{-10s_i^2}{\prod_{k=1}^n s_k(s_k-2s_i)\prod_{k\neq i}(s_k-s_i)}=\frac{10}{\prod_{k\neq i}s_k(s_k-s_i)(s_k-2s_i)}\]
    For the other type of partition, we get 
    \[\frac{(2s_i+s_j)(s_i+s_j)2s_i(s_i+2s_j)2s_j(s_i+s_j)s_i(s_i-s_j)s_j(s_j-s_i)}{2s_is_i(2s_i-s_j)(s_i+s_j)s_js_i2s_j(2s_j-s_i)s_j\prod_{k=1}^ns_k\prod_{k\neq i}(s_k-s_i)\prod_{k\neq j}(s_k-s_j)}\]
    which can be reduced to 
     \[\frac{(2s_i+s_j)(s_i+2s_j)(s_i+s_j)}{s_i(2s_i-s_j)s_j(2s_j-s_i)\prod_{k=1}^ns_k\prod_{k\neq i,j}(s_k-s_i)(s_k-s_j)}.\]
     We directly see that the virtual dimension is given by $3n-3$. Let us specialize to $n=3$ with a virtual dimension of $6$. Consider the tautological bundle $\oo^{[3]}_{\A^3}$ and let $c_2=c_2^T((\oo^{[3]}_{\A^3})^\vee)$ the second equivariant Chern class of the dual bundle. Let us integrate $c_2^3$ against the virtual class. For $\lambda=\{\mathbf{u}_0,\mathbf{u}_1,\mathbf{u}_3\}$ the Chern roots of $(\oo^{[3]}_{\A^3})^\vee$  are $\mathbf{u}_0(\underline{s}),\mathbf{u}_1(\underline{s}),\mathbf{u}_3(\underline{s})$. It follows that for $\lambda=\{0,e_i,2e_i\}$ we have 
     \[c_2^3|_{\lambda}=(2s_i^2)^3=8s_i^6\]
     while for $\lambda=\{0,e_i,e_j\}$ we have 
     \[c_2|_{\lambda}=(s_is_j)^3=s_i^3s_j^3.\]
     For $\{i,j,k\}=\{1,2,3\}$ the contribution at the partition $\{0,e_i,2e_i\}$ is then 
     \[\frac{80s_i^6}{s_j(s_j-s_i)(s_j-2s_i)s_k(s_k-s_i)(s_k-2s_i)}\]
     while the contribution at $\{0,e_i,e_j\}$ is 
     \[\frac{(2s_i+s_j)(s_i+2s_j)(s_i+s_j)s_i^3s_j^3}{s_i^2(2s_i-s_j)s_j^2(2s_j-s_i)s_k(s_k-s_i)(s_k-s_j)}=\frac{(2s_i+s_j)(s_i+2s_j)(s_i+s_j)s_is_j}{(2s_i-s_j)(2s_j-s_i)s_k(s_k-s_i)(s_k-s_j)}.\]
     Summing over all 3 choices for each of the two types of partitions we get
     \[\int_{[\operatorname{Hilb}^3(\A^3)]^{\operatorname{vir}}}c_2^T(\oo_{\A^3}^{[3]})^3=20\frac{(s_1+s_2+s_3)^3}{s_1s_2s_3}-31\frac{(s_1s_2+s_1s_3+s_2s_3)(s_1+s_2+s_3)}{s_1s_2s_3}+11.\]
     In particular, if we restrict to the Calabi-Yau torus $T_{\operatorname{CY}}=\{(t_1,t_2,t_3)\in T\: |\: t_1t_2t_3=1\}$ the virtual integral evaluates to $11$.
\end{exmp}

\section{Iterated residue formula}
\label{sec:nil-fil}
In this section we choose a dimension vector of the form 
\[\underline{d}=(1,d_1,\dots, d_{r})\in (\Z_{\geq 0})^{r+1}.\]
We let $\hat{d}=(d_1,\dots, d_r)\in (\Z_{\geq 0})^{r}$ so that $\underline{d}=(1,\hat{d})$.
We show how tautological virtual integrals on $\operatorname{NHilb}^{\underline{d}}(\A^n)$ can in some cases be expressed in terms of virtual integrals defined via $\operatorname{naNHilb}^{\underline{d}}_0(\A^n)$. We then develop an iterated residue formula for these latter types of integrals. The key to proving this formula is a partial resolution 
\[\widetilde{\operatorname{naNHilb}^{\underline{d}}_0}(\A^n)\to \operatorname{naNHilb}^{\underline{d}}_0(\A^n)\]
which fibers over the flag variety $\operatorname{Flag}_{\hat{d}}(\C^n)$.
\subsection{Reducing the integral to the nilpotently filtered locus}
Let $(\mathcal{V}_\bullet,\psi_2,s,\psi_1)$ be the universal structure on $\operatorname{naNHilb}^{\underline{d}}(\A^n)$. Let $\mathcal{Q}$ be the quotient of $\mathcal{H}om^{\operatorname{fil}}(\sym^2\mathcal{V}_1,\mathcal{V}_1)$ by the subsheaf of maps which takes $\mathcal{V}_1\otimes \mathcal{V}_k$ to $\mathcal{V}_{k+1}$. We let 
\[\mathcal{E}_{\operatorname{punct}}=\mathcal{H}om((\oo^n)^{\vee},\mathcal{V}_0/\mathcal{V}_1)\oplus \mathcal{Q}\]
and note that $\operatorname{naNHilb}_0^{\underline{d}}(\A^n)\subset \operatorname{naNHilb}^{\underline{d}}(\A^n)$ identifies with the zero locus of a ($T$-equivariant) section of $\mathcal{E}_{\operatorname{punct}}$. Since $\operatorname{naNHilb}_0^{\underline{d}}(\A^n)$ is smooth the embedding is regular with normal bundle $\mathcal{E}_{\operatorname{punct}}|_{\operatorname{naNHilb}_0^{\underline{d}}(\A^n)}$. This immediately allows us to compute the tangent space.
\begin{lemma}
Let $\lambda_\bullet$ be a nested partition such that the corresponding fixed point is in $\operatorname{naNHilb}_{0}^{\underline{d}}(\A^n)$. Let $\mathbf{u}_0,\dots, \mathbf{u}_{d-1}$ be an enumeration of $\lambda_\bullet$ in the sense of \cref{Partition setup}. Then the trace of $T_{\operatorname{naNHilb}_0^{\underline{d}}(\A^n),\lambda_\bullet}$ is given by
   \[(t_1+t_2+\cdots+t_n)\sum_{k=1}^{d-1} t^{-\mathbf{u}_k}+\sum_{1\leq i\leq j\leq d-1} \sum_{\substack{1\leq k\leq d-1\\ w(j)< w(k)}}t^{\mathbf{u}_i+\mathbf{u}_j-\mathbf{u}_k}-\sum_{j=1}^{d-1}\sum_{\substack{1\leq k\leq d-1\\ w(j)\leq w(k)}}t^{\mathbf{u}_j-\mathbf{u}_k}.\]
\end{lemma}
\begin{proof}
    It is clear that
    \[T_{\operatorname{naNHilb}_0^{\underline{d}}(\A^n)}=T_{\operatorname{naNHilb}^{\underline{d}}(\A^n)}-\mathcal{E}_{\operatorname{punct}}.\]
    Since $d_0=1$ we see that
    \[N_0^\lambda/N_1^\lambda=\operatorname{span}_{\C}(x^{\mathbf{u}_0})\]
    which is the trivial $T$-representation. We conclude
    \[\mathcal{H}om((\oo^n)^\vee, \mathcal{V}_0/\mathcal{V}_1)|_{\lambda_\bullet}=\sum_{i=1}^nt_i.\]
    Furthermore, we get 
    \begin{align*}\mathcal{Q}_{\lambda_\bullet}&=\sum_{1\leq i\leq j\leq d-1}\sum_{\substack{1\leq k\leq d-1\\ w(j)\leq w(k)}}t^{\mathbf{u}_i+\mathbf{u}_j-\mathbf{u}_k}-\sum_{1\leq i\leq j\leq d-1}\sum_{\substack{1\leq k\leq d-1 \\ w(j)+1\le w(k)}}t^{\mathbf{u}_i+\mathbf{u}_j-\mathbf{u}_k} \\
    &=\sum_{1\leq i\leq j\leq d-1}\sum_{\substack{1\leq k\leq d-1\\ w(j)= w(k)}}t^{\mathbf{u}_i+\mathbf{u}_j-\mathbf{u}_k}.
    \end{align*}
    Recall that for $d_0=1$ we have that $T_{\operatorname{naNHilb}^{\underline{d}}(\A^n),\lambda_\bullet}$ equals
    \[(t_1+t_2+\cdots+t_n)\sum_{k=0}^{d-1} t^{-\mathbf{u}_k}+\sum_{1\leq i\leq j\leq d-1} \sum_{\substack{1\leq k\leq d-1\\ w(j)\leq w(k)}}t^{\mathbf{u}_i+\mathbf{u}_j-\mathbf{u}_k}-\sum_{j=1}^{d-1}\sum_{\substack{1\leq k\leq d-1\\ w(j)\leq w(k)}}t^{\mathbf{u}_j-\mathbf{u}_k}.\]
    From this, we see that $T_{\operatorname{naNHilb}^{\underline{d}}(\A^n),\lambda_\bullet}-\mathcal{E}_{\operatorname{punct},\lambda_\bullet}$ is equal to 
   \[(t_1+t_2+\cdots+t_n)\sum_{k=1}^{d-1} t^{-\mathbf{u}_k}+\sum_{1\leq i\leq j\leq d-1} \sum_{\substack{1\leq k\leq d-1\\ w(j)< w(k)}}t^{\mathbf{u}_i+\mathbf{u}_j-\mathbf{u}_k}-\sum_{j=1}^{d-1}\sum_{\substack{1\leq k\leq d-1\\ w(j)\leq w(k)}}t^{\mathbf{u}_j-\mathbf{u}_k}.\]
  
\end{proof}

Let us denote by $\operatorname{NHilb}^{\underline{d}}_{\operatorname{nil-fil}}(\A^n)\subset \operatorname{naNHilb}_0^{\underline{d}}(\A^n)$ the intersection of the punctual locus and associativity locus. Note that if the dimension vector is full, then this agrees with the punctual Hilbert scheme of points, but in general it will be a closed subscheme of the latter. Note that $\operatorname{NHilb}_{\operatorname{nil-fil}}^{\underline{d}}(\A^n)$ is cut out of $\operatorname{naNHilb}_0^{\underline{d}}(\A^n)$ via a section of $\mathcal{E}_{\operatorname{ass}}|_{\operatorname{naNHilb}_0^{\underline{d}}(\A^n)}$. We can therefore equip it with a perfect obstruction theory. Note that we are in the situation of \cref{double section zero locus} and so the inclusion 
\[\operatorname{NHilb}^{\underline{d}}_{\operatorname{nil-fil}}(\A^n)\hookrightarrow \operatorname{NHilb}^{\underline{d}}(\A^n)\]
admits a relative perfect obstruction theory from $\mathcal{E}_{\operatorname{punct}}$ compatible with the two previously defined perfect obstruction theories. We naturally want to apply \cref{relative virtual integrals}, so we try to describe the fixed-point locus of $\operatorname{NHilb}^{\underline{d}}_{\operatorname{nil-fil}}(\A^n)$. We note that for $\underline{d}=(1,1,\dots, 1)$ we get the punctual Hilbert scheme, which contains the full fixed-point locus. In general, we get the following.

\begin{lemma} \label{nil-fil monomial}
   Let
    \[I_{r+1}\subset \cdots \subset I_1\subset I_0= \C[x_1,\dots, x_n]\]
be a flag of monomial ideals with corresponding flag of partitions
    \[\lambda_1\subset \lambda_2\subset \cdots \subset \lambda_{r+1} \subset (\Z_{\geq 0})^n.\]
Then the corresponding fixed point of $\operatorname{NHilb}^{\underline{d}}(\A^n)$ factors through $\operatorname{NHilb}_{\operatorname{nil-fil}}^{\underline{d}}(\A^n)$ if and only if 
for all $k\geq 1$, $\mathbf{u}\in \lambda_{k+1}\backslash \lambda_{k}$ and $i=1,\dots, n$ we have
\[\mathbf{u}+e_i\notin \lambda_{k+1}.\]
\end{lemma}
\begin{proof}
Note first that $\operatorname{dim}_{\C}I_k/I_{k+1}=d_k$. In particular since $\dim_{\C}I_0/I_1=1$ we must have 
    \[I_1=\mathfrak{m}=(x_1,\dots, x_n)\]
    so that
    \[\lambda_1=\{\mathbf{0}\}.\]
    We note that the corresponding fixed point is in $\operatorname{NHilb}_{\operatorname{nil-fil}}^{\underline{d}}(\A^n)$ if and only if 
    \[I_1\cdot I_k\subset I_{k+1}\]
    for all $k$. This means that for all $k\geq 0$ and all $\mathbf{u}\notin \lambda_k$ and $\mathbf{v}\notin \lambda_1$ we have 
    \[\mathbf{u}+\mathbf{v}\notin \lambda_{k+1}.\]
    This statement is trivially true if $\mathbf{u}\notin \lambda_{k+1}$ so we can restrict to $\mathbf{u}\in \lambda_{k+1}\backslash\lambda_k$. Since $\mathbf{v}\notin \lambda_1$ if and only if $\mathbf{v}\neq 0$ we see that it is sufficient to consider $\mathbf{v}=e_i$ for $i=1,\dots, n$.
\end{proof}
We can now apply \cref{relative virtual integrals} to get the following result. 
\begin{prop}
Let $\alpha\in A_{\operatorname{NHilb}_{\operatorname{nil-fil}}^{\underline{d}}(\A^n),T}^*(\operatorname{NHilb}^{\underline{d}}(\A^n))$ be a local $T$-equivariant Chow cohomology class supported on the nilpotently filtered locus. Then 
\[\int_{[\operatorname{NHilb}^{\underline{d}}(\A^n)]^{\operatorname{vir}}} \alpha= \int_{[\operatorname{NHilb}_{\operatorname{nil-fil}}^{\underline{d}}(\A^n)]^{\operatorname{vir}}}\frac{ \alpha|_{\operatorname{NHilb}_{\operatorname{nil-fil}}^{\underline{d}}(\A^n)}}{e^T(\mathcal{E}_{\operatorname{punct}}|_{\operatorname{NHilb}_{\operatorname{nil-fil}}^{\underline{d}}(\A^n)})}.\]
If $\underline{d}=(1,1,\dots, 1)$ then
\[j_*\left(\frac{[\operatorname{NHilb}_0^{r+1}(\A^n)]^{\operatorname{vir}}}{e^T(\mathcal{E}_{\operatorname{punct}}|_{\operatorname{NHilb}_0^{r+1}(\A^n)})}\right)=[\operatorname{NHilb}^{r+1}(\A^n)]^{\operatorname{vir}}\]
and in particular for any $\alpha\in A_T^*(\operatorname{NHilb}^{r+1}(\A^n))$ we get
\[\int_{[\operatorname{NHilb}^{r+1}(\A^n)]^{\operatorname{vir}}} \alpha= \int_{[\operatorname{NHilb}_0^{r+1}(\A^n)]^{\operatorname{vir}}}\frac{ \alpha|_{\operatorname{NHilb}_0^{r+1}(\A^n)}}{e^T(\mathcal{E}_{\operatorname{punct}}|_{\operatorname{NHilb}_0^{r+1}(\A^n)})}.\]

\end{prop}

\subsection{The partial resolution}

We construct a partial resolution given as a birational map 
\[\tau: \widetilde{\operatorname{naNHilb}_0^{\underline{d}}}(\A^n)\to \operatorname{naNHilb}_0^{\underline{d}}(\A^n) \]
such that the domain fibers over the flag variety
\[\mu:\widetilde{\operatorname{naNHilb}_0^{\underline{d}}}(\A^n)\to \operatorname{Flag}_{\hat{d}}(\C^n). \]
Using Atiyah-Bott localization on the flag variety this allows us to rewrite virtual integrals $ \operatorname{NHilb}_{\operatorname{nil-fil}}^{\underline{d}}(\A^n)$ as an iterated residue formula depending on virtual integrals on the fibers of $\mu$. We first define the variety. 

\begin{defi}
Let $\underline{d}=(1,\hat{d})=(1,d_1,\dots, d_r)\in \Z_{\geq 0}^{r+1}$ and $n\geq d_1+\cdots +d_r$. Let $\Psi_1: (\oo^n)^{\vee} \to \mathcal{V}_1$  be the universal map of $\operatorname{naNHilb}_0^{\underline{d}}(\A^n)$ and let 
\[(\oo^n)^{\vee} \twoheadrightarrow \mathcal{F}_1 \supset \mathcal{F}_2\supset \cdots \supset \mathcal{F}_{r+1}=0\]
be the dualized universal flag on the flag variety $\operatorname{Flag}_{\hat{d}}(\C^n)$. We let $\widetilde{\operatorname{naNHilb}_0^{\underline{d}}}(\A^n)$ be the closed locus of $\operatorname{naNHilb}_0^{\underline{d}}(\A^n)\times \operatorname{Flag}_{\hat{d}}(\C^n)$ where $\Psi_1$ factors through a (necessarily unique) map 
\[\mathcal{F}_1\to \mathcal{V}_1\]
which respects the flag structures. We denote by 
\begin{center}
    \begin{tikzcd}
        \widetilde{\operatorname{naNHilb}_0^{\underline{d}}}(\A^n) \arrow{d}{\tau} \arrow{r}{\mu} & \operatorname{Flag}_{\hat{d}}(\C^n) \\
        \operatorname{naNHilb}_0^{\underline{d}}(\A^n)&
    \end{tikzcd}
\end{center}
the projections to the non-associative Hilbert scheme and the flag variety. 
\end{defi}
\begin{lemma}
    The map $\mu:\widetilde{\operatorname{naNHilb}_0^{\underline{d}}}(\A^n) \to \operatorname{Flag}_{\hat{d}}(\C^n)$ is smooth. 
\end{lemma}
\begin{proof}
For each $1\leq k \leq r$ let 
\[X_k\subset \operatorname{naNHilb}_0^{(1,d_1,\dots, d_k)}(\A^n)\times \operatorname{Flag}_{\hat{d}}(\C^n)\]
be the pullback of $\widetilde{\operatorname{naNHilb}^{(1,d_1,\dots, d_k)}_0}(\A^n)$ under $1\times q_k$ where $q_k:\operatorname{Flag}_{\hat{d}}(\C^n)\to \operatorname{Flag}_{(d_1,\dots, d_k)}(\C^n)$. Projection on the non-associative Hilbert scheme clearly induce maps
\[\widetilde{\operatorname{naNHilb}_0^{\underline{d}}}(\A^n)=X_r\to X_{r-1}\to \cdots \to X_1.\]
On $X_k$ we have a universal flag 
\[(\oo^n)^{\vee}\twoheadrightarrow \mathcal{F}_1\supset \cdots \supset \mathcal{F}_{r+1}=0\]
from $\operatorname{Flag}_{\hat{d}}(\C^n)$ and a universal flag
    \[\mathcal{V}_1\supset \cdots \supset \mathcal{V}_{k+1}=0\]
    with universal maps $\Psi_1,\Psi_2$ such that 
    \[\Psi_1: (\oo^n)^{\vee}\to \mathcal{V}_1\]
    factors through $(\oo^n)^\vee \to \mathcal{F}_1\to \mathcal{F}_1/\mathcal{F}_{k+1}$. If we let $\widetilde{\Psi}_1: \mathcal{F}_1/\mathcal{F}_{k+1}\to \mathcal{V}_1$ be the induced map then a proof completely analogous to that of \cref{projective tower structure} then shows that
    \[X_{k+1}\to X_k\]
    identifies with the structure map of the Grassmanian bundle
    \[\operatorname{Gr}_{X_k}(d_r,(\operatorname{ker}\widetilde{\Psi}_1\oplus \Psi_2)^\vee).\]
    In particular, $X_{k+1}\to X_k$ is smooth and it will suffice to show that $X_1\to \operatorname{Flag}_{\hat{d}}(\C^n)$ is smooth. We note that 
    \[\operatorname{naNHilb}_0^{(1,d_1)}(\A^n)=\operatorname{Gr}_{\operatorname{pt}}(d_1,\C^n)\]
    with universal bundle $\mathcal{V}_1$ coinciding with the dual universal bundle of the Grasmannian and universal maps given by $\Psi_2=0$ and $\Psi_1: (\oo^n)^{\vee}\to \mathcal{V}_1$ the universal map of the Grassmannian. We see that $X_1\subset \operatorname{Gr}_{\operatorname{pt}}(d_1,\C^n)\times \operatorname{Flag}_{\hat{d}}(\C^n)$ is now simply the locus where $\Psi_1$ factors through $\mathcal{F}_1/\mathcal{F}_2$ so that the map 
    \[X_1\to \operatorname{Flag}_{\hat{d}}(\C^n)\]
    simply identifies with the structure map of 
    \[\operatorname{Gr}_{\operatorname{Flag}_{\hat{d}}}(d_1,\mathcal{F}_1/\mathcal{F}_2)\]
    which is again smooth.
\end{proof}

\begin{corollary}
    The scheme $\widetilde{\operatorname{naNHilb}_0^{\underline{d}}}(\A^n)$ is smooth and projective.
\end{corollary}
\begin{proof}
    Smoothness follows from the above. Projectiveness follows since it is a closed subscheme of a product of projective schemes.
\end{proof}

\begin{lemma}
    The map $\tau: \widetilde{\operatorname{naNHilb}_0^{\underline{d}}}(\A^n)\to \operatorname{naNHilb}_0^{\underline{d}}(\A^n)$ is projective birational. 
\end{lemma}
\begin{proof}
    Projectiveness is immediate since it is a morphism of projective schemes. Let $U\subset \operatorname{naNHilb}_0^{\underline{d}}(\A^n) $ be the Zariski open subset where $\Psi_1$ is surjective. Note that it is non-empty, containing for example the Porteous-type points, i.e., monomial ideals where the corresponding partition only contains unit vectors. We construct an inverse of $\tau$ on $U$. Explicitly, we define $U\to \operatorname{Flag}_{\hat{d}}(\C^n)$ via the flag dual to 
    \[\Psi_1:(\oo^n)^\vee \twoheadrightarrow \mathcal{V}_\bullet\]
    and the universal property of the flag variety.
\end{proof}

\begin{defi}
    We define \[\widetilde{\operatorname{NHilb}_{\operatorname{nil-fil}}^{\underline{d}}}(\A^n):=\tau^{-1}\left(\operatorname{NHilb}_{\operatorname{nil-fil}}^{\underline{d}}(\A^n)\right).\]
    Note that it is equipped with a $T$-equivariant perfect obstruction theory being exhibited as the zero locus of a section of $\tau^*\mathcal{E}_{\operatorname{ass}}$. 
\end{defi}
\begin{lemma}
    We have 
    \[\tau_*[\widetilde{\operatorname{NHilb}_{\operatorname{nil-fil}}^{\underline{d}}}(\A^n)]^{\operatorname{vir}}=[\operatorname{NHilb}_{\operatorname{nil-fil}}^{\underline{d}}(\A^n)]^{\operatorname{vir}}.\]
    In particular, for $\alpha\in A_T^*(\operatorname{NHilb}^{\underline{d}}_0(\A^n))$ we have 
    \[\int_{[\operatorname{NHilb}_{\operatorname{nil-fil}}^{\underline{d}}(\A^n)]^{\operatorname{vir}}}\alpha=\int_{[\widetilde{\operatorname{NHilb}_{\operatorname{nil-fil}}^{\underline{d}}}(\A^n)]^{\operatorname{vir}}} \tau^*\alpha.\]
\end{lemma}
\begin{proof}
    This follows by realizing the virtual fundamental classes as refined Euler classes in the sense of Fulton. Since $\tau$ is proper of degree 1 (being birational) the pushforward statement follows directly from \cite[Proposition 14.1, (d), (iii)]{Fult}. The integration statement is now simply the projection formula.
\end{proof}
\subsection{Localization over the flag variety}
The key to proving that tautological virtual integration on $\widetilde{\operatorname{NHilb}_0^{\underline{d}}}(\A^n)$ admits an iterated residue formula is to apply Atiyah-Bott localization on the flag variety. We will conduct this procedure in this section and provide fixed points and tangent space calculations which we will need later.

\begin{lemma}
    Let $\underline{d}=(1,\hat{d})=(1,d_1,\dots, d_r)\in \Z_{\geq 0}^{r+1}$, $n\geq 0$ and $d=1+d_1+\cdots+d_r$. Let $e_1,\dots, e_n \in \C^n$ be the standard unit vectors. The fixed locus of $\operatorname{Flag}_{\hat{d}}(\C^n)$ is 0-dimensional, reduced, and proper. Letting the symmetric group $\sigma  \in S_n$ act on the reference flag 
    \[f_{\bullet} =\operatorname{span}(e_1,\dots, e_{d_1})\subset \operatorname{span}(e_1,\dots, e_{d_1+d_2})\subset \cdots \subset \operatorname{span}(e_1,\dots, e_{d_1+\cdots d_r})\]
    by $e_i\mapsto e_{\sigma(i)}$ we obtain a bijection of sets
    \[\operatorname{Flag}_{\hat{d}}(\C^n)^T\cong S_n/(S_{d_1}\times \cdots \times S_{d_r}\times S_{n-(d-1)}).\]
    The $T$-equivariant Euler class of the tangent space at $\sigma(f_\bullet)$ is given by 
    \[e^T(T_{\operatorname{Flag}_{\hat{d}}(\C^n),\sigma(f_\bullet)})=\prod_{j=1}^{d-1}\prod_{\underset{w(i)>w(j)}{i=1}}^n(s_{\sigma(i)}-s_{\sigma(j)})\]
    in $A_*^T(\operatorname{pt})\cong \Z[s_1,\dots, s_n]$. Here 
    \[(w(1),\dots, w(n))=(\underbrace{1,\dots, 1}_{d_1 \text{ times}},\underbrace{2,\dots, 2}_{d_2 \text{ times}},\dots, \underbrace{r\dots, r}_{d_r \text{ times}},\underbrace{r+1,\dots, r+1}_{n-(d-1) \text{ times}}).\]
\end{lemma}
\begin{proof}
Using the reference flag, we get 
\begin{center}
    \begin{tikzcd}
        \operatorname{GL}_n/P_{f_\bullet} \arrow{r}{\sim} & \operatorname{Flag}_{\hat{d}}(\C^n) \\[-20pt]
        g \arrow[mapsto]{r} & g(f_\bullet)
    \end{tikzcd}
\end{center}
where $P_{f_\bullet}$ is the parabolic group of linear automorphisms preserving $f_\bullet$, from which the above can be easily derived. 
\end{proof}

\begin{defi}
    For $\sigma(f_\bullet)\in \operatorname{Flag}_{\hat{d}}(\C^n)$ we set 
    \[H_\sigma:=\mu^{-1}(\sigma(f_\bullet))\cap \widetilde{\operatorname{NHilb}_{\operatorname{nil-fil}}^{\underline{d}}}(\A^n) \]
    and let $\iota_\sigma: H_\sigma\hookrightarrow\widetilde{\operatorname{NHilb}_{\operatorname{nil-fil}}^{\underline{d}}}(\A^n)$ be the inclusion. Note that $H_\sigma$ inherits a perfect obstruction theory being cut out of the smooth variety $\mu^{-1}(\sigma(f_\bullet))$ by a section of $\tau^*\mathcal{E}_{\operatorname{ass}}|_{\mu^{-1}(\sigma(f_\bullet))}$. 
\end{defi}
We can now rewrite the virtual class of the nilpotently filtered locus as a sum of virtual classes of $[H_\sigma]$. 

\begin{corollary}
    For $\sigma \in S_n/(S_{d_1}\times \cdots \times S_{n-(d-1)})$ set $T_\sigma=T_{\operatorname{Flag}_{\hat{d}}(\C^n), \sigma(f_\bullet)}$ and $p_\sigma: H_\sigma \to \operatorname{pt}$. We have an identity
    \[[\widetilde{\operatorname{NHilb}_{\operatorname{nil-fil}}^{\underline{d}}}(\A^n)]^{\operatorname{vir}}=\sum_{\sigma\in S_n/(S_{d_1}\times \cdots \times S_{n-(d-1)})} \iota_{\sigma,*}\frac{[H_\sigma]^{\operatorname{vir}}}{e^T(p_\sigma^* T_{\sigma})}\]
    in the localized Chow group $A_*^{T}(\widetilde{\operatorname{NHilb}_{\operatorname{nil-fil}}^{\underline{d}}}(\A^n))_{\operatorname{loc}}$. In particular, for $\alpha \in A_{T}^*(\widetilde{\operatorname{NHilb}_{\operatorname{nil-fil}}^{\underline{d}}}(\A^n))$ it holds that
\[\int_{[\widetilde{\operatorname{NHilb}_{\operatorname{nil-fil}}^{\hat{d}}}(\A^n)]^{\operatorname{vir}}}\alpha =\sum_{\sigma\in S_n/(S_{d_1}\times \cdots \times S_{n-(d-1)})} \frac{\int_{[H_\sigma]^{\operatorname{vir}}}\iota_\sigma^* \alpha}{\prod_{j=1}^{d-1}\prod_{\underset{w(i)>w(j)}{1\leq i \leq n}}(s_{\sigma(i)}-s_{\sigma(j)})}\]
in the localized Chow group $A_*^T(\operatorname{pt})_{\operatorname{loc}}$.
\end{corollary}
\begin{proof}
    We apply \cref{general localization} to $g=\mu\circ \iota$ for 
    \[\iota: \widetilde{\operatorname{NHilb}_{\operatorname{nil-fil}}^{\underline{d}}}(\A^n)\hookrightarrow \widetilde{\operatorname{naNHilb}_{0}^{\underline{d}}}(\A^n) .\]
    Letting $I$ be the ideal of the closed immersion $\iota$ we have 
    \[\mathbb{L}_{g}=[I/I^2 \to \iota^*\Omega_{\mu}]\]
    and so can simply define the relative perfect obstruction theory by
        \begin{center}
        \begin{tikzcd}
            \mathbb{E}_{g} \arrow{d}&[-30pt]:=&[-35pt] \Big{[}&[-35pt] (\iota_*\tau^*\mathcal{E}^{ass})^\vee \arrow{d}{s_{ass}}\arrow{r} & \iota^*\Omega_{\mu}\arrow[equal]{d} \Big{]} \\
            \mathbb{L}_{g}&[-30pt]\cong&[-35pt] \Big{[} &[-35pt] I/I^2 \arrow{r} & \iota^*\Omega_{\mu} \Big{]}
        \end{tikzcd}
    \end{center}
    which is clearly compatible with the perfect obstruction theory on $\widetilde{\operatorname{NHilb}_{\operatorname{nil-fil}}^{\underline{d}}}(\A^n)$. 
    We conclude that
    \[[\widetilde{\operatorname{NHilb}_{\operatorname{nil-fil}}^{\underline{d}}}(\A^n)]^{\operatorname{vir}}=\sum_{\sigma} \iota_{\sigma,*}g^!\frac{[\sigma(f_\bullet)]}{e^T(T_{\sigma})}.\]
    If we give $\mu$ the smooth relative perfect obstruction theory and $\iota$ the perfect obstruction theory from $(\tau^*\mathcal{E}_{\operatorname{ass}},\tau^* s_{\operatorname{ass}})$ these are clearly compatible with the given perfect obstruction theory on $g$. 
    Functoriality of the virtual pullback now shows that
    \[g^!=\iota^!\mu^!=0^!\mu^*\]
    for $0^!$ the refined Gysin of the zero section of $\tau^*\mathcal{E}_{\operatorname{ass}}$ and $\mu^*$ the flat pullback. Since 
    \[0^!\mu^*[\sigma(f_\bullet)]=0^![\mu^{-1}(\sigma(f_\bullet))]=[H_\sigma]^{\operatorname{vir}}\]
    we see that
    \[g^!\frac{[\mu^{-1}(\sigma(f_\bullet))]}{e^T(T_{\sigma})}=\frac{[H_\sigma]^{\operatorname{vir}}}{e^T(p_\sigma^*T_{\sigma})}\]
    from which the result follows. The integration formula is a direct consequence of the projection formula.
\end{proof}
Next we study the structure of $[H_\sigma]^{\operatorname{vir}}$.
\begin{lemma}
    The map 
    \[\mu^{-1}(\sigma(f_\bullet)) \hookrightarrow \widetilde{\operatorname{naNHilb}^{\underline{d}}_{0}}(\A^n)\xrightarrow{\tau} \operatorname{naNHilb}^{\underline{d}}_{0}(\A^n)\]
    is a closed immersion. Explicitly, letting 
    $(\C^n)^\vee\twoheadrightarrow \sigma(f^\vee_\bullet)$
    be the dual flag associated with $\sigma(f_\bullet)$ it identifies with the locus where the universal map $\Psi_1: (\oo^n)^\vee \to \mathcal{V}_1$ factors through a (necessarily unique) map of flags
    \[\oo \otimes \sigma(f^\vee_\bullet) \to \mathcal{V}_\bullet.\]
    In particular, if we let $\mathcal{E}_\sigma$ be the quotient of the sheaf
    \[\mathcal{H}om((\oo^n)^\vee,\mathcal{V}_1)/\mathcal{H}om^{\operatorname{fil}}(\oo\otimes \sigma(f_\bullet^\vee),\mathcal{V}_\bullet)\]
     then $\mu^{-1}(\sigma(f_\bullet))$ is the zero locus in $\operatorname{naNHilb}_0^{\underline{d}}(\A^n)$ of the section of $\mathcal{E}_\sigma$ induced by $\Psi_1$.
\end{lemma}
\begin{proof}
This is immediate from the definition of $\widetilde{\operatorname{naNHilb}^{\underline{d}}_{0}}(\A^n)$. 
\end{proof}
\begin{corollary} \label{fixed points on flag fiber}
Let $\lambda_\bullet=(\lambda_1\subset \cdots \subset \lambda_{r+1}) $ be a nested partition satisfying the hypothesis of \cref{nil-fil monomial}. The corresponding fixed point of $\operatorname{NHilb}_{\operatorname{nil-fil}}^{\underline{d}}(\A^n)$ is in $H_\sigma$ if and only if 
\[e_{\sigma(j)}\notin \lambda_i\]
whenever $w(j)\geq i-1$.
\end{corollary}
\begin{proof}
Recall that 
\[\sigma(f_i)=\operatorname{span}(e_{\sigma(1)},\dots, e_{\sigma(d_1+\cdots+d_i)})\]
with $\sigma(f_0)=\{0\}$. The dual flag is 
\[\sigma(f_i^{\vee})=\operatorname{ker}(\sigma(f_{r})^\vee \to \sigma(f_{i-1})^\vee)\]
which as a quotient of $(\C^n)^\vee=\operatorname{span}(x_1,\dots, x_n)$ can be described as 
\begin{align*}I^{\sigma}_i/I^{\sigma}_{r+1}:=&\operatorname{span}(x_k \: | \: k\neq \sigma(1),\dots,\sigma(d_1+\cdots+d_{i-1}))/\operatorname{span}(x_k \: | \: k\neq \sigma(1),\dots,\sigma(d_1+\cdots+d_{r}))\\
=&\operatorname{span}(x_{\sigma(j)} \: | \:w(j)\geq i-1)/\operatorname{span}(x_{\sigma(j)} \: | \: w(j)\geq r)\end{align*}
where $I_1^{\sigma}=(\C^n)^\vee$. 
By definition 
\[\mathcal{V}_i|_{\lambda}=I_{\lambda_i}/I_{\lambda_{r+1}}\]
with $\psi_1: (\C^n)^\vee\to I_{\lambda_1}/I_{\lambda_{r+1}}$ the restriction of the quotient map along
\[(\C^n)^\vee\subset (x_1,\dots, x_n)=I_{\lambda_1}.\]
It follows that $\psi_1$ factors through a map of flags $\sigma(f_\bullet^\vee)\to \mathcal{V}_\bullet|_{\lambda}$ if an only if 
\[I_i\subset I_{\lambda_i}\]
for all $i$ i.e., if and only if 
\[e_{\sigma(j)}\notin \lambda_i\]
whenever $w(j)\geq i-1$.
\end{proof} 
\begin{corollary}
Let $\lambda_\bullet$ be a nested partition such that the corresponding fixed point is in $H_\sigma$. Let $\mathbf{u}_0,\dots, \mathbf{u}_{d-1}$ be an enumeration of $\lambda_\bullet$ in the sense of \cref{Partition setup}. Then the trace of $T_{\mu^{-1}(\sigma(f_\bullet)),\lambda_\bullet}$ is given by
   \[\sum_{i=1}^{d-1}\sum_{\substack{1\leq k\leq d-1 \\ w(i)\leq w(k)}}t^{e_{\sigma(i)}-\mathbf{u}_k}+\sum_{1\leq i\leq j\leq d-1} \sum_{\substack{1\leq k\leq d-1\\ w(j)< w(k)}}t^{\mathbf{u}_i+\mathbf{u}_j-\mathbf{u}_k}-\sum_{j=1}^{d-1}\sum_{\substack{1\leq k\leq d-1\\ w(j)\leq w(k)}}t^{\mathbf{u}_j-\mathbf{u}_k}.\]
\end{corollary}
\begin{proof}
We note that
\[T_{\mu^{-1}(\sigma(f_\bullet))}=T_{\operatorname{naNHilb}^{\underline{d}}_0(\A^n)}|_{\sigma(f_\bullet))}-\mathcal{E}_{\sigma}|_{\sigma(f_\bullet))}\]
which follows directly from the above. We directly see that
\[\mathcal{E}_\sigma=(t_1+\cdots +t_n)\sum_{k=1}^{d-1} t^{-\mathbf{u}_k}-\sum_{i=1}^{d-1}\sum_{\substack{1\leq k\leq d-1 \\ w(i)\leq w(k)}}t^{e_{\sigma(i)}-\mathbf{u}_k}\]
which combined with the previous calculation of 
\[T_{\operatorname{naNHilb}^{\underline{d}}_0(\A^n),\lambda_\bullet}\]
yields the result.
\end{proof}
Again, the cancellation afterwards is trickier. We instead add a recursive definition of the multiset of weights similar to that of \cref{recursive tangent weights}. 
\begin{lemma} \label{recursive fiber weights}
    Let $\lambda_\bullet$ be a nested partition such that the associated fixed point is in $H_\sigma$. For $0\leq m\leq r$ we define multisets
    \[S^{T,\sigma}_{m}(\lambda_\bullet)\subset S^{T}_{m}(\lambda_\bullet)\]
    recursively as follows. Pick an enumeration $\mathbf{u}_0,\dots, \mathbf{u}_{d-1}$ of $\lambda_\bullet$. Set
    \[S^{T,\sigma}_{0}(\lambda_\bullet)=\varnothing \]
    and define $S^{T,\sigma}_{m}(\lambda_\bullet)$ recursively as 
    \[\left(S^{T,\sigma}_{m-1}(\lambda_\bullet)\cup \{e_{\sigma(i)}\: | \: w(i)=m\}\cup\{\mathbf{u}_i+\mathbf{u}_j \: |\: 1\leq i\leq j, \: w(j)=m-1\}\right) \Big{\backslash}\{\mathbf{u}_j\: | \: 1\leq j, \: w(j)=m\}.\]
    Then the trace of $T_{\mu^{-1}(\sigma(f_\bullet)),\lambda_\bullet}$ is given by
        \[\sum_{m=0}^r \sum_{\mathbf{v}\in \lambda_{m+1}\backslash \lambda_m} \sum_{\mathbf{u}\in S_m^{T,\sigma}(\lambda_\bullet)} t^{\mathbf{u}-\mathbf{v}}.\]
\end{lemma}

\subsection{The residue formula}
\begin{set} \label{tautological integral setup}
     Let $\underline{d}=(1,\hat{d})=(1,d_1,\dots, d_r)\in \Z_{\geq 0}^{r+1}$, $n\geq 0$ and $d=1+d_1+\cdots+d_r$. Let $W$ be a $q$-dimensional $T$-representation, let $E=W\otimes_{\C}\oo_{\A^n}$ be the associated $T$-equivariant bundle on $\A^n$. Let $E^{[\underline{d}]}$ be the associated tautological bundle on $\operatorname{NHilb}^{\underline{d}}(\A^n)$ of rank $m=d\cdot q$. We note that
    \[E^{[\underline{d}]}=W\otimes_{\C} \oo_{\A^n}^{{[\underline{d}]}}.\]
    We let $\theta_1,\dots, \theta_q$ be the Chern roots of $W$ and $\upsilon_0,\upsilon_1,\dots, \upsilon_{d-1}$ be the Chern roots of $\oo^{[\underline{d}]}$ ordered such that $\oo^{[d_0+d_1+\cdots+d_i]}$ has Chern roots $\upsilon_0, \upsilon_1,\dots, \upsilon_{d_1+\cdots d_i}$. Then $E^{[\underline{d}]}$ has Chern roots 
    \[\theta_i+\upsilon_j\]
    for $1\leq i\leq k$ and $0\leq j\leq d-1$. Any polynomial in the Chern classes of $E^{[\underline{d}]}$ is therefore a bisymmetric polynomial in $\theta_1,\dots ,\theta_q$ and $\upsilon_0,\dots, \upsilon_{d-1}$. When we develop the tautological integral formulas we will therefore just pick $\theta_1,\dots, \theta_q\in A_T^1(\operatorname{pt})$ and consider such bisymmetric polynomials. As it is more convenient, we will also substitute $\upsilon_i$ with $\eta_i:=-\upsilon_i$. Note that we can then choose $\eta_i$ such that 
    \[\eta_i|_{\lambda_\bullet}=\mathbf{u}_i(\underline{s})\]
    for a fitting enumeration $\mathbf{u}_0,\dots, \mathbf{u}_{d-1}$ of a partition $\lambda_\bullet$. Note also that $\eta_0=\upsilon_0=0$, as it corresponds to the first Chern class of the trivial summand coming from the unit of the algebra. 
\end{set}

\begin{prop}
     Let $\theta_1,\dots, \theta_q\in A_T^1(\operatorname{pt})$ and let
     \[P=P(\theta_1,\dots, \theta_q,\eta_1,\dots, \eta_{d-1})\]
     be a bisymmetric polynomial as in \cref{tautological integral setup}. Let 
     \[A_*^T(\operatorname{pt})\cong \Z[s_1,\dots, s_n].\]
     Then there exists a polynomial $Q$ in $k+(d-1)$ variables with coefficients in $\Z$ such that for all 
    \[\sigma \in S_n/(S_{d_1}\times \cdots \times S_{d_r}\times S_{n-(d-1)})\]
    we have
     \[\int_{[H_{\sigma}]^{\operatorname{vir}}}P=Q(\theta_1,\dots ,\theta_q,s_{\sigma(1)},\dots,s_{\sigma(d-1)}). \]
\end{prop}
\begin{proof}
We first note that
\[\oo_{\A^n}^{[\underline{d}]}=\mathcal{V}_\bullet|_{\operatorname{NHilb}^{\underline{d}}(\A^n)}\]
for $\mathcal{V}_\bullet$ the universal flag on the non-associative nested Hilbert scheme. We apply virtual localization
\[\int_{[H_{\id}]^{\operatorname{vir}}}P=\sum_{\lambda_\bullet\in H_{\operatorname{id}}^T} P_{\lambda_\bullet}\cdot \frac{e^T(\mathcal{E}_{ass,\lambda_{\bullet}})}{e^T(T_{\mu^{-1}(\sigma(f_\bullet)),\lambda_\bullet})}.\]
We can choose an enumeration $\mathbf{u}_0,\dots, \mathbf{u}_{d-1}$ of $\lambda_\bullet$ satisfying the properties of \cref{Partition setup} and such that
\[\eta_i|_{\lambda_\bullet}=\mathbf{u}_i(\underline{s}).\]
Note that for $\lambda_\bullet \in H_{\operatorname{id}}^T$ we have that $e_j\notin \lambda_\bullet$ for $j>d-1$. It follows that $P_{\lambda_\bullet}$ is a polynomial in $\theta_1,\dots, \theta_q,s_1,\dots, s_{d-1}$. From the explicit calculation of their $K$-theory class above, we can choose polynomials $Q_{1,\lambda_\bullet},Q_{2,\lambda_\bullet}$ such that
\[e^T(\mathcal{E}_{ass,\lambda_{\bullet}})=Q_{1,\lambda_\bullet}(\theta_1,\dots, \theta_q,s_1,\dots, s_{d-1})\]
and 
\[e^T(T_{\mu^{-1}(\sigma(f_\bullet)),\lambda_\bullet})=Q_{2,\lambda_\bullet}(\theta_1,\dots, \theta_q,s_1,\dots, s_{d-1}).\]
We now simply define 
\[Q:=\sum_{\lambda_\bullet\in H_{\operatorname{id}}^T}P_{\lambda_\bullet}\frac{Q_{1,\lambda_\bullet}}{Q_{2,\lambda_\bullet}}.\]
For  
\[\sigma \in S_n/(S_{d_1}\times \cdots \times S_{d_r}\times S_{n-(d-1)})\]
and a nested partition $\lambda_\bullet$ we define $\sigma(\lambda_\bullet)$ by simply substituting $e_i\mapsto e_{\sigma(i)}$. The assignment $\lambda_\bullet \mapsto \sigma(\lambda_\bullet)$ defines a bijection 
\[H_{\operatorname{id}}^T\to H_{\sigma}^T\]
and one directly see from the concrete formulas that
\[e^T(\mathcal{E}_{ass,\sigma(\lambda_{\bullet})})=Q_{1,\lambda_\bullet}(\theta_1,\dots, \theta_{k},s_{\sigma(1)},\dots, s_{\sigma(d-1)})\]
and 
\[e^T(T_{\mu^{-1}(\sigma(f_\bullet)),\sigma(\lambda_\bullet)})=Q_{2,\lambda_\bullet}(\theta_1,\dots, \theta_q,s_{\sigma(1)},\dots, s_{\sigma(d-1)}).\]
Applying virtual localization to $[H_\sigma]^{\operatorname{vir}}$ and comparing gives the result.
\end{proof}
We can now obtain a iterated residue formula for virtual tautological integrals on $\operatorname{NHilb}_{\operatorname{nil-fil}}^{\underline{d}}(\A^n)$. The main input is the following iterated residue formula for integration on the flag variety $\operatorname{Flag}_{\hat{d}}(\C^n)$ generalizing \cite[Proposition 5.4]{ThomMorin}.

\begin{prop}\label{prop:weighted-residue}
Let $\hat d=(d_1,\dots,d_r)\in\Z_{\ge0}^r$ and set $k = |\hat d| := d_1+\cdots+d_r\le n$.
Let $Q$ be a polynomial on $\C^k$ and let $s_1,\dots,s_n$ be parameters.

Partition $\{1,\dots,n\}$ into blocks
\[
B_1=\{1,\dots,d_1\},\;
B_2=\{d_1+1,\dots,d_1+d_2\},\dots,
B_r=\{k-d_r+1,\dots,k\},\;
B_{r+1}=\{k+1,\dots,n\},
\]
and define a weight function
\[
w:\{1,\dots,n\}\to\{1,\dots,r+1\},\qquad
w(i)=p\ \text{if and only if } i\in B_p.
\]

Then
\begin{equation}\label{eq:weighted-residue}
\sum_{\sigma\in S_n/(S_{d_1}\times\cdots\times S_{d_r}\times S_{n-k})}
\frac{Q\bigl(s_{\sigma(1)},\dots,s_{\sigma(k)}\bigr)}
     {\displaystyle\prod_{j=1}^{k}\prod_{\substack{1\le i\le n\\w(i)>w(j)}}
      \bigl(s_{\sigma(i)}-s_{\sigma(j)}\bigr)}
=
\underset{\bz=\infty}{\operatorname{Res}}
\frac{\displaystyle\prod_{1\le m\neq \ell\le k}(z_m-z_\ell)\,Q(\bz)\,dz_1\cdots dz_k}
     {\displaystyle\prod_{\substack{1\le \ell<m\le k\\w(\ell)<w(m)}}(z_m-z_\ell)
       \prod_{i=1}^n\prod_{\ell=1}^k(s_i-z_\ell)}.
\end{equation}
\end{prop}

\begin{proof}
Consider the meromorphic $k$-form on $(\mathbb{P}^1)^k$
\[
\omega_k(\bz)
=
\frac{\displaystyle\prod_{1\le m\neq \ell\le k}(z_m-z_\ell)\,
Q(\bz)\,dz_1\cdots dz_k}
     {\displaystyle\prod_{\substack{1\le \ell<m\le k\\w(\ell)<w(m)}}(z_m-z_\ell)
       \prod_{i=1}^n\prod_{\ell=1}^k(s_i-z_\ell)},
\qquad
\bz=(z_1,\dots,z_k).
\]
We will compute the iterated residue
\[
\sires_{\bz=\infty}\omega_k
:=\sires_{z_1=\infty}\cdots\sires_{z_k=\infty}\omega_k
\]
by successively applying the Residue Theorem on $\mathbb{P}^1$ in the variables
$z_k,z_{k-1},\dots,z_1$, as in \cite{bsz}.\\
\noindent \textbf{Step 1: Isolating the dependence on $z_k$.}
Write the full Vandermonde product as
\[
\prod_{1\le m\neq \ell\le k}(z_m-z_\ell)
=
\Bigl(\prod_{\substack{1\le m\neq \ell\le k-1}}(z_m-z_\ell)\Bigr)
\Bigl(\prod_{\ell<k}(z_k-z_\ell)(z_\ell-z_k)\Bigr).
\]
Among the cross-block factors in the denominator
\[
\prod_{\substack{1\le \ell<m\le k\\w(\ell)<w(m)}}(z_m-z_\ell),
\]
the only ones involving $z_k$ are those with $m=k$ and $\ell<k$,
namely
\[
\prod_{\substack{\ell<k\\w(\ell)<w(k)}}(z_k-z_\ell).
\]
Thus, we can factor $\omega_k$ as
\[
\omega_k
=
F_{k-1}(\bz')\;
\frac{\displaystyle
\prod_{\ell<k}(z_k-z_\ell)(z_\ell-z_k)}
     {\displaystyle\prod_{\substack{\ell<k\\w(\ell)<w(k)}}(z_k-z_\ell)
       \prod_{i=1}^n(s_i-z_k)}\,dz_k,
\]
where $\bz'=(z_1,\dots,z_{k-1})$ and
\[
F_{k-1}(\bz')
=
\frac{\displaystyle\prod_{1\le m\neq \ell\le k-1}(z_m-z_\ell)\,
Q(\bz)\,dz_1\cdots dz_{k-1}}
     {\displaystyle\prod_{\substack{1\le \ell<m\le k-1\\w(\ell)<w(m)}}(z_m-z_\ell)
       \prod_{i=1}^n\prod_{\ell=1}^{k-1}(s_i-z_\ell)}.
\]

For each $\ell<k$ we have:
\begin{itemize}
\item If $w(\ell)<w(k)$, then a denominator factor $(z_k-z_\ell)$ cancels one
  factor in $(z_k-z_\ell)(z_\ell-z_k)$, leaving $(z_\ell-z_k)$.
\item If $w(\ell)=w(k)$, there is no denominator factor, and we simply keep
  $(z_k-z_\ell)(z_\ell-z_k)=-(z_k-z_\ell)^2$.
\end{itemize}
Hence, as a function of $z_k$,
\[
\omega_k
=
\Phi_k(\bz')\;
\frac{R_k(z_k;\bz')}{\prod_{i=1}^n(s_i-z_k)}\,dz_k,
\]
where $\Phi_k(\bz')$ is independent of $z_k$ and $R_k(z_k;\bz')$ is a polynomial
in $z_k$. In particular, the only poles in the $z_k$ variable are at
$z_k=s_1,\dots,s_n$.

\medskip

\noindent\textbf{Step 2: Residue in $z_k$.}
By the Residue Theorem on $\mathbb{P}^1$,
\[
\sires_{z_k=\infty}\omega_k = -\sum_{j=1}^n \sires_{z_k=s_j}\omega_k.
\]
For a simple pole in $s_j-z_k$, we have
\[
\sires_{z_k=s_j}
\frac{R_k(z_k;\bz')}{\prod_{i=1}^n(s_i-z_k)}\,dz_k
=
-\frac{R_k(s_j;\bz')}{\prod_{i\ne j}(s_i-s_j)}.
\]
Thus, the minus sign from the Residue Theorem cancels this minus sign, giving
\[
\sires_{z_k=\infty}\omega_k
=
\sum_{j=1}^n
\Phi_k(\bz')\,
\frac{R_k(s_j;\bz')}{\prod_{i\ne j}(s_i-s_j)}.
\]

Now track the factors involving $z_k$ more explicitly. As observed above,
for each $\ell<k$ we obtain after substituting $z_k=s_j$ at least one factor
$(z_\ell-s_j)$ in the numerator. On the other hand, in $F_{k-1}(\bz')$ we still
have the denominator
\[
\prod_{i=1}^n(s_i-z_\ell),
\qquad \ell=1,\dots,k-1.
\]
Hence for each $\ell<k$ the factor $(z_\ell-s_j)$ cancels the factor
$(s_j-z_\ell)$ in this product, leaving
\[
\prod_{i\ne j}(s_i-z_\ell)
\]
in the denominator and at most one $(z_\ell-s_j)$ factor remaining in the
numerator. In particular, after the $z_k$-residue, the pole at $z_\ell=s_j$
disappears for all $\ell<k$.

Collecting all factors, we obtain (up to an overall sign depending only on $k$)
\begin{equation}\label{eq:after-first-res}
\sires_{z_k=\infty}\omega_k
=
\sum_{j=1}^n
\frac{\displaystyle
\prod_{1\le m\neq \ell\le k-1}(z_m-z_\ell)\,
Q(z_1,\dots,z_{k-1},s_j)\,
\prod_{\ell=1}^{k-1}(z_\ell-s_j)\,dz_1\cdots dz_{k-1}}
     {\displaystyle
\prod_{\substack{1\le \ell<m\le k-1\\w(\ell)<w(m)}}(z_m-z_\ell)
\prod_{\ell=1}^{k-1}\prod_{i\ne j}(s_i-z_\ell)
\prod_{i\ne j}(s_i-s_j)}.
\end{equation}
This has the same structure as the original integrand, but with one fewer
variable and with $j$ excluded from the set of indices appearing in the
denominators.

\medskip

\noindent\textbf{Step 3: Iteration.}
We now repeat the same procedure with $z_{k-1}$, then $z_{k-2}$, and so on.

Inductively, after having taken residues in $z_k,\dots,z_{m+1}$ we obtain a sum
over ordered $(k-m)$-tuples of distinct indices $(j_{m+1},\dots,j_k)$, and for
each such choice an integrand in the remaining variables $z_1,\dots,z_m$ of the
same form:
\[
\frac{\displaystyle\prod_{1\le a\neq b\le m}(z_a-z_b)\,
Q(\dots)\,dz_1\cdots dz_m}
     {\displaystyle\prod_{\substack{1\le \ell<m\le m\\w(\ell)<w(m)}}(z_m-z_\ell)
       \prod_{\ell=1}^m\prod_{i\in I_m}(s_i-z_\ell)}
\cdot (\text{constant depending on } s_{j_{m+1}},\dots,s_{j_k}),
\]
where $I_m\subset\{1,\dots,n\}$ is the set of indices not yet used. In
particular, the only poles in $z_m$ are at $z_m=s_i$ with $i\in I_m$.

The same argument as in Step~2 shows that:

\begin{itemize}
\item At the $z_m$ step, we pick an index $j_m\in I_m$ and obtain a factor
  $1/\prod_{i\in I_m\setminus\{j_m\}}(s_i-s_{j_m})$.
\item The pole at $z_\ell=s_{j_m}$ disappears for every $\ell<m$; hence
  $j_m$ cannot reappear at later steps. Thus, the indices $j_1,\dots,j_k$
  are all distinct.
\item Tracking the cross-block factors, one checks that for the chosen
  index $j_m$ with weight $w(j_m)$, the only factors that survive in the
  final denominator are those with $w(i)>w(j_m)$. The factors with
  $w(i)\le w(j_m)$ are cancelled against factors arising from the
  Vandermonde numerator at earlier steps.
\end{itemize}

Proceeding inductively down to $z_1$, we thus obtain a sum over ordered
$k$-tuples of distinct indices $(j_1,\dots,j_k)$, or equivalently over
cosets $\sigma\in S_n/(S_{d_1}\times\cdots\times S_{d_r}\times S_{n-k})$,
of terms of the form
\[
\frac{Q(s_{j_1},\dots,s_{j_k})}
     {\displaystyle\prod_{m=1}^k\prod_{\substack{i\in\{1,\dots,n\}\\w(i)>w(j_m)}}
      (s_i-s_{j_m})}
=
\frac{Q\bigl(s_{\sigma(1)},\dots,s_{\sigma(k)}\bigr)}
     {\displaystyle\prod_{j=1}^k\prod_{\substack{1\le i\le n\\w(i)>w(j)}}
      (s_{\sigma(i)}-s_{\sigma(j)})}.
\]

\medskip

\noindent\textbf{Step 4: signs.}
At each step there is
a sign coming from exchanging factors $(z_m-z_\ell)$ with $(z_\ell-z_m)$, and
a global sign from expanding the Vandermonde product
$\prod_{1\le m\neq \ell\le k}(z_m-z_\ell)$. The total sign is
$(-1)^{\sum_p\binom{d_p}{2}}$ both from these exchanges and from the
Vandermonde, so these sign contributions cancel in the final expression.

Thus, the iterated residue $\sires_{\bz=\infty}\omega_k$ is exactly the
left-hand side of~\eqref{eq:weighted-residue}, which proves the identity.
\end{proof}

\begin{corollary}
    Let $\theta_1,\dots, \theta_q\in A_T^1(\operatorname{pt})$ and let
    \[P=P(\theta_1,\dots, \theta_q,\eta_1,\dots, \eta_{d-1})\]
    be a bisymmetric polynomial as in \cref{tautological integral setup}. Then 
    \[\int_{[\operatorname{NHilb}^{\underline{d}}_{\operatorname{nil-fil}}(\A^n)]^{\operatorname{vir}}}P=\underset{\underline{z}=\infty}{\operatorname{Res}}\frac{\prod_{1\leq m\neq l\leq d-1}(z_m-z_l)Q(\theta_1,\dots, \theta_q,z_1,\dots, z_{d-1})d\underline{z}}{\prod_{\substack{1\leq l<m\leq d-1 \\ w(l)<w(m)}}(z_m-z_l)\prod_{i=1}^n \prod_{l=1}^{d-1}(s_i-z_l)}\]
    where $\underline{z}=(z_1,\dots, z_{d-1})$ and
    \[Q(\theta_1,\dots, \theta_q,s_1,\dots, s_{d-1})=\int_{[H_{\id}]^{\operatorname{vir}}}P.\]
    Explicitly, one can calculate $Q$ as
    \[\sum_{\lambda_\bullet} P_{\lambda_\bullet}\frac{\prod_{\substack{1\leq i,j,k\leq d-1 \\ i<k}} \widetilde{\prod}_{\substack{0\leq m\leq d-1\\ w(k),w(j)\leq w(m)}}(\mathbf{u}_i+\mathbf{u}_j+\mathbf{u}_k-\mathbf{u}_m)(\underline{s})\widetilde{\prod}_{j=1}^{d-1}\prod_{\substack{1\leq k\leq d-1\\ w(j)\leq w(k)}}(\mathbf{u}_j-\mathbf{u}_k)(\underline{s})}{\prod_{i=1}^n\widetilde{\prod}_{\substack{1\leq k\leq d-1 \\ w(i)\leq w(k)}}(e_i-\mathbf{u}_k)(\underline{s})\prod_{1\leq i\leq j\leq d-1} \widetilde{\prod}_{\substack{1\leq k\leq d-1\\ w(j)<w(k)}}(\mathbf{u}_i+\mathbf{u}_j-\mathbf{u}_k)(\underline{s})}\]
    or using the notation of \cref{recursive bundle weights} and \cref{recursive fiber weights}
    \[\sum_{\lambda_\bullet}P_{\lambda_\bullet}\prod_{m=1}^r \prod_{\mathbf{v}\in \lambda_{m+1}\backslash \lambda_m}\frac{\prod_{\substack{\mathbf{u}\in S_m^{ass}(\lambda_\bullet)\\ \mathbf{u}\neq \mathbf{v}}} (\mathbf{u}-\mathbf{v})(\underline{s})}{\prod_{\substack{\mathbf{u}\in S_m^{T,\id}(\lambda_\bullet)\\ \mathbf{u}\neq \mathbf{v}}} (\mathbf{u}-\mathbf{v})(\underline{s})}\]
    Here the sums are taken over all admissible nested partitions $\lambda_{\bullet}$ satisfying $\mathbf{u}+e_i\notin \lambda_{k+1}$ for all $\mathbf{u}\in \lambda_{k+1}\backslash \lambda_k$ and all $k\geq 1$ and that $e_j\notin \lambda_{i}$ whenever $w(j)\geq i-1$ and for $\mathbf{u}_0,\dots, \mathbf{u}_{d-1}$ an enumeration of $\lambda_\bullet$ as in \cref{Partition setup} we define 
    \[P_{\lambda_\bullet}=P(\theta_1,\dots,\theta_q,\lambda_\bullet(\underline{z}))=P(\theta_1,\dots, \theta_q,\mathbf{u}_1(\underline{s}),\dots, \mathbf{u}_{d-1}(\underline{s})).\]
\end{corollary}
\subsection{Residue vanishing}
The above formula shows that the virtual integral can be exhibited as a sum of iterated residues, i.e., we have that
\[\sum_{\lambda_{\bullet}}\underset{\underline{z}=\infty}{\operatorname{Res}}\frac{P(\theta_1,\dots,\theta_q,\lambda_\bullet(\underline{z}))\prod_{1\leq m\neq l\leq d-1}(z_m-z_l)\prod_{m=1}^r \prod_{\mathbf{v}\in \lambda_{m+1}\backslash \lambda_m}\prod_{\substack{\mathbf{u}\in S_m^{ass}(\lambda_\bullet)\\ \mathbf{u}\neq \mathbf{v}}} (\mathbf{u}-\mathbf{v})(\underline{z})d\underline{z}}{\prod_{\substack{1\leq l<m\leq d-1 \\ w(l)<w(m)}}(z_m-z_l)\prod_{i=1}^n \prod_{l=1}^{d-1}(s_i-z_l)\prod_{m=1}^r \prod_{\mathbf{v}\in \lambda_{m+1}\backslash \lambda_m}\prod_{\substack{\mathbf{u}\in S_m^{T,\id}(\lambda_\bullet)\\ \mathbf{u}\neq \mathbf{v}}} (\mathbf{u}-\mathbf{v})(\underline{z})} \]
calculates 
\[\int_{[\operatorname{NHilb}^{\underline{d}}_{\operatorname{nil-fil}}(\A^n)]^{\operatorname{vir}}}P(\theta_1,\dots, \theta_q,\eta_1,\dots, \eta_{d-1}).\]
In this section we prove that most of these residues vanish. In fact, we prove that the only non-vanishing residue is the one related to the Porteous point
\[\lambda_k=\{0,e_1,e_2,\dots, e_{d_1+\dots +d_{k-1}}\}.\]
To prove the vanishing of an iterated residue we use the following criterion from \cite{ThomMorin}.
\begin{prop} \label{General residue vanishing}
\cite[Proposition 6.3]{ThomMorin} Let $p(\underline{z})$ and $q(\underline{z})$ be polynomials in the variables $z_1,\dots, z_r$ and assumme that $q(\underline{z})=\prod_{i=1}^NL_i(\underline{z})$ is a product of linear factors. Let $d\underline{z}=d_{z_1}\cdots d_{z_r}$. The iterated residue
\[\underset{\underline{z}=\infty}{\operatorname{Res}}\frac{p(\underline{z})d\underline{z}}{q(\underline{z})}\]
vanishes if there exist $1\leq M\leq r$ satisfying either of the following conditions
\begin{itemize}
    \item $\operatorname{deg}(p(\underline{z});\{r,r-1,\dots, M\})+r-M+1<\operatorname{deg}(q(\underline{z});\{r,r-1,\dots, M\})
    $
\end{itemize}
or 
\begin{itemize}
    \item $\operatorname{deg}(q(\underline{z});\{M\})=\operatorname{lead}(q(\underline{z}),M)$ and $\operatorname{deg}(p(\underline{z});\{M\})+1<\operatorname{deg}(q(\underline{z});\{M\}).$
\end{itemize}
\end{prop}
Recall that 
\[\operatorname{lead}(q(\underline{z}),M):=\# \left\{i; \operatorname{max}\{l;\operatorname{coef}(L_i,z_l)\neq0\}=M\right\}\]
i.e., the number of linear terms $L_i$ where the highest indexed variable appearing is $z_M$.
\begin{prop}
    Let $\theta_1,\dots, \theta_q\in A_T^1(\operatorname{pt})$ and let
    \[P=P(\theta_1,\dots, \theta_q,\eta_1,\dots, \eta_{d-1})\]
    be a bisymmetric polynomial as in \cref{tautological integral setup}. Let $\lambda_{\bullet}$ be a nested partition not of Porteuos type and assume that the corresponding fixed point is in $H_{\id}$. Then
    \[\underset{\underline{z}=\infty}{\operatorname{Res}}\frac{P(\theta_1,\dots,\theta_q,\lambda_\bullet(\underline{z}))\prod_{1\leq m\neq l\leq d-1}(z_m-z_l)\prod_{m=1}^r \prod_{\mathbf{v}\in \lambda_{m+1}\backslash \lambda_m}\prod_{\substack{\mathbf{u}\in S_m^{ass}(\lambda_\bullet)\\ \mathbf{u}\neq \mathbf{v}}} (\mathbf{u}-\mathbf{v})(\underline{z})d\underline{z}}{\prod_{\substack{1\leq l<m\leq d-1 \\ w(l)<w(m)}}(z_m-z_l)\prod_{i=1}^n \prod_{l=1}^{d-1}(s_i-z_l)\prod_{m=1}^r \prod_{\mathbf{v}\in \lambda_{m+1}\backslash \lambda_m}\prod_{\substack{\mathbf{u}\in S_m^{T,\id}(\lambda_\bullet)\\ \mathbf{u}\neq \mathbf{v}}} (\mathbf{u}-\mathbf{v})(\underline{z})} =0.\]
\end{prop}

\begin{proof}
        For the proof we let $p(\underline{z})$ and $q(\underline{z})$ be the numerator and the denominator of the expression in the iterated residue formula. Since $\lambda_\bullet$ is not Porteous we can pick the largest $1\leq M \leq d-1$ such that $e_M\notin \lambda_{\bullet}$. We first note that $e_M\notin \lambda_{\bullet}$ means that $z_M$ does not appear in the term 
        \[P(\theta_1,\dots,\theta_q,\lambda_\bullet(\underline{z}))\]
        nor in the term 
        \[\prod_{m=1}^r \prod_{\mathbf{v}\in \lambda_{m+1}\backslash \lambda_m}\prod_{\substack{\mathbf{u}\in S_m^{ass}(\lambda_\bullet)\\ \mathbf{u}\neq \mathbf{v}}} (\mathbf{u}-\mathbf{v})(\underline{z}).\]
        We proceed by reverse induction on $M$. Suppose first that $M=d-1$. Since $z_{d-1}$ is the variable with the largest index it follows by definition that
    \[\operatorname{deg}(q(\underline{z});\{d-1\})=\operatorname{lead}(q(\underline{z}),d-1).\]
    We note that
    \[e_{d-1}\notin S_m^{T,\id}(\lambda_\bullet)\]
    for $m<r$ and that $e_{d-1}$ appears exactly once in $S_r^{T,\id}(\lambda_\bullet)$. It follows that it appears in exactly $|\lambda_{r+1}\backslash \lambda_r|=d_r$ linear terms of 
    \[\prod_{m=1}^r \prod_{\mathbf{v}\in \lambda_{m+1}\backslash \lambda_m}\prod_{\substack{\mathbf{u}\in S_m^{T,\id}(\lambda_\bullet)\\ \mathbf{u}\neq \mathbf{v}}} (\mathbf{u}-\mathbf{v})(\underline{z}).\]
    Clearly it appears in $n$ linear terms of 
    \[\prod_{i=1}^n \prod_{l=1}^{d-1}(s_i-z_l)\]
    and $d_1+\cdots+d_{r-1}$ linear terms of 
    \[\prod_{\substack{1\leq l<m\leq d-1 \\ w(l)<w(m)}}(z_m-z_l).\]
    We conclude that
    \[\operatorname{deg}(q(\underline{z});\{d-1\})=d-1+n.\]
    For $p(\underline{z})$ we see that $z_{d-1}$ only appears in
    \[\prod_{1\leq m\neq l\leq d-1}(z_m-z_l)\]
    where it appears in $2(d-1)-2$ terms. From this we see that
    \[\operatorname{deg}(p(\underline{z});\{d-1\})+1=2(d-1)-1<d-1+n=\operatorname{deg}(q(\underline{z});\{d-1\})\]
    since $n\geq d-1$ by assumption. We conclude the vanishing from \cref{General residue vanishing}. For general $M$ we can pick an enumeration $\mathbf{u}_0,\dots, \mathbf{u}_{d-1}$ of $\lambda_\bullet$ in such a way that
    \[\mathbf{u}_M\neq e_M,\: \mathbf{u}_{M+1}=e_{M+1},\dots,\: \mathbf{u}_{d-1}=e_{d-1}\]
    We note that $e_{M}\in S_m^{T,\operatorname{id}}(\lambda_\bullet)$ if and only if $m\geq w(M)$ and in this case it appears exactly once. It follows that $z_M$ appears in $d_{w(M)}+d_{w(M)+1}+\cdots +d_{r}$ linear terms of
     \[\prod_{m=1}^r \prod_{\mathbf{v}\in \lambda_{m+1}\backslash \lambda_m}\prod_{\substack{\mathbf{u}\in S_m^{T,\id}(\lambda_\bullet)\\ \mathbf{u}\neq \mathbf{v}}} (\mathbf{u}-\mathbf{v})(\underline{z}).\]
    These terms are exactly of the form 
    \[z_{M}-\mathbf{u}_i(\underline{z})\]
    for $d_1+\cdots d_{w(M)-1}+1\leq i\leq d-1$ and we see that $z_M$ is a leading term if and only if $i\leq M$. We note that the terms, where $z_M$ is not leading, are of the form 
    \[z_M-z_i\]
    for $i>M$. From 
     \[\prod_{i=1}^n \prod_{l=1}^{d-1}(s_i-z_l)\]
     we have $n$ linear terms containing $z_M$ and $z_M$ is the leading term in all of them. At last 
     \[\prod_{\substack{1\leq l<m\leq d-1 \\ w(l)<w(m)}}(z_m-z_l)\]
     has $d_1+\cdots +d_{w(M)-1}+d_{w(M)+1}+\cdots +d_{r}$ linear terms containing $z_{M}$. The terms where it is not leading are of the form 
     \[z_m-z_M\]
     for $m>M$. From this we see that
     \[\operatorname{deg}(q(\underline{z});\{M\})=d-1+n+d_{w(M)+1}+\cdots d_{r}\geq d-1+n.\]
     For $p(\underline{z})$ we again see that $z_M$ only appears in
    \[\prod_{1\leq m\neq l\leq d-1}(z_m-z_l)\]
    where it appears in $2(d-1)-2$ terms so that
     \[\operatorname{deg}(p(\underline{z});\{M\})+1=2(d-1)-1<d-1+n\leq\operatorname{deg}(q(\underline{z});\{M\}).\]
     We see that the terms $z_M-z_i$ and $z_m-z_M$ for $i,m>M$ where $z_M$ is not leading each cancels with a unique term of 
    \[\prod_{1\leq m\neq l\leq d-1}(z_m-z_l).\]
    Let
    \[\frac{p(\underline{z})}{q(\underline{z})}=\frac{p'(\underline{z})}{q'(\underline{z})}\]
    be the result of canceling those term. We now see that 
     \[\operatorname{deg}(q'(\underline{z});\{M\})=\operatorname{lead}(q'(\underline{z}),M)\]
     and since $q'$ and $p'$ is obtained from $q$ and $p$ by removing the same number of linear terms with $z_M$ it still holds that
     \[\operatorname{deg}(p'(\underline{z});\{M\})+1<\operatorname{deg}(q'(\underline{z});\{M\})\]
     so that the residue vanishes by \cref{General residue vanishing}.
\end{proof}
Writing out the terms in the residue formula explicitly for the Porteous point, we arrive at the following formula.
\begin{corollary} \label{iterated residue 1}
    Let $\theta_1,\dots, \theta_q\in A_T^1(\operatorname{pt})$ and let
    \[P=P(\theta_1,\dots, \theta_q,\eta_1,\dots, \eta_{d-1})\]
    be a bisymmetric polynomial as in \cref{tautological integral setup}. Then 
    \[\int_{[\operatorname{NHilb}^{\underline{d}}_{\operatorname{nil-fil}}(\A^n)]^{\operatorname{vir}}}P \]
    can be calculated via the iterated residue formula
    \[\underset{\underline{z}=\infty}{\operatorname{Res}}\frac{P(\theta_1,\dots,\theta_q,\underline{z})\prod_{1\leq m\neq l\leq d-1}(z_m-z_l)\prod_{\substack{1\leq i,j,k\leq d-1 \\ i<k}} \widetilde{\prod}_{\substack{0\leq m\leq d-1\\ w(k),w(j)\leq w(m)}}(z_i+z_j+z_k-z_m)d\underline{z}}{\prod_{\substack{1\leq l<m\leq d-1 \\ w(l)<w(m)}}(z_m-z_l)\prod_{i=1}^n \prod_{l=1}^{d-1}(s_i-z_l)\prod_{1\leq i\leq j\leq d-1} \widetilde{\prod}_{\substack{1\leq k\leq d-1\\ w(j)< w(k)}}(z_i+z_j-z_k)}. \]
\end{corollary}

\subsection{Removing the restriction on the number of points}
We recall that for all the machinery with the partial resolution to work we had to assume that $d_1+\cdots+d_r=d-1\leq n$. However, the iterated residue formula of \cref{iterated residue 1} makes sense for any $\underline{d},n$. In this section we will prove that this is actually the case. The key is to increase $n$ sufficiently and compare the two integrals
\begin{set}
    Let $n\leq N$ and set
    \[\A^n\hookrightarrow \A^n\times \A^{N-n}\simeq \A^N.\]
    In this section we set $T=(\C^*)^N$. We note that the above immersion is $T$-invariant and that there is a natural $T$-action on $\operatorname{naNHilb}^{\underline{d}}(\A^n)$ by composing with the projection $(\C^*)^N\to (\C^*)^n$.
\end{set}
\begin{lemma}
    There exist a $T$-equivariant extension of the immersion 
    \[\operatorname{NHilb}^{\underline{d}}(\A^n)\hookrightarrow \operatorname{NHilb}^{\underline{d}}(\A^N)\]
    to 
    \[\operatorname{naNHilb}^{\underline{d}}(\A^n)\to \operatorname{naNHilb}^{\underline{d}}(\A^N).\]
    The latter exhibits $\operatorname{naNHilb}^{\underline{d}}(\A^n)$ as the zero locus of a section of the bundle
    \[\mathcal{H}om(\mathcal{I},\mathcal{V}_0)\]
    where $\mathcal{V}_0$ is the universal flag bundle and $\mathcal{I}$ is the kernel of the map 
    \[(\oo^N)^\vee \twoheadrightarrow (\oo^n)^\vee\]
    given by projection onto the first $n$ summands. When $d_0=1$ the immersion restricts to an immersion
    \[\operatorname{naNHilb}_0^{\underline{d}}(\A^n)\hookrightarrow \operatorname{naNHilb}_0^{\underline{d}}(\A^n)\]
    which exhibits the domain as the zero locus of a section of 
    \[\mathcal{H}om(\mathcal{I},\mathcal{V}_1).\]
\end{lemma}
\begin{proof}
    We let $M_{n,\underline{d}}$ and $M_{N,\underline{d}}$ be as in \cref{definition of prequotient space} both defined using the same reference flag $N_\bullet$. We define 
    \begin{center}
        \begin{tikzcd}
            M_{n,\underline{d}} \arrow[hook]{r} & M_{N,\underline{d}} \\[-20pt]
            (\psi_1,\psi_2,v) \arrow[mapsto]{r} & (\psi_1\circ q, \psi_2 ,v)
        \end{tikzcd}
    \end{center}
    where $q:(\C^N)^\vee \twoheadrightarrow (\C^n)^\vee$ is projection onto the first $n$ summands. One easily checks that this is $T$- and $P_{\underline{d}}$-equivariant, thus given us the desired closed immersion. The description of a zero locus of a section of the above vector bundle also follows directly from the definition, as an $S$-valued point $(\mathcal{F}_\bullet, \psi_1,\psi_2,s)$ of $\operatorname{naNHilb}^{\underline{d}}(\A^N)$ factors through $\operatorname{naNHilb}^{\underline{d}}(\A^n)$ if and only if $\psi_1$ factors through $(\oo^N)^\vee \twoheadrightarrow (\oo^n)^\vee$. Note that on $\operatorname{naNHilb}_0^{\underline{d}}(\A^N)$ we have that $\psi_1$ factors through $\mathcal{V}_1$, so we can restrict to this bundle.
\end{proof}

\begin{lemma}
    Let
    $i:\operatorname{NHilb}^{\underline{d}}(\A^n)\hookrightarrow \operatorname{NHilb}^{\underline{d}}(\A^N)$
    be the inclusion, let $\mathcal{I}$ be the kernel of
    \[(\oo^N)^\vee \twoheadrightarrow (\oo^n)^\vee \]
    on $\operatorname{NHilb}^{\underline{d}}(\A^N)$ and set
    \begin{align*}\mathcal{E}_{n}=\mathcal{H}om(\mathcal{I},\mathcal{V}_0)|_{\operatorname{NHilb}^{\underline{d}}(\A^n)}, && \mathcal{E}_{n,0}=\mathcal{H}om(\mathcal{I},\mathcal{V}_1)|_{\operatorname{NHilb}_{\operatorname{nil-fil}}^{\underline{d}}(\A^n)}.\end{align*}
    We have 
    \[i_*[\operatorname{NHilb}^{\underline{d}}(\A^n)]^{\operatorname{vir}}=e^T(\mathcal{E}_n))\cap [\operatorname{NHilb}^{\underline{d}}(\A^N)]^{\operatorname{vir}}\]
    and
    \[i_*[\operatorname{NHilb}_{\operatorname{nil-fil}}^{\underline{d}}(\A^n)]^{\operatorname{vir}}=e^T(\mathcal{E}_{n,0})\cap [\operatorname{NHilb}_{\operatorname{nil-fil}}^{\underline{d}}(\A^N)]^{\operatorname{vir}}.\]
    In particular, 
    \[\int_{[\operatorname{NHilb}^{\underline{d}}(\A^n)]^{\operatorname{vir}}}\alpha=\int_{[\operatorname{NHilb}^{\underline{d}}(\A^N)]^{\operatorname{vir}}} e^T(\mathcal{E}_n)\cdot\alpha\]
    and
    \[\int_{[\operatorname{NHilb}_{\operatorname{nil-fil}}^{\underline{d}}(\A^n)]^{\operatorname{vir}}}\alpha=\int_{[\operatorname{NHilb}_{\operatorname{nil-fil}}^{\underline{d}}(\A^N)]^{\operatorname{vir}}} e^T(\mathcal{E}_{n,0})\cdot\alpha.\]
\end{lemma}
\begin{proof}
    Let $j$ be the extension of $i$ to the non-associative Hilbert scheme. Note that $j^*\mathcal{V}_\bullet$ is still the universal bundle on $\operatorname{naNHilb}^{\underline{d}}(\A^n)$. It follows that the associativity bundle also pullback to the associativity bundle. We can therefore form the following pullback squares.
    \begin{center}
        \begin{tikzcd}
            \operatorname{NHilb}^{\underline{d}}(\A^n) \arrow{d}{i} \arrow{r} & \operatorname{naNHilb}^{\underline{d}}(\A^n) \arrow{r}{j} \arrow{d}{j} & \operatorname{naNHilb}^{\underline{d}}(\A^N) \arrow{d}{0_n}\\
            \operatorname{NHilb}^{\underline{d}}(\A^N) \arrow{d} \arrow{r} & \operatorname{naNHilb}^{\underline{d}}(\A^N) \arrow{r}{s_n}\arrow{d}{s_{\operatorname{ass}}} & \mathcal{H}om(\mathcal{I},\mathcal{V}_\bullet) \\
            \operatorname{naNHilb}^{\underline{d}}(\A^N)  \arrow{r}{0_{\operatorname{ass}}} & \mathcal{E}_{\operatorname{ass}} & 
        \end{tikzcd}
    \end{center}
    and from this we get
\begin{align*}
    i_*[\operatorname{NHilb}^{\underline{d}}(\A^n)]^{\operatorname{vir}}&= i_* 0_{\operatorname{ass}}^![\operatorname{naNHilb}^{\underline{d}}(\A^n)] \\
    &=  0_{\operatorname{ass}}^!j_*[\operatorname{naNHilb}^{\underline{d}}(\A^n)]  \\
    &= 0^!_{ass}(e^T(\mathcal{H}om(\mathcal{I},\mathcal{V}_\bullet))\cap[\operatorname{naNHilb}^{\underline{d}}(\A^N)]) \\
    &=e^T(\mathcal{H}om(\mathcal{I},\mathcal{V}_\bullet)|_{\operatorname{NHilb}^{\underline{d}}(\A^N)})\cap0^!_{ass}[\operatorname{naNHilb}^{\underline{d}}(\A^N)] \\
    &=e^T(\mathcal{E}_n)\cap[\operatorname{NHilb}^{\underline{d}}(\A^N)]^{vir} 
\end{align*}
where we have used that $j$ is regular so that 
\[[\operatorname{naNHilb}^{\underline{d}}(\A^n)]=0_n^![\operatorname{naNHilb}^{\underline{d}}(\A^N)]\]
is a refined Euler class of $\mathcal{H}om(\mathcal{I},\mathcal{V}_0)$ i.e., $j_*[\operatorname{naNHilb}^{\underline{d}}(\A^n)]=e^T(\mathcal{H}om(\mathcal{I},\mathcal{V}_\bullet))\cap[\operatorname{naNHilb}^{\underline{d}}(\A^N)]$. The proof for the nilpotently filtered locus is exactly the same. 
\end{proof}

\begin{theorem}
  \label{iterated residue 2}
    Let $\theta_1,\dots, \theta_q\in A_T^1(\operatorname{pt})$ and let
    \[P=P(\theta_1,\dots, \theta_q,\eta_1,\dots, \eta_{d-1})\]
    be a bisymmetric polynomial as in \cref{tautological integral setup}. Then 
    \[\int_{[\operatorname{NHilb}^{\underline{d}}_{\operatorname{nil-fil}}(\A^n)]^{\operatorname{vir}}}P \]
    can be calculated via the iterated residue formula
    \[\underset{\underline{z}=\infty}{\operatorname{Res}}\frac{P(\theta_1,\dots,\theta_q,\underline{z})\prod_{1\leq m\neq l\leq d-1}(z_m-z_l)\prod_{\substack{1\leq i,j,k\leq d-1 \\ i<k}} \widetilde{\prod}_{\substack{0\leq m\leq d-1\\ w(k),w(j)\leq w(m)}}(z_i+z_j+z_k-z_m)d\underline{z}}{\prod_{\substack{1\leq l<m\leq d-1 \\ w(l)<w(m)}}(z_m-z_l)\prod_{i=1}^n \prod_{l=1}^{d-1}(s_i-z_l)\prod_{1\leq i\leq j\leq d-1} \widetilde{\prod}_{\substack{1\leq k\leq d-1\\ w(j)< w(k)}}(z_i+z_j-z_k)} \]
\end{theorem}
\begin{proof}
    If $n\geq d-1$ this is simply the content of \cref{iterated residue 1}. If not, we can find $N\geq  d-1$. Then 
        \[\int_{[\operatorname{NHilb}_{\operatorname{nil-fil}}^{\underline{d}}(\A^n)]^{\operatorname{vir}}}P=\int_{[\operatorname{NHilb}_{\operatorname{nil-fil}}^{\underline{d}}(\A^N)]^{\operatorname{vir}}} e^T(\mathcal{E}_{n,0})\cdot P.\]
        We note that
        \[e^T(\mathcal{E}_{n,0})=\prod_{i=n+1}^N\prod_{l=1}^{d-1}(s_i-\eta_l).\]
       If we apply \cref{iterated residue 1} to calculate $\int_{[\operatorname{NHilb}_{\operatorname{nil-fil}}^{\underline{d}}(\A^N)]^{\operatorname{vir}}} P'$ for $P'=e^T(\mathcal{E}_{n,0})\cdot\alpha$  we get
    \[\underset{\underline{z}=\infty}{\operatorname{Res}}\frac{P'(\theta_1,\dots,\theta_q,\underline{z})\prod_{1\leq m\neq l\leq d-1}(z_m-z_l)\prod_{\substack{1\leq i,j,k\leq d-1 \\ i<k}} \widetilde{\prod}_{\substack{0\leq m\leq d-1\\ w(k),w(j)\leq w(m)}}(z_i+z_j+z_k-z_m)d\underline{z}}{\prod_{\substack{1\leq l<m\leq d-1 \\ w(l)<w(m)}}(z_m-z_l)\prod_{i=1}^N \prod_{l=1}^{d-1}(s_i-z_l)\prod_{1\leq i\leq j\leq d-1} \widetilde{\prod}_{\substack{1\leq k\leq d-1\\ w(j)< w(k)}}(z_i+z_j-z_k)} \]
    but 
    \[P'(\theta_1,\dots,\theta_q,\underline{z})=P(\theta_1,\dots,\theta_q,\underline{z})\cdot \prod_{i=n+1}^N\prod_{l=1}^{d-1}(s_i-z_l)\]
    and we get the cancellation
    \[\frac{\prod_{i=n+1}^N\prod_{l=1}^{d-1}(s_i-z_l)}{\prod_{i=1}^N \prod_{l=1}^{d-1}(s_i-z_l)}=\frac{1}{\prod_{i=1}^n \prod_{l=1}^{d-1}(s_i-z_l)}\]
    giving us the desired formula.
\end{proof}

\bibliographystyle{abbrv}
\bibliography{ref}
\end{document}